 \def\useaos{1}

\ifdefined\useaos

\documentclass[aos,preprint]{imsart}

\arxiv{arXiv:1612.05612}

\usepackage{macros}
\usepackage{hyperref}
\usepackage[numbers]{natbib}
\usepackage{amsmath}
\usepackage{overpic}

\usepackage{xr}

\begin{document}
\begin{frontmatter}

  \title{Asymptotic Optimality in \\ Stochastic Optimization}
  \runtitle{Asymptotic Optimality in Stochastic Optimization}

  \begin{aug}
    \author{\fnms{John C.}
      \snm{Duchi}\thanksref{t2}\ead[label=e1]{jduchi@stanford.edu}} \and
    \author{\fnms{Feng}
      \snm{Ruan}\thanksref{t2}\ead[label=e2]{fengruan@stanford.edu}}

    \thankstext{t2}{JCD and FR partially supported by
      NSF Award 1553086 and the Toyota Research Institute. FR
      additionally supported by the E.K.\ Potter
      Stanford Graduate Fellowship.}

    \runauthor{Duchi and Ruan}

    \affiliation{Stanford University}

    \address{Department of Statistics\\
      Stanford University\\
      Sequoia Hall \\
      Stanford, California, 94305-4065 \\
      \printead{e1}\\
      \phantom{E-mail:\ }\printead*{e2} }
  \end{aug}

  \begin{abstract}
  We study local complexity measures for stochastic convex optimization
  problems, providing a local minimax theory analogous to that of
  H\'{a}jek and Le Cam for classical statistical problems.
  We give complementary optimality results,
  developing fully online methods
  that adaptively achieve optimal convergence guarantees. Our results provide
  function-specific lower bounds and convergence results that make precise
  a correspondence between statistical difficulty and the geometric notion
  of tilt-stability from optimization.  As
  part of this development, we show how variants of Nesterov's dual
  averaging---a stochastic gradient-based procedure---guarantee finite
  time identification of constraints in optimization problems, while
  stochastic gradient procedures fail.  Additionally, we
  highlight a gap between problems with linear and nonlinear
  constraints: standard stochastic-gradient-based procedures are
  suboptimal even for the simplest nonlinear constraints, necessitating
  the development of asymptotically optimal
  Riemannian stochastic gradient methods.
\end{abstract}

  \begin{keyword}[class=MSC]
    \kwd{62F10, 62F12, 68Q25}
  \end{keyword}

  \begin{keyword}
    \kwd{Local asymptotic minimax theory, convex analysis, stochastic gradients,
      manifold identification}
  \end{keyword}

\end{frontmatter}
\else  

\documentclass[11pt]{article}
\usepackage{macros}
\usepackage{macros}
\usepackage{hyperref}
\usepackage[numbers]{natbib}
\usepackage{amsmath}
\usepackage{overpic}
\usepackage[top=2.54cm,left=3cm,right=3cm,bottom=2.54cm]{geometry}

\title{Asymptotic Optimality in Stochastic Optimization\thanks{
    JCD and FR partially supported by
    NSF Award 1553086 and the Toyota Research Institute. FR
    additionally supported by the E.K.\ Potter
    Stanford Graduate Fellowship.
}}
\author{John C.\ Duchi and Feng Ruan \\ Stanford University}
\date{October 2018}

\fi


\ifdefined\useaos
\else

\maketitle
\begin{abstract}
  We study local complexity measures for stochastic convex optimization
  problems, providing a local minimax theory analogous to that of
  H\'{a}jek and Le Cam for classical statistical problems.
  We give complementary optimality results,
  developing fully online methods
  that adaptively achieve optimal convergence guarantees. Our results provide
  function-specific lower bounds and convergence results that make precise
  a correspondence between statistical difficulty and the geometric notion
  of tilt-stability from optimization.  As
  part of this development, we show how variants of Nesterov's dual
  averaging---a stochastic gradient-based procedure---guarantee finite
  time identification of constraints in optimization problems, while
  stochastic gradient procedures fail.  Additionally, we
  highlight a gap between problems with linear and nonlinear
  constraints: standard stochastic-gradient-based procedures are
  suboptimal even for the simplest nonlinear constraints, necessitating
  the development of asymptotically optimal
  Riemannian stochastic gradient methods.
\end{abstract}

\fi


\section{Introduction}

In this paper, we consider smooth stochastic convex optimization problems
of the form
\begin{equation}
  \begin{split}
    \minimize_x ~ f(x) & \defeq \E_{P}[\loss(x; \statrv)]
    = \int_{\statdomain} \loss(x; \statval) dP(\statval) \\
    \subjectto ~ x \in \xdomain & \defeq
    \{x \in \R^n : f_i(x) \le 0 ~ \mbox{for}~i = 1, \ldots, m\},
  \end{split}
  \label{eqn:problem}
\end{equation}
where each $f_i : \R^n \to \R$ is convex and smooth ($\mc{C}^2$), $\statrv
\sim P$ is a random variable, and for $\statval \in \statdomain$ the
function $\R^n \ni x \mapsto \loss(x; \statval)$ is convex and continuously
differentiable.  We study algorithms that attempt
to solve problem~\eqref{eqn:problem} using
a sample $\statrv_1, \ldots, \statrv_k \simiid P$.
In this setting, we investigate the
optimality properties of stochastic optimization procedures, providing 
both problem-specific lower bounds on the performance of any
method and giving optimal algorithms that adapt to problem structure.

Problems of the form~\eqref{eqn:problem} are of broad interest, as they
encompass a variety of problems in statistics, machine learning, and
optimization~\cite{HastieTiFr09}. Because of their wide applicability, it is
important to carefully understand the difficulty of such problems. This
includes understanding fundamental limits---how well the best algorithm
can behave on problem~\eqref{eqn:problem}---as well as
adaptivity, meaning the extent to which algorithms can adapt to the specific
problem at hand.  In this paper, we address these problems, showing
function-specific difficulty measures and developing a variant of Nesterov's
dual averaging algorithm~\cite{Nesterov09} that is (often) optimal, though
we demonstrate that alternative methods are necessary when the constraint
functions $f_i$ are nonlinear (and we provide one potential method).
Unifying our results is an understanding of the stability of solutions to
optimization problems under perturbations; we make precise connections
between Poliquin and Rockafellar's ``tilt stability''~\cite{PoliquinRo98}
and statistical and computational difficulty, giving an analogue of Fisher
information for stochastic optimization problems~\eqref{eqn:problem}.

A standard approach to providing
optimality guarantees is the minimax risk~\cite{NemirovskiYu83,
  VanDerVaart98, AgarwalBaRaWa12}. Here, one defines a class
$\mc{F}$ of functions of interest (such as Lipschitz convex functions) and
measures algorithmic performance by the worst-case behavior over this
function class. Minimax risk is an imprecise hammer: a function $f$
may belong to a number of classes of functions, and the risk may
differ substantially between these classes. The approach is
also often too conservative: if $f$ is decreasing
quickly near the boundary of $\xdomain$, it should be ``easier'' to
solve problem~\eqref{eqn:problem}. H\'{a}jek and Le Cam's local minimax
theory~\cite{VanDerVaart98, VanDerVaartWe96, LeCamYa00} addresses these
issues in classical statistical problems, giving \emph{problem-specific}
notions of difficulty and making rigorous the centrality of the Fisher
information.  In this paper,
we build on these results to answer the following: how hard is it to solve the
particular problem~\eqref{eqn:problem}?

The idea in this line of work (see also \citet{ZhuChDuLa16})
is to define a
shrinking neighborhood of problems, investigating
worst-case complexity in this neighborhood.  For stochastic optimization
problems~\eqref{eqn:problem}, the objective $(x, \statval) \mapsto \loss(x;
\statval)$ is generally known, while the probability distribution $P$ is
not; with that in mind, we study neighborhoods $\mc{P}_k(P)$
whose
elements are tilted variants $\wt{P}$ of the measure $P$ satisfying
\mbox{$d\wt{P}(s) \in [1 \pm c k^{-\half}] dP(s)$}, so that $\mc{P}_k(P)$
shrinks to $P$ as $k \to \infty$. Letting $\wt{x}$ denote the minimizer of
the objective~\eqref{eqn:problem} when $\wt{P}$ replaces $P$ and $L : \R^n
\to \R$ be a loss, we consider local minimax
complexity measures of the form
\begin{equation}
  \label{eqn:local-complexity}
  \inf_{\what{x}_k}
  \sup_{\wt{P} \in \mc{P}_k} \E_{\wt{P}}\left[ L( \what{x}_k(\statrv_1,
    \ldots, \statrv_k) - \wt{x}) \right],
\end{equation}
where the expectation is taken over $\statrv_i \simiid \wt{P}$.  To describe
our lower bound, we leverage the \emph{tilt-stability} of an optimization
problem~\cite{PoliquinRo98}, which describes the changes in solutions to
problem~\eqref{eqn:problem} when the tilt $f_v(x) \defeq f(x) - v^T x$
replaces $f(x)$. Letting $x_v$ denote the minimizer of $f_v(x)$ over
$\xdomain$, let us assume the objective~\eqref{eqn:problem} is smoothly tilt
stable, so $x_v = x\opt + D v + o(\norm{v})$ for some matrix $D$; we show
(Proposition~\ref{proposition:perturbation}) the precise dependence of $D$
on the problem~\eqref{eqn:problem} via the objective $\loss$, distribution
$P$, and constraints $\xdomain$.  Our first main result
(Theorem~\ref{theorem:local-minimax-lower}) provides a lower bound on local
complexity measures of the form~\eqref{eqn:local-complexity}. Here the
matrix $\Gamma \defeq D \cov(\nabla \loss(x\opt; \statrv)) D$ is analogous
to the classical inverse Fisher information~\cite{VanDerVaart98}, and
Theorem~\ref{theorem:local-minimax-lower} shows that $\E[L(Z_k)], Z_k \sim
\normal(0, k^{-1} \Gamma)$ is asymptotically a lower bound for the local
complexity~\eqref{eqn:local-complexity}.

The next question we address is whether our problem-dependent lower
bounds are accurate: are there procedures that achieve these
guarantees, and can we adapt to specific problem geometry?  The classical
sample average approximation (or empirical risk minimization)
approach~\cite{ShapiroDeRu09}, which sets $\what{x}_k = \argmin_{x \in
  \xdomain} \{\frac{1}{k} \sum_{i = 1}^k \loss(x; \statrv_i)\}$, is one
approach. As we discuss in the sequel, it is optimal and adaptive.
Given the scale of
many modern problems, however, it is important to develop
computationally efficient online procedures.
%
%
To that end, our second contribution
(Sections~\ref{sec:dual-averaging-convergence} and
\ref{sec:dual-averaging-normality}) is the development of
stochastic-gradient-based procedures that are (asymptotically) optimal,
achieving the infimum in the local
complexity~\eqref{eqn:local-complexity} for smooth enough functions
$\loss$.

We develop a variant of Nesterov's dual
averaging~\cite{Nesterov09}; we iterate
\begin{equation}
  \label{eqn:variant-of-dual-averaging}
  x_{k+1} \defeq \argmin_{x\in \xdomain} 
  \left\{\bigg(\sum_{i=1}^k \stepsize_i \nabla \loss(x_i; \statrv_i)\bigg)^T x +
    \half \ltwo{x-x_0}^2\right\},
\end{equation} 
where $\stepsize_i$ denotes a stepsize sequence.  In the case that $\xdomain
= \R^n$, this method reduces to the stochastic gradient method,
and \citet{PolyakJu92} show that the averages $\wb{x}_k = \frac{1}{k}
\sum_{i = 1}^k x_i$ are asymptotically normal with
the optimal covariance we derive. In contrast, we show that
(i) the
iteration~\eqref{eqn:variant-of-dual-averaging} converges a.s.\ and
identifies the active constraints in problem~\eqref{eqn:problem} in finite
time, and (ii) as long as the constraints $f_i$ are linear, dual averaging
is optimal and adaptive
(Theorems~\ref{theorem:as-convergence}--\ref{theorem:asymptotic-normality}).
Stochastic projected gradient descent methods \emph{do not} enjoy these
guarantees.  An intriguing gap arises when the constraints are nonlinear:
our proposed algorithm and classical dual
averaging~\cite{Nesterov09} cannot be optimal with nonlinear
constraints, even for $\xdomain = \{x \in \R^n : \ltwo{x}^2 \le 1\}$.  To
address this, we develop an asymptotically optimal manifold-based
online algorithm (Theorem~\ref{theorem:asymptotic-normality-rsgd}), showing
that closing this gap is possible but nontrivial.

The unifying aspect of both threads---algorithms and lower
bounds---throughout this work is the geometry of the
problem~\eqref{eqn:problem}. Letting $x\opt$ denote the minimizer of the
problem (our coming assumptions make this unique), we give a perturbation
analysis~\cite{BonnansSh98} of parameterized versions of
problem~\eqref{eqn:problem} that shows how the active constraints $\{i :
f_i(x\opt) = 0\}$ affect solutions to the perturbed
problems~\eqref{eqn:local-complexity}. A similar perturbation analysis is
also central to our results on optimal constraint identification and the
asymptotic covariance structure of the iterates of dual
averaging~\eqref{eqn:variant-of-dual-averaging}, providing a unifying
geometric theme to our results and allowing us to provide computational and
optimization-based analogues of the Fisher information.

\subsection{Related Work}
That problem geometry strongly influences optimization algorithms is
well-known. In statistics, geometric conditions involving the
continuity of the estimand with respect to the underlying probability
measure are central to minimax
analyses~\cite{Birge83,DonohoLi87,DonohoLi91a}, and the Fisher information
characterizes classical asymptotics~\cite{LeCamYa00, VanDerVaart98,
  VanDerVaartWe96}.  Our approach to local asymptotic minimax lower bounds
builds out of the literature on semi- and non-parametric
efficiency~\cite{Stein56a, IbragimovHa81, BickelKlRiWe98, VanDerVaart98},
where one wishes to estimate a finite-dimensional parameter of an
infinite-dimensional nuisance, thus studying hardest finite-dimensional
subproblems; we connect these hardest subproblems to stability in
optimization.  In deterministic optimization, work by \citet{BurkeMo94} and
\citet{Wright93} shows how projected gradient and Newton methods identify
active constraints and converge quickly once
identified, and such identification underlies active set
methods~\cite{NocedalWr06}.

On the algorithmic side, there is a substantial literature on stochastic
approximation and optimization procedures, with growing recent importance
for large-sample problems~\cite{RobbinsMo51, Kushner74, PolyakJu92,
  Zinkevich03, BottouBo07, NemirovskiJuLaSh09, Xiao10, DuchiHaSi11}. Early
works, beginning with \citet{RobbinsMo51} and continuing through work by
(among others) \citet{Ermoliev69, Ermoliev83}, \citet{Venter67},
\citet{Fabian73}, \citet{Kushner74}, and \citet{Walk83}, develop probability
one convergence with and without constraints, as well as asymptotic
normality results in restricted situations~\cite{Venter67,
  Fabian73}. \citet{PolyakJu92} show the importance of averaging
stochastic gradient methods with ``long stepping,'' establishing a generic
asymptotic normality result. Our results are a
natural descendant of this work, but they require new development, and given
the subtleties that nonlinear constraints introduce for asymptotics, we
require extensions to and connections with Riemannian
methods~\cite{AbsilMaSe09, Boumal14, TripuraneniStJiReJo17}.  Recent
progress on incremental gradient methods---which approximate the population
expectation~\eqref{eqn:problem} by an empirical average---develops efficient
estimators using limited computation~\cite{LeRouxScBa12, JohnsonZh13,
  DefazioBaLa14, LinMaHa15}, though the methods do not apply in fully online
stochastic scenarios.

\subsection{Notation and basic definitions}
We let $\R_+ = \{x \in \R : x \ge 0\}$ and $\R_{++} = \{x \in \R : x > 0\}$.
For any $m \in \N$, we use $[m]$ to denote the set of integers $\{1, 2,
\ldots, m\}$.  For a set $C$, we use $\relint(C)$ to denote its relative
interior \cite[Section 6]{Rockafellar70} and $\setindic{x}{C}$ to
denote the extended real valued function
\begin{equation*}
\setindic{x}{C} = 
\begin{cases}
	0 & x\in C \\
	+\infty & x \not\in C
\end{cases}
\end{equation*}
For a vector $v$, $\norm{v}$ denotes its Euclidean norm. 
For a matrix $A$, $A^{\dag}$ is its Moore-Penrose inverse, and
$\matrixnorm{A} = \sup_{\norm{v} = 1} \norm{Av}$ is its $l_2$ operator norm.


\section{Background and assumptions}
\label{sec:assumptions}

Before moving to our main results, we collect important assumptions,
definitions, and recapitulate a few results on stochastic optimization.  As
we view our results through the lens of stability and perturbation, we also
present a perturbation result on tilt-stability of optimization problems
that underpins our development.

\subsection{Main assumptions}

We begin by formalizing the problems we consider. 
This involves specifying smoothness and identifiability 
properties on $f$ and $x\opt$, the unique minimizer 
of problem~\eqref{eqn:problem} (our assumptions 
ensure uniqueness).
\begin{assumption}
  \label{assumption:smoothness-of-objective}
  There exists
  $L < \infty$ such that
  \begin{equation*}
    \norm{\nabla f(x) - \nabla f(x\opt)} \le L \norm{x - x\opt}
    ~~ \mbox{for~all}~ x \in \mc{X}.
  \end{equation*}
  There exist
  $C, \epsilon \in (0, \infty)$ such that for
  $x \in \xdomain \cap \{x: \norm{x - x\opt} \leq \epsilon\}$,
  \begin{equation*}
    \norm{\nabla f(x) - \nabla f(x\opt)
    - \nabla^2 f(x\opt)(x - x\opt)} \le C \norm{x - x\opt}^{2}.
  \end{equation*}
\end{assumption}

\noindent
Because we study perturbation of solutions and rates of convergence,
we require constraint qualifications to make precise
guarantees. The normal cone
to the set $\mc{X}$ at the point $x$ is
\begin{equation*}
  \normalcone_{\mc{X}}(x) \defeq
  \left\{v \in \R^n : \<v, y - x\> \le 0 ~ \mbox{for~all~} y \in \mc{X} \right\}.
\end{equation*}
The optimality conditions for convex
programming~\cite{BoydVa04,HiriartUrrutyLe93ab} for
problem~\eqref{eqn:problem} are that $x\opt$ minimizes $f$ over $\mc{X}$ if
and only if $-\nabla f(x\opt) \in \normalcone_{\mc{X}}(x\opt)$.  The
condition that $-\nabla f(x\opt) \in \normalcone_{\xdomain}(x\opt)$ is
insufficient for our identification and perturbation results, so we make a
standard constraint
qualification~\cite[Def.~2.4]{Wright93,BurkeMo94,HareLe04}.  Throughout, we
let $\numactive$ be the number of \emph{active} constraints in
problem~\eqref{eqn:problem}, that is, the number of all indices $i$ such that
$f_i(x\opt) = 0$. Without loss of generality, we assume $f_1, \ldots,
f_{\numactive}$ are the only active constraints.
\begin{assumption}
  \label{assumption:linear-independence-cq}
  The vector $\nabla f(x\opt)$ satisfies
  \begin{equation*}
    -\nabla f(x\opt) \in \relint \normalcone_{\mc{X}}(x\opt).
  \end{equation*}
  The constraint functions $\{f_1, \ldots, f_m\}$ are $\mc{C}^2$ 
  near $x\opt$. Additionally, the active constraints
  $\{f_1, \ldots, f_\numactive\}$ satisfy either
  \begin{enumerate}[i.]
  \item \label{item:linear-independence-cq}
    The set $\{\nabla f_i(x\opt)\}_{i = 1}^\numactive$ is linearly independent
  \item \label{item:affine-cq} The functions $f_i$ are affine.
  \end{enumerate}
\end{assumption}
\noindent
Assumption~\ref{assumption:linear-independence-cq} implies
there exists a strictly positive
$\lambda\opt \in \R_{++}^\numactive$ such that
\begin{equation}
  \label{eqn:strict-complementary}
  \nabla f(x\opt) + \sum_{i=1}^\numactive \lambda_i\opt \nabla f_i(x\opt) = 0,
\end{equation}
and $\lambda\opt$ is unique under
Assumption~\ref{assumption:linear-independence-cq}.i.  This follows by
standard constraint qualifications~\cite[Chapter VII.2]{HiriartUrrutyLe93ab}
for linear or independent constraints, which implies that
$\normalcone_{\mc{X}}(x\opt) = \{\sum_{i = 1}^\numactive \lambda_i \nabla
f_i(x\opt), \lambda \in \R^{\numactive}_+\}$, whose relative interior is the
set with $\lambda$ strictly positive. The set of $\lambda \in
\R^\numactive_+$ satisfying the KKT condition $\nabla f(x\opt) + \sum_i
\lambda_i \nabla f_i(x\opt) = 0$ is a compact convex polyhedron.

We require two additional assumptions on the structure of the function
$f$. We define the critical tangent set to $\mc{X}$ at $x$ by
\begin{equation}
  \label{eqn:critical-tangent}
  \tangentset_{\mc{X}}(x)
  \defeq \left\{w \in \R^n :
  \nabla f_i(x)^T w = 0,
  ~~ \mbox{for~} i \in [m] ~\mbox{s.t.}~ f_i(x) = 0 \right\}.
\end{equation}
With this definition, we make the following standard second-order sufficiency,
or restricted strong convexity,
assumption~\cite{Shapiro89,Wright93,DontchevRo14}.
\begin{assumption}
  \label{assumption:restricted-strong-convexity}
  There exists $\mu > 0$ such that for any
  $w \in \tangentset_{\mc{X}}(x\opt)$,
  \begin{equation*}
    w^T \left[\nabla^2 f(x\opt)+ \sum_{i=1}^\numactive \lambda_i\opt \nabla^2 f_i(x\opt)\right]
    w \geq \mu \norm{w}^2. 
  \end{equation*}
\end{assumption}
\noindent
Assumption~\ref{assumption:restricted-strong-convexity} guarantees the
uniqueness of minimizers of the function $f$ over
$\mc{X}$; more, it implies
$f$ has the following growth properties.
\begin{lemma}[Wright~\cite{Wright93}, Theorem 3.2(i)]
  \label{lemma:restricted-strong-convexity}
  Under Assumption~\ref{assumption:restricted-strong-convexity},
  there exists $\epsilon > 0$ such that
  \begin{equation*}
    \<\nabla f(x), x - x\opt\>
    \ge f(x) - f(x\opt) \geq \epsilon
    \min\left\{\norm{x - x\opt}^2, \norm{x - x\opt} \right\}
    ~~ \mbox{for~} x \in \xdomain.
  \end{equation*}
\end{lemma}


Finally, we make the standard assumption~\cite{RobbinsSi71, PolyakJu92,
  NemirovskiJuLaSh09} that the noise in the functions
$\loss$ is not too substantial.
\begin{assumption}
  \label{assumption:noise-level}
  There exists $C < \infty$ such for all $x \in \mc{X}$,
  \begin{equation*}
    \E[\norm{\nabla \loss(x; \statrv) - \nabla \loss(x\opt; \statrv)}^2]
    \le C \norm{x - x\opt}^2.
  \end{equation*}
  The gradients $\nabla \loss(x\opt; \statrv)$ have finite
  covariance $\covmat \defeq \cov(\nabla \loss(x\opt; \statrv))$.
\end{assumption}
\noindent
We provide two remarks on Assumption~\ref{assumption:noise-level}.  First,
Assumptions~\ref{assumption:smoothness-of-objective}
and~\ref{assumption:noise-level}, coupled with Jensen's inequality, imply
that for any $x \in \xdomain$ we have
\begin{align}
  \E[\norm{\nabla \loss(x; \statrv) - \nabla f(x)}^2]
  \le \E[\norm{\nabla \loss(x; \statrv)}^2]
  & \le C \left(1 + \norm{x - x\opt}^2\right),
    \label{eqn:noise-level-gradient-norms}
\end{align}
where $C < \infty$ is some constant.
Second, many statistical applications
and stochastic programming problems, including
linear and logistic regression, satisfy
Assumption~\ref{assumption:noise-level}. Verifying the assumptions
for these is routine~\cite{PolyakJu92}.


\subsection{Perturbation of optimal solutions and classical asymptotics}

The unifying thread throughout this work is the importance of perturbation
results for optimal solutions of optimization problems, which form the
building blocks of classical asymptotic results for
problem~\eqref{eqn:problem} (cf.~\citet{Shapiro89}), for the local minimax
lower bounds we develop, and for the identification and optimality results
we provide for stochastic gradient-based algorithms.

With this in mind, we consider tilt-stability properties of
solutions to problem~\eqref{eqn:problem}. Tilt stability is the Lipschitz
continuity of minimizers of \emph{tilted} versions of an objective $f$,
\emph{viz.} minimizers of $f_v(x) \defeq f(x) - \<v, x\>$ for $v$ near $0$;
the notion has been influential in variational analysis and the development
of optimization algorithms for some time~\cite{PoliquinRo98, DontchevRo14,
  DrusvyatskiyLe13}.  In our case, we can provide an implicit function
theorem for the KKT system associated with the optimality conditions for
problem~\eqref{eqn:problem} under tilt-like perturbations of the objective.
To make this concrete, let $v \in \R^n$ be a perturbation vector, and
assuming that $f_v$ is still convex, we consider approximate tilts of
$f$ satisfying
\begin{equation}
\label{eqn:taylor-f-v}
  f_v(x) = f(x) - v^T x + c_v + o(\norm{v}^2 + \norm{x - x_0}^2)
\end{equation}
for $v$ near 0 and $x$ near $x_0$, where $x_0$ minimizes
$f_0(x)$ over $\mc{X}$ (i.e.\ $x_0 = x\opt$) and $c_v$ depends only on $v$.
We then consider the tilted problem
\begin{equation}
  \label{eqn:tilted-problem}
  \minimize ~ f_v(x) ~~ \subjectto f_i(x) \le 0, i = 1, \ldots, m,
\end{equation}
whose minimizer we denote by $x_v$.
By assumption, the
problem~\eqref{eqn:tilted-problem}
is convex,
so we equivalently assume that
$\nabla_x f_v(x) = \nabla f_0(x) - v + o(\norm{v} + \norm{x - x_0})$.
Let $\mc{L}(x, \lambda) = f(x) + \sum_{i = 1}^m \lambda_i f_i(x)$ denote
the Lagrangian for problem~\eqref{eqn:problem}, and define the Hessian of
the problem at optimality by
\begin{equation*}
  H\opt \defeq \nabla^2_x \mc{L}(x\opt, \lambda\opt)
  = \nabla^2 f(x\opt) + \sum_{i = 1}^\numactive \lambda_i\opt
  \nabla^2 f_i(x\opt).
\end{equation*}
Let $\projmat[\tangentset]$ denote the orthogonal projection onto the
tangent set~\eqref{eqn:critical-tangent} at $x\opt$, which we recall is
$\tangentset_{\mc{X}}(x\opt) = \cap_{i = 1}^\numactive\{w : w^T \nabla
f_i(x\opt) = 0\}$.  That is, if $A \in \R^{\numactive \times n}$ denotes the
matrix with rows $\nabla f_i(x\opt)^T$, then $\projmat[\tangentset] = I -
A^T(AA^T)^\dag A$. We then have the following perturbation result, an
implicit function theorem for the KKT system generated by
problem~\eqref{eqn:tilted-problem}.
\begin{proposition}
  \label{proposition:perturbation}
  Let Assumptions~\ref{assumption:smoothness-of-objective}, 
  \ref{assumption:linear-independence-cq}
  and~\ref{assumption:restricted-strong-convexity} hold. Assume 
  that for any $v\in \R^n$, the function $f_v(x)$ is convex, and 
  satisfies the Taylor expansion at Eq~\eqref{eqn:taylor-f-v}. Then 
  the minimizer $x_v$ of 
  Eq~\eqref{eqn:tilted-problem} satisfies
  \begin{equation*}
    x_v = x_0 + \projmat[\tangentset] {H\opt}^\dag \projmat[\tangentset]
    v + o(\norm{v}).
  \end{equation*}
\end{proposition}
\noindent
Though Proposition~\ref{proposition:perturbation} is essentially known,
because of its centrality in our development, we provide a proof
based on~\cite[Theorem
  2G.8]{DontchevRo14} in Section~\ref{sec:proof-perturbation}.

\subsection{The classical M-estimator}

Proposition~\ref{proposition:perturbation} underlies both achievability
results for stochastic convex optimization~\cite{Shapiro89} and, as we show
in the sequel, local asymptotic
minimax results. To illustrate, we give a heuristic
sketch to show how Proposition~\ref{proposition:perturbation} yields
asymptotic normality of standard M-estimators for
problem~\eqref{eqn:problem}.  Given a sample $\statrv_1, \ldots, \statrv_k$,
define
\begin{equation}
  \what{x}_k \in \argmin_{x \in \mc{X}}
  \bigg\{\what{f}_k(x) \defeq \frac{1}{k} \sum_{i = 1}^k \loss(x; \statrv_i)
  \bigg\}.
  \label{eqn:m-estimator}
\end{equation}
Taylor's theorem implies there are matrices $\what{E}_k(x)$ and
$E(x)$, both $o(1)$ as $x \to x\opt$ (we assume
heuristically this is uniform in $k$), such that
\begin{align*}
  \nabla \what{f}_k(x) & = \nabla \what{f}_k(x\opt)
  + (\nabla^2 \what{f}_k(x\opt) + \what{E}_k(x))(x - x\opt) ~~ \mbox{and} ~~
  \\ 
  \nabla f(x) & = \nabla f(x\opt) + (\nabla^2 f(x\opt) + E(x))(x - x\opt).
\end{align*}
Then, defining $\what{v}_k = \nabla f(x\opt) - \nabla \what{f}_k(x\opt)$,
we have that
\begin{align*}
  \nabla \what{f}_k(x)
  & = \nabla f(x) - \what{v}_k
  + \left(\nabla^2 \what{f}_k(x\opt) - \nabla^2 f(x\opt)
  + \what{E}_k(x) - E(x)\right)(x - x\opt) \\
  & = \nabla f(x) - \what{v}_k + \left(o_p(1) + o(1)\right) \cdot (x - x\opt),
\end{align*}
where $o(1) \to 0$ as $x \to x\opt$,
and the expansion~\eqref{eqn:taylor-f-v} holds.
Applying
Proposition~\ref{proposition:perturbation}
yields that $\what{x}_k$ satisfies
$\what{x}_k - x\opt = \projmat[\tangentset]{H\opt}^\dag \projmat[\tangentset]
\what{v}_k + o_p(\norms{\what{v}_k})$, and
finally noting that $\sqrt{k} \cdot \what{v}_k
\cd \normal(0, \Sigma)$ gives the following corollary.
\begin{corollary}[Shapiro~\cite{Shapiro89}, Theorem 3.3]
  \label{corollary:shapiro-normality}
  Let
  Assumptions~\ref{assumption:smoothness-of-objective}--\ref{assumption:noise-level}
  hold and $\wt{x}_k \in \mc{X}$ satisfy
  $\what{f}_k(\wt{x}_k) - \inf_{x \in \mc{X}} \what{f}_k(x) =
  o_P(1/k)$. Then
  \begin{equation}
    \label{eqn:shapiro-normality}
    \sqrt{k}(\wt{x}_k - x\opt)
    \cd \normal\left(0, \projmat[\tangentset] {H\opt}^\dag
      \projmat[\tangentset] \Sigma \projmat[\tangentset]
      {H\opt}^\dag \projmat[\tangentset]\right)~\text{as $k \to \infty$}.
  \end{equation}
\end{corollary}
\noindent
This result shows the M-estimator $\what{x}_k$ is asymptotically normal with
the active constraints restricting (and improving) the covariance.

Corollary~\ref{corollary:shapiro-normality} leads to two questions. First,
is the result improvable?  In
Section~\ref{sec:optimality}, we show that in a local minimax sense, the
result is indeed optimal, so that it is essentially unimprovable.  Second,
the M-estimator~\eqref{eqn:m-estimator} is not really a procedure, as it may
require non-trivial computation.  Because
Corollary~\ref{corollary:shapiro-normality} allows estimators that are
$o(1/k)$ accurate, recent efficient methods for minimization of
finite sums using careful variance reduction and sampling
techniques~\cite{LeRouxScBa12, JohnsonZh13, DefazioBaLa14, LinMaHa15}
achieve the asymptotic normality~\eqref{eqn:shapiro-normality},
given a sample of size $k$, while computing $O(k \log k)$
gradients $\nabla \loss(x; \statrv)$ in total (the methods require storing
the entire dataset $\statrv_1, \ldots, \statrv_k$ and iterating through it
multiple times). It is, however, not immediate that the
rates~\eqref{eqn:shapiro-normality} are achievable using online or purely
stochastic gradient methods that compute a
\emph{single} stochastic gradient for each observation $\statrv_i$.
In
Section~\ref{sec:dual-averaging-normality}, we show this
is possible, developing asymptotically optimal online procedures.



\section{Optimality guarantees}
\label{sec:optimality}

With the asymptotic normality guarantee of
Corollary~\ref{corollary:shapiro-normality}, it is of interest to understand
the best possible (statistical) behavior for optimization procedures.
As we discuss in the introduction, standard
minimax complexity guarantees~\cite{NemirovskiYu83, AgarwalBaRaWa12}
are too imprecise: they fail to provide guidance on
specific to the problem at hand.
With this in mind, we consider a local
asymptotic minimax variant of problem~\eqref{eqn:problem}.  It
is natural to assume that the loss $\loss(x; \statval)$ is specified---we have a
way to measure performance of the decision vector $x$---but the distribution
$P$ may be unknown or is a nuisance parameter (we simply wish to find the
minimizing $x$).

We thus consider the
difficulty of solving problem~\eqref{eqn:problem} over small
neighborhoods of $P$.  To define these neighborhoods, for $d \in \N$, we
parameterize $P$ via a vector $u \in \R^d$ (where the
original problem corresponds to $u = 0$ and $P_0$), denoting the objective
of problem~\eqref{eqn:problem} by $f_0(x) = \E_{P_0}[\loss(x; \statrv)]$ and
its (unique) optimum by $x_0$. The perturbed
distributions $P_u$
dovetail with our results on stability of minimizers under tilt-perturbation
(Proposition~\ref{proposition:perturbation}): in appropriate cases, we show
that $f_u(x) = \E_{P_u}[\loss(x; \statrv)] \approx f_0(x) - u^T \Sigma (x -
x_0)$, where $\Sigma = \cov(\nabla \loss(x_0; \statrv))$.
Our results elucidate the precise correspondence
between tilt-stability and difficulty of stochastic optimization.

\subsection{Tilted distributions}
To define the perturbed problems, let $\nearlinear : \R \to [-1, 1]$ be any
three-times continuously differentiable function, where
\begin{equation}
  \label{eqn:near-linear}
  \nearlinear(t) = t ~~ \mbox{for~} t \in \left[-1/2, 1/2\right],
\end{equation}
the derivative $\nearlinear' \ge 0$ is nonnegative, and the first three
derivatives of $\nearlinear$ are bounded.  (The choice $[-1/2, 1/2]$ is
immaterial; any interval containing 0 on which $\nearlinear(t) = t$
suffices.) Now, let
\begin{equation*}
  \pertfuncset \defeq
  \left\{\pertfunc : \statdomain \to \R^d \mid
  \E_{P_0}[\pertfunc(\statrv)] = 0,
  \E_{P_0}[\norm{\pertfunc(\statrv)}^2] < \infty \right\}
\end{equation*}
(the maximal tangent set to the set of
distributions on $\statdomain$ at $P_0$,
cf.~\cite[Ch.~25]{VanDerVaart98}).
Then for $\pertfunc \in \pertfuncset$ and $u \in
\R^d$ we consider the tilted distribution
\begin{equation}
  \label{eqn:tilterriffic}
  dP_u(\statval)
  = \frac{1 + \nearlinear(u^T g(\statval))}{C_u}
  dP_0(\statval)
  ~~ \mbox{where} ~~
  C_u = 1 + \int \nearlinear(u^T g(\statval)) dP_0(\statval).
\end{equation}
This distribution approximates $dP_u(\statval) \propto e^{u^T
  \pertfunc(\statval)} dP_0(\statval)$ as $u \to 0$, providing a slight
reweighting in directions $\pertfunc$ specifies.  Such tilted constructions
are central to proving lower bounds for semi-parametric inference problems
(e.g.~\cite[Example 25.16]{Stein56a, IbragimovHa81, VanDerVaart98}) where
the goal is to infer a finite-dimensional parameter of a distribution
$P_0$. The lower bound and essential geometric difficulty arise by embedding
hardest one-dimensional sub-problems into the broader problem. In this
context, we identify the correct score (or influence)
function~\cite{IbragimovHa81, VanDerVaart98} for constrained stochastic
optimization.

Thus, for
$u \in \R^d$, we consider convex programs $\program_u$ defined by
\begin{equation}
  \begin{split}
    \minimize_x &~~ f_u(x) \defeq \E_{P_u}[\loss(x; \statrv)]
    = \int \loss(x; \statval) dP_u(\statval) \\
    \subjectto &~~ f_i(x) \leq 0 ~~\mbox{for~} i = 1, \ldots, m,
  \end{split}
  \label{eqn:stochastic-u-opt-problem}
\end{equation}
letting $x_u$ denote the minimizer of the tilted convex
program~\eqref{eqn:stochastic-u-opt-problem}.  We develop a local asymptotic 
minimax theory as $u$ varies in neighborhoods of 
zero of radius $\propto 1 / \sqrt{k}$, where $k$ denotes the sample size.


\newcommand{\rem}{{\rm Rem}}

We require one additional assumption to show our lower bounds.
\begin{assumption}
  \label{assumption:regularity-gradients}
  For $P_0$-almost all $\statval$, the function $\loss(\cdot; \statval)$ is
  $\mc{C}^2$ in a neighborhood of $x_0$. There are a remainder $\rem :
  \xdomain \times \statdomain \to \R^{n \times n}$ and $M :
  \statdomain \to \R_+$ satisfying
  \begin{equation*}
    \nabla^2 \loss(x;\statval) = \nabla^2 \loss(x_0; \statval) + 
    \rem(x; \statval)
  \end{equation*}
  where for some $\delta > 0$,
  \begin{equation*}
    \sup_{\norm{x - x_0} \le \delta}
    \matrixnorm{\rem(x; \statval)} \le M(\statval)
    ~~~\mbox{and} ~~~
    \E_{P_0}[M(\statrv)] < \infty.
  \end{equation*}
  Additionally, we have the following integrability conditions:
  \begin{equation*}
    \E_{P_0}\left[M(\statrv) \norm{\nabla \loss(x_0; \statrv)}\right] < \infty,
    ~~~
    \E_{P_0} \left[\norm{\nabla \loss(x_0; \statrv)}
      \matrixnorm{\nabla^2 \loss(x_0; \statrv)}\right] < \infty,
  \end{equation*}
  and for some $\delta > 0$
  \begin{equation*}
    \sup_{\norm{x - x_0} \le \delta}
    \E_{P_0} [|\loss(x; \statrv)| \norm{\nabla \loss(x_0; \statrv)}^2] < \infty.
  \end{equation*}
\end{assumption}
\noindent
Note that $\rem(x; \statval) \to 0$ as $x \to x_0$ by assumption that
$\loss(\cdot; \statval)$ is $\mc{C}^2$.

\subsection{A local asymptotic minimax theorem}
With this assumption, we have the following theorem, which provides a local
asymptotic minimax lower bound on optimization.  In the
theorem, we use the notation of
Proposition~\ref{proposition:perturbation}, where $\projmat[\tangentset] \in
\R^{n \times n}$ denotes the orthogonal projection onto the tangent
space $\tangentset$ and $H\opt = \nabla^2 f(x\opt) + \sum_{i =
  1}^\numactive \lambda_i\opt \nabla^2 f_i(x\opt)$.  We also recall that $L :
\R^n \to \R$ is quasiconvex if for all $\alpha \in \R$ the sub-level sets
$\{x \in \R^n : L(x) \le \alpha\}$ are convex, and let $\E_{P_u^k}$ denote
expectation under $k$ i.i.d.\ observations $\statrv_i \sim P_u$.
\begin{theorem}
  \label{theorem:local-minimax-lower}
  Let
  Assumptions~\ref{assumption:smoothness-of-objective}--\ref{assumption:regularity-gradients}
  hold and let $L : \R^n \to \R$ be a symmetric quasi-convex loss.  For any
  sequence of estimators $\what{x}_k : \statdomain^k
  \to \R^n$,
  \begin{equation}
    \label{eqn:local-minimax-lower}
    \sup_{d \in \N, g \in \pertfuncset}
    \lim_{c \to \infty}
    \liminf_{k \to \infty}
    \sup_{\ltwo{u} \leq c/\sqrt{k}}
    \E_{P_u^k}\left[L(\sqrt{k}(\what{x}_k - x_{u}))\right]
    \ge \E[L(Z)],
  \end{equation}
  where
  \begin{equation*}
    Z \sim \normal\left(0, \projmat[\tangentset]
      {H\opt}^\dag \projmat[\tangentset] \cov(\nabla \loss(x_0; \statrv))
      \projmat[\tangentset] {H\opt}^\dag \projmat[\tangentset]
    \right).
  \end{equation*}
  Moreover, $g(\statval) = \nabla \loss(x_0; \statval) - \E_{P_0}[\nabla
    \loss(x_0; \statrv)]$ achieves the
  supremum~\eqref{eqn:local-minimax-lower}.
\end{theorem}

\paragraph{Remarks}
We provide the proof of Theorem~\ref{theorem:local-minimax-lower} in
Section~\ref{sec:proof-local-minimax-lower}, discussing it
here.  It is important that the
limit in $k$ is taken before that in $c$, as this
provides the local nature of the result: the neighborhoods of problems, as
given by the tilted distributions $P_u$ in
Eq.~\eqref{eqn:tilterriffic}, have size decreasing as
$O(1 / \sqrt{k})$. The rescaling of the estimator error $\what{x}_k - x_u$ by
$\sqrt{k}$ reflects our expectation that
$\sqrt{k}(\what{x}_k - x_u)$ is $O(1)$ for good
estimators $\what{x}_k$ by Corollary~\ref{corollary:shapiro-normality}.

We may consider alternative choices of the
neighborhood of $P_0$. One is to use
$\phi$-divergences~\cite{AliSi66,Csiszar67}, where for $\phi$ convex with
$\phi(1) = 0$, one defines
\begin{equation*}
  \dphis{P}{Q} \defeq \int \phi\left(\frac{dP}{dQ}\right) dQ \ge
  0.
\end{equation*}
For example, KL-divergence has $\phi(t) = t \log t - t + 1$, the
$\chi^2$-divergence uses $\phi(t) = \half (t - 1)^2$, and the squared
Hellinger distance corresponds to $\phi(t) = \half (\sqrt{t} - 1)^2$. It is
no loss of generality to assume that $\phi'(1) = 0$ in the definition of
$D_\phi$, as $\phi_*(t) = \phi(t) - t \phi'(1) + \phi'(1)$ satisfies $D_\phi
= D_{\phi_*}$.  Consider now any $\phi$-divergence with $\phi$ a $\mc{C}^2$
function in a neighborhood of $1$ and $\phi''(1) > 0$.
Then Lebesgue's dominated convergence theorem implies (see
Section~\ref{sec:proof-local-minimax-lower}) that the normalization $C_u = 1
+ o(\norm{u}^2)$, and so
\begin{align*}
  \dphi{P_u}{P_0}
  & = \int \phi\left(\frac{1 + \nearlinear(u^T \pertfunc(\statval))}{C_u}\right)
    dP_0(\statval) \\
  & = \half \phi''(1) u^T \cov(\pertfunc(\statrv)) u
  + o(\norm{u}^2),
\end{align*}
where we use that $\nearlinear(t) = t$ for $t$ near $0$.
Replacing the supremum in the
local minimax lower bound~\eqref{eqn:local-minimax-lower} by any
$\phi$-divergence ball, where we let $x_P$ denote the minimizer
of problem~\eqref{eqn:problem} with distribution $P$ on the data $\statrv$,
yields
\begin{corollary}
  Let the conditions of Theorem~\ref{theorem:local-minimax-lower} hold and
  $\phi : \R \to \R \cup \{+\infty\}$ be convex. Assume that $\phi$ is
  $\mc{C}^2$ in a neighborhood of $1$. Then
  \begin{equation*}
    \liminf_{c \to \infty}
    \liminf_{k \to \infty}
    \sup_{P : \dphis{P}{P_0} \le c / k}
    \E_{P^k}\left[L(\sqrt{k}(\what{x}_k - x_P))\right]
    \ge \E[L(Z)]
  \end{equation*}
  where $Z \sim \normal(0, \projmat[\tangentset]
  {H\opt}^\dag \projmat[\tangentset] \covmat \projmat[\tangentset]
  {H\opt}^\dag \projmat[\tangentset])$.
\end{corollary}
\noindent
That is, our lower bounds imply lower bounds for natural nonparametric
choices of the neighborhood of $P_0$.

It is possible to prove a somewhat stronger result than
Theorem~\ref{theorem:local-minimax-lower}, which we do not do for
simplicity, where instead of the inner supremum over all vectors $u$ such
that $\ltwo{u} \le c / \sqrt{k}$, we take an integral against the uniform
measure $\pi$ supported on the ball $\{u : \ltwo{u} \le c / \sqrt{k}\}$ (see
the constructions in \citet[Chs.\ 6--7]{LeCamYa00}).
We then have a super-efficiency result~\cite{VanDerVaart97}: if
$\what{x}_k$ denotes an estimator based on the sample $\statrv_1,
\ldots, \statrv_k$, the set of $u \in \R^d$ for
problems~\eqref{eqn:stochastic-u-opt-problem} for which $\what{x}_k$ achieves
$\limsup_k \E_{P_u^k}[L(\sqrt{k} (\what{x}_k - x_u))] <
\E[L(Z)]$, for $Z$ as in the theorem, has Lebesgue measure zero.

\section{Convergence and manifold
  identification for dual averaging}
\label{sec:dual-averaging-convergence}

As we discuss following Corollary~\ref{corollary:shapiro-normality}, the
$\what{x}_k = \argmin_{x \in
  \xdomain} \frac{1}{k} \sum_{i = 1}^k \loss(x; \statrv_i)$ achieves optimal
asymptotic convergence. In this and the next section, we investigate the
possibilities of efficient purely online stochastic gradient-based
estimators. These have advantages---small storage space requirements, and
they take a single pass through the data---that make them especially
suitable for modern large-scale regimes~\cite{Zinkevich03, BottouBo07,
  ShalevSr08, NemirovskiJuLaSh09}. We study three aspects of these methods:
identification of the active constraints (those $i$ such that $f_i(x\opt) =
0$), almost sure convergence, and optimal asymptotic behavior.  While
stochastic gradient descent methods fail to even identify the active
constraints, we develop a variant of Nesterov's dual
averaging~\cite{Nesterov09} that identifies active constraints in finite
time and (as we show in the next section) is asymptotically optimal when the
set $\xdomain$ is a polytope; when the constraints are nonlinear,
significant difficulties arise, which we also discuss.

We first consider the stochastic gradient method~\cite{RobbinsMo51,
  PolyakJu92, NemirovskiJuLaSh09} for problem~\eqref{eqn:problem}, to
minimize $f(x)$ subject to $x \in \xdomain$.  This procedure requires a
stochastic gradient oracle, which at each iteration provides a random vector
$g_k$ satisfying $\E[g_k \mid x_k] = \nabla f(x_k)$.  In
problem~\eqref{eqn:problem}, drawing $\statrv_k \sim P$ and computing $g_k =
\nabla \loss(x_k; \statrv_k)$ evidently satisfies this condition.  Given
stochastic gradients $g_k$, the stochastic gradient method iteratively
updates
\begin{equation}
  \label{eqn:standard-sgd}
  x_{k + 1} = \argmin_{x \in \xdomain} \left\{\<g_k, x - x_k\>
    + \frac{1}{2 \stepsize_k} \ltwo{x - x_k}^2 \right\},
\end{equation}
where $\stepsize_k \propto k^{-\steppow}$ for some $\steppow \in [\half, 1]$ is
a stepsize. While the iterates~\eqref{eqn:standard-sgd} converge to the
global optimum $x\opt$, they fail to identify
optimal constraints~\cite{LeeWr12}. As a simple example, we may consider a
problem with $f(x) = x$ and $\xdomain = [-1, 1] = \{x \mid
x^2 - 1 \le 0\}$, which satisfies the assumptions of
Theorem~\ref{theorem:local-minimax-lower} and has $x\opt = -1$.  Consider
stochastic gradients $g_k = 1 + \noise_k$ for $\noise_k \simiid \normal(0,
1)$; the iteration~\eqref{eqn:standard-sgd}
satisfies $\P(x_k \ge -1 + \stepsize_k) \ge 1 - \Phi(1)$, where
$\Phi$ is the standard normal CDF. That is, $x_k \ge
-1 + \stepsize_k$ with constant probability at each iteration---it jumps off
of the constraint infinitely often.

This instability is one of the motivations for Nesterov's dual averaging
algorithm~\cite{Nesterov09}, which iterates
\begin{equation}
  \label{eqn:standard-dual-averaging}
  z_k = \sum_{i = 1}^k g_i,
  ~~ x_{k + 1} = \argmin_{x \in \xdomain}
  \left\{\<z_k, x\> + \frac{1}{2 \stepsize_k} \ltwo{x}^2 \right\}.
\end{equation}
Practically, this
procedure has much better constraint identification properties~\cite{Xiao10,
  LeeWr12} because of the averaging effects in the
definition of $z_k$. \citet{Xiao10} notes its strong performance in
application to $\ell_1$-regularized problems, while \citet{LeeWr12} give
arguments showing that dual averaging spends most of its time on
the ``optimal manifold'' for a variant of problem~\eqref{eqn:problem}, which
essentially corresponds to the set of zeros of the active constraints $\{x :
f_i(x) = 0, i \in [\numactive]\}$.
The work~\cite{LeeWr12} motivates this section, and we are able
to show finite identification of the optimal constraints for a variant of
the dual averaging method and its probability 1 convergence.

\subsection{Almost sure convergence}

We study a variant of dual averaging, which we view as a lazy-projected
gradient algorithm, as it interpolates the stochastic
gradient method and dual averaging. Given a sequence of positive stepsizes 
$\{\stepsize_k\}_{k\in \N}$, initializing $z_0 = 0$,
at each iteration $k$, we update
\begin{equation}
  \begin{split}
    & \mbox{Update} ~
    x_{k} = \argmin_{x \in \xdomain}
    \Big\{\<z_{k-1}, x\> + \half \ltwo{x}^2 \Big\} \\
    & \mbox{Draw}~ \statrv_k \simiid P, ~~
    \mbox{compute}~  g_{k}  = \nabla \loss(x_{k}; \statrv_{k}), ~~
    \mbox{set}~ z_{k} = z_{k-1} + \stepsize_{k} g_{k} \\
  \end{split}
  \label{eqn:variant-dual-averaging}
\end{equation}
In contrast to the standard dual averaging
update~\eqref{eqn:standard-dual-averaging},
procedure~\eqref{eqn:variant-dual-averaging} constructs $z_k$ as a weighted
average and regularizes with $\half \ltwo{x}^2$. This has two consequences:
first, in the unconstrained case, we recover the stochastic
gradient method, which \citet{PolyakJu92} show (when combined with
averaging) is asymptotically normal with optimal
covariance.  The form~\eqref{eqn:variant-dual-averaging} also allows us to
prove the convergence $x_k \cas x\opt$
and finite time identification results. 
Without further comment, we assume the stepsizes $\stepsize_k$ satisfy
\begin{equation}
  \label{eqn:stepsizes}
  \stepsize_k = \stepsize_0 k^{-\steppow}
  ~~ \mbox{where} ~~
  \stepsize_0 > 0 ~~ \mbox{and} ~~
  \half < \steppow < 1.
\end{equation}

\newcounter{saveassumption}
\setcounter{saveassumption}{\value{assumption}-1}
\renewcommand{\theassumption}{\Alph{saveassumption}'}

We may prove our results under slightly weaker conditions than the i.i.d.\
sampling assumed in the update~\eqref{eqn:variant-dual-averaging}, which we
specify now for completeness. In particular, we assume that at each iteration
$k$ we observe a noisy gradient $g_k = \nabla f(x_k) + \noise_k(x_k)$, where
$\noise_k : \xdomain \to \R^n$ is a random function with the property that
$\E[\noise_k(x)] = 0$ for all $x \in \xdomain$. We make the following assumption.
\begin{assumption}
  \label{assumption:noise-vectors}
  Define the filtration
  $\mc{F}_k \defeq \sigma(\noise_1, \ldots, \noise_k)$. The noise
  $\noise_k$ has the decomposable structure
  $\noise_k(x) = \noisezero_k + \noisier_k(x)$, where
  $\noisezero_k$ and $\noisier_k(x)$ are both martingale difference
  sequences adapted to the filtration $\mc{F}_k$.
  There exists a constant $C < \infty$ such that
  \begin{equation*}
    \E[\norms{\noisezero_k}^2 \mid \mc{F}_{k-1}] \le C
    ~~~ \mbox{and} ~~~
    \E[\norms{\noisier_k(x)}^2 \mid \mc{F}_{k-1}]
    \le C \norm{x - x\opt}^2.
  \end{equation*}
  Additionally, 
  $\frac{1}{\sqrt{k}} \sum_{i = 1}^k \noisezero_i \cd \normal(0, \covmat)$
  for some $\covmat \succeq 0$.
\end{assumption}
\noindent
Assumptions~\ref{assumption:smoothness-of-objective}
(smoothness of $f$) and~\ref{assumption:noise-level} (variance bounds on
$\nabla \loss(x; \statrv)$) imply \ref{assumption:noise-vectors}
when $g_k = \nabla \loss(x_k; \statrv_k) = \nabla \loss(x_k; \statrv_k)
- \nabla \loss(x\opt; \statrv_k) + \nabla \loss(x\opt; \statrv_k)$ as in the
update~\eqref{eqn:variant-dual-averaging}. The additional generality causes
no special difficulty in the proofs, so for the remainder of this paper we
let Assumption~\ref{assumption:noise-vectors} hold.

\setcounter{assumption}{\value{saveassumption}+1}
\renewcommand{\theassumption}{\Alph{assumption}}

We begin with the almost sure convergence of $x_k$. This
a.s.\ convergence requires no constraint
qualifications, just that there exists $\epsilon > 0$, 
such that $f(x) - f(x\opt) \ge \epsilon \norm{x - x\opt}^2$ 
for $x \in \xdomain$ near $x\opt$.
\begin{theorem}
  \label{theorem:as-convergence}
  Let $x_k$ be generated by the dual averaging
  iterates~\eqref{eqn:variant-dual-averaging} with
  stepsizes~\eqref{eqn:stepsizes}, let
  Assumptions~\ref{assumption:smoothness-of-objective} and
  \ref{assumption:noise-vectors} (or \ref{assumption:noise-level}) hold, and
  let the growth condition on $f$ in the conclusion of
  Lemma~\ref{lemma:restricted-strong-convexity} hold.  Then
  \begin{equation*}
    x_k \cas x\opt.
  \end{equation*}
\end{theorem}
\noindent
See Section~\ref{sec:proof-as-convergence} for a proof of the theorem.

\subsection{Constraint identification}

\newcommand{\kident}{k_{\rm ident}}

To segue into our results on identification of the optimal
surface of the constraint set $\xdomain$, note that
Theorem~\ref{theorem:as-convergence} implies inactive constraints are
inactive at some finite time: for some (random) $k <
\infty$ we have $\sup_{l \ge k} f_i(x_l) < 0$ for $i >
\numactive$.
Conversely, Theorem~\ref{theorem:as-convergence} says little about whether
$x_k$ identifies the constraints active at $x\opt$.

In brief, under the constraint qualifications of
Assumption~\ref{assumption:linear-independence-cq}, for the modified dual
averaging iteration~\eqref{eqn:variant-dual-averaging}, there is a (random)
iterate $\kident$ such that for $k \ge \kident$, we have $f_i(x_k) = 0$ for
$i \in [\numactive]$.  To provide this guarantee, we give our second set of
results on perturbation of optimal solutions to convex programs, showing
that solutions to linearized versions of problem~\eqref{eqn:problem}
belong to $\{x : f_i(x) = 0, i \le \numactive\}$. The linear approximation
(as opposed to the quadratic approximations in
Proposition~\ref{proposition:perturbation}) is a less immediate application
of the results on parametrized optimization~\cite{BonnansSh98, Shapiro89,
  Wright93}, but (nearly) linear minimization problems dovetail with the
updates~\eqref{eqn:variant-dual-averaging}.

We give a few heuristics. Consider the problem
\begin{equation}
  \label{eqn:optimal-linear-opt}
  \minimize_x  ~ \<\nabla f(x\opt), x\> ~~~
  \subjectto  ~ f_i(x) \le 0, ~~ i = 1, \ldots, m,
\end{equation}
which has a linear objective. By
Assumption~\ref{assumption:linear-independence-cq}, the point $x\opt$
satisfies the KKT conditions for this problem and is optimal, but it may not
be unique.  The dual averaging iteration~\eqref{eqn:variant-dual-averaging}
eventually approximates a slightly perturbed version of the linear
objective~\eqref{eqn:optimal-linear-opt} because $x_k \cas x\opt$ and we
expect $\sum_{i = 1}^k \stepsize_i g_i = \sum_{i = 1}^k \stepsize_i \nabla
f(x_i) + o(\sum_{i = 1}^k \stepsize_i)$. This motivates the next two
perturbation results, which we graphically describe in
Figure~\ref{fig:identification}. The intuition for each is that $-z_k$ is
in $\normalcone_\xdomain(x)$ for some $x$ near enough $x\opt$, in which
case the constraint qualifications
(Assumption~\ref{assumption:linear-independence-cq}) imply that the
projected point must lie on the set described by the active constraints
at $x\opt$.

\paragraph{Nonlinear constraints}

We begin with a perturbation result for the case in which the constraints
are nonlinear, as the linear independence constraint qualification
(Assumption~\ref{assumption:linear-independence-cq}.\ref{item:linear-independence-cq})
makes the argument easier in this case.
Let $x\opt$ be a point such that $f_i(x\opt) = 0$ for
$1 \leq i \leq \numactive$ and $f_i(x\opt) < 0$ for
$\numactive+1\leq i\leq m$. Let $\lambda\opt \in \R^\numactive$ with
$\lambda\opt > 0$ be otherwise arbitrary, and define
$g = -\sum_{i=1}^\numactive \lambda_i\opt \nabla f_i(x\opt)$. Let
$x_0 \in \R^n$, and $v \in \R^n$ and $\delta > 0$, and consider
the tilted and quadratically perturbed version of
problem~\eqref{eqn:optimal-linear-opt}
\begin{equation}
  \label{eqn:nonlinear-constraints-prob}
  \begin{split}
    \minimize_x ~~ & \<g, x\> + \<v, x\> + \frac{\delta}{2} \norm{x - x_0}^2 \\
    \subjectto ~~ & f_i(x) \le 0, ~ i = 1, \ldots, m.
  \end{split}
\end{equation}
The problem~\eqref{eqn:nonlinear-constraints-prob} has a unique minimizer that
we denote $x_{v,\delta}\opt$.  Then we have the following lemma, whose proof we
provide in Supplement, Sec.~\ref{sec:proof-perturbation-nonlinear}.
\begin{lemma}
  \label{lemma:perturbation-result-nonlinear}
  Let the sequence $(v_k, \delta_k) \in \R^n \times \R_{++}$ satisfy $v_k
  \to 0$, $\delta_k \to 0$, and that $x_k \defeq x\opt_{v_k, \delta_k} \to
  x\opt$ as $k \to \infty$.  Then there exists $K < \infty$ such that
  $f_i(x_k)=0$ for $i \in [\numactive]$ and $k \ge K$.
\end{lemma}

\paragraph{Linear constraints}

Considering linear constraints allows weaker assumptions than the case in
which the constraints $f_i$ are nonlinear. Assume that the matrix $A \in
\R^{\numactive \times n}$ and vector $b \in \R^\numactive$ represent the
active constraints, while $C \in \R^{(m - \numactive) \times n}$ and $d \in
\R^{m - \numactive}$ coincide with the inactive constraints, so that $A
x\opt = b$ and $C x\opt < d$.  Specializing the
problem~\eqref{eqn:optimal-linear-opt} and the tilted
problem~\eqref{eqn:nonlinear-constraints-prob} to this setting, for $(v,
\delta) \in \R^n \times \R_+$ we consider
\begin{equation}
  \label{eqn:perturbed-linear-optimization}
  \begin{split}
    \minimize_{x} &~~  \<g, x\> + \<v, x\> + \frac{\delta}{2}
    \norm{x - x_0}^2 \\
    \subjectto &~~ Ax \leq b,
    ~~ Cx \leq d.
  \end{split}
\end{equation}
As before, we assume that for some $\lambda\opt \in \R_{++}^{m_0}$ we have
$g = A^T \lambda\opt$ so that $x\opt$ is a minimizer of
problem~\eqref{eqn:perturbed-linear-optimization} at $v = 0, \delta = 0$.
The next lemma is the analogue of
Lemma~\ref{lemma:perturbation-result-nonlinear} for the linear case. As in
Lemma~\ref{lemma:perturbation-result-nonlinear}, $x\opt_{v,\delta}$ denotes
the unique optimum for the perturbed
problem~\eqref{eqn:perturbed-linear-optimization} with $\delta > 0$. We
provide a proof of the lemma in
Supplement Sec.~\ref{sec:proof-perturbation-linear}.
\begin{lemma}
  \label{lemma:perturbation-result-linear}
  Let the sequence $(v_k, \delta_k) \in \R^n \times \R_{++}$ satisfy $v_k
  \rightarrow 0$, $\delta_k \rightarrow 0$, and that $x_k \defeq x_{v_k,
    \delta_k}\opt \rightarrow x\opt$ as $k\rightarrow \infty$. Then there
  exists $K < \infty$ such that $A x_k = b$ for $k \geq K$.
\end{lemma}


\begin{figure}
  \providecommand{\normalcone}{\mc{N}}
  \begin{center}
    \begin{overpic}[width=.6\columnwidth]{
        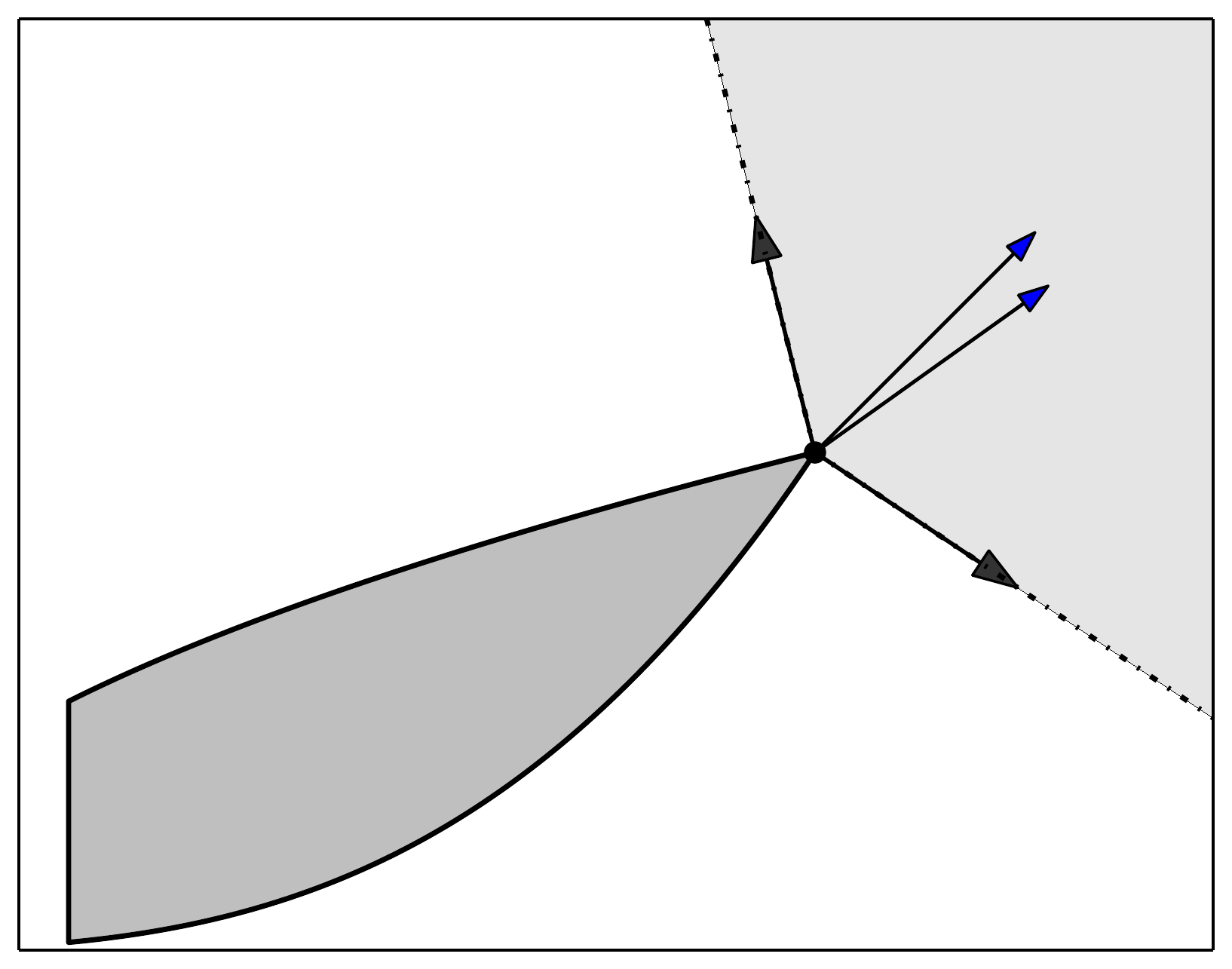}
      \put(30, 41){$\{x : f_1(x) = 0\}$}
      \put(40, 10){$\{x : f_2(x) = 0\}$}
      \put(66, 36){$x\opt$}
      \put(46, 59){$\nabla f_1(x\opt)$}
      \put(81, 33){$\nabla f_2(x\opt)$}
      \put(30, 19){$\xdomain$}
      \put(65, 70){$\normalcone_\xdomain(x\opt)$}
      \put(83, 51){$-g$}
      \put(78, 61){$-(g + v)$}
    \end{overpic}
    \caption{\label{fig:identification} The set $\xdomain
      = \{x \in \R^2 :  x_1 \ge 0, f_1(x) \le 0, f_2(x) \le 0\}$, the
      top and bottom boundaries of $\xdomain$ corresponding to $f_1$ and
      $f_2$. The normal cone $\normalcone_\xdomain(x\opt)$ is the
      convex hull of $\nabla f_1(x\opt)$ and $\nabla f_2(x\opt)$. The
      vectors $-v$ and its perturbation
      $-(v + g)$ both belong to $\relint \normalcone_\xdomain(x\opt)$.}
    %
  \end{center}
\end{figure}


With the identification results provided by
Lemmas~\ref{lemma:perturbation-result-nonlinear}
and~\ref{lemma:perturbation-result-linear}, we can now show a result that
demonstrates that our variant~\eqref{eqn:variant-dual-averaging} of dual
averaging identifies the optimal manifold in
finite time with probability 1.
\begin{theorem}
  \label{theorem:manifold-identification}
  Let
  Assumptions~\ref{assumption:smoothness-of-objective}--\ref{assumption:noise-level}
  (or \ref{assumption:noise-vectors})
  hold. Then with probability one, there exists some (random) $K < \infty$
  such that $k \ge K$ implies
  \begin{equation*}
    f_i(x_k) = 0 ~~ \mbox{for~} i \le \numactive
    ~~~ \mbox{and}~~~
    \sup_{k \ge K} f_i(x_k) < 0 ~~\mbox{for~} i > \numactive.
  \end{equation*}
\end{theorem}
\noindent
We provide the proof of Theorem~\ref{theorem:manifold-identification} in
Section~\ref{sec:proof-manifold-identification}. The outline of the proof,
though, is apparent from the above lemmas and
Theorem~\ref{theorem:as-convergence}. Letting $A_k = \sum_{i = 1}^k
\stepsize_i$, the dual
averaging iterates~\eqref{eqn:variant-dual-averaging}
perform the update
\begin{equation*}
  x_{k + 1}
  = \argmin_{x \in \xdomain}
  \left\{\!\<z_k, x\> + \half \norm{x}^2\right\}
  = \argmin_{x \in \xdomain}
  \bigg\{\!
  \<\nabla f(x\opt) + v_k, x\>
  + \frac{1}{2 A_k} \norm{x}^2 \bigg\}
\end{equation*}
where $v_k = \frac{1}{A_k} (z_k - A_k \nabla f(x\opt)) = o(1)$,
 equivalent to
problems~\eqref{eqn:nonlinear-constraints-prob}
and~\eqref{eqn:perturbed-linear-optimization}.



\section{Stochastic gradient procedures: asymptotic normality}
\label{sec:dual-averaging-normality}

Now that we have established that dual averaging converges almost surely and
in finite time identifies constraints active at $x\opt$, we turn
asymptotic normality results.  We focus first on the case that the
constraints are linear, where dual averaging is locally asymptotically
minimax optimal. As we demonstrate, however, nonlinearity forces
a departure from this optimality. Consequently, in
Section~\ref{sec:riemannian-fun} we develop a joint dual-averaging and
Riemannian stochastic gradient procedure that is both online---it
sequentially computes only a single gradient $\nabla \loss(x;
\statrv_i)$ from each observation---and asymptotically optimal.

\subsection{Dual averaging: asymptotic normality}

When the problem is unconstrained with $\xdomain = \R^n$, \citet{PolyakJu92}
show that under our assumptions, the stochastic gradient method is
asymptotically normal when combined with averaging. In the notation of
Theorem~\ref{theorem:local-minimax-lower}, $\wb{x}_k = \frac{1}{k} \sum_{i =
  1}^k x_i$ satisfies $\sqrt{k} (\wb{x}_k - x\opt) \cd \normal(0, \nabla^2
f(x\opt)^{-1} \cov(\nabla \loss(x\opt; \statrv)) \nabla^2 f(x\opt)^{-1})$,
which is optimal.  In the constrained case, identical results hold if we
solve the problem over a subspace (i.e.\ $\{x : Ax = b\}$); there are no
differences from the classical case~\cite{PolyakJu92}. We thus expect our
dual averaging variant to behave as follows: eventually, we identify the
active constraints, that is, we have $A x_k = b$ and $C x_k < d$ for all
sufficiently large $k$.  Once this occurs, the iterations of the dual
averaging variant are \emph{identical} to those of the stochastic gradient
method in the subspace $\{x : Ax = b\}$. Thus, we expect asymptotic
normality, with the asymptotic covariance reflecting variability only in the
null space of $A$. While our development tracks this idea, the
``sufficiently large $k$'' for active set identification is random, and to
have $A x_k = b$ for all $k$ depends on the entire future noise sequence
$\{\noise_i\}_{i = k}^\infty$, making this intuitive argument fail.  With a
bit more delicacy, we can provide a similar argument that builds off of
Polyak and Juditsky's treatment. Now, define the orthogonal
projector onto the null space $\{w : Aw = 0\}
= \tangentset_\xdomain(x\opt)$,
\begin{equation*}
  \projmat \defeq I - A^T (AA^T)^\dag A.
\end{equation*}
We then have the following theorem.
\begin{theorem}
  \label{theorem:asymptotic-normality}
  Let
  Assumptions~\ref{assumption:smoothness-of-objective}--\ref{assumption:noise-vectors}
  hold, and assume that $\stepsize_k \propto k^{-\beta}$ for some $\beta \in
  (\half, 1)$. Let $\covmat = \cov(\grad \loss(x\opt; \statrv))$. Then
  \begin{equation*}
    \frac{1}{\sqrt{k}} \sum_{i=1}^k (x_i - x\opt) \cd 
    \normal \left(0, \projmat (\nabla^2 f(x\opt))^\dag \projmat \Sigma \projmat 
    (\nabla^2 f(x\opt))^\dag \projmat\right).
  \end{equation*}
\end{theorem}
\noindent
We defer the
proof of Theorem~\ref{theorem:asymptotic-normality} to
Supplement Sec.~\ref{sec:proof-asymptotic-normality}.

\subsection{Slow convergence for nonlinear constraint sets}
\label{sec:slow-convergence}

Theorems~\ref{theorem:as-convergence}
and~\ref{theorem:manifold-identification} guarantee almost sure convergence
and finite time constraint identification, but
Theorem~\ref{theorem:asymptotic-normality} provides an optimal convergence
rate only when the constraints are linear,
and this is fundamental. Indeed, we provide two results
showing the sub-optimality of dual averaging (both our variant and
Nesterov's original version~\cite{Nesterov09}) on a simple optimization
problem.

To make this failure concrete, let $e_1$ be
the first standard basis vector. Consider the problem (for $n \ge 2$) 
with $\statdomain = \R^n$, $\statrv \sim \normal(0, I)$, 
$\xdomain = \{x: \norm{x}^2-1 \le 0\}$ and 
$\loss(x; \statval) =  -(e_1 + \statval)^T x$. In this case, 
problem~\eqref{eqn:problem} becomes
\begin{equation}  
\label{eqn:simple-linear-ball-problem} 
  \minimize_{x \in \R^n} ~ -\! e_1^T x 
  ~~ \subjectto ~ \ltwo{x}^2 \le 1. 
\end{equation}
The optimum for program~\eqref{eqn:simple-linear-ball-problem} is $x\opt =
e_1$.  The Lagrangian for the problem is $\mc{L}(x, \lambda) = -e_1^T x +
\frac{\lambda}{2} (\ltwo{x}^2 - 1)$ with optimal dual multiplier $\lambda\opt =
1$, whence
Corollary~\ref{corollary:shapiro-normality} and the lower bound of
Theorem~\ref{theorem:local-minimax-lower} show that the optimal asymptotic
covariance is $I - e_1 e_1^T$.  As we show, however,
dual averaging and our variant are
suboptimal even with $g_k = e_1 + \statrv_k$ for
$\statrv_k \simiid \normal(0, I)$.

We first consider the
variant~\eqref{eqn:variant-dual-averaging} of dual averaging
with $z_k = \sum_{i = 1}^k \stepsize_i g_i$.
\begin{observation}
  \label{observation:failure-dual-averaging}
  Let the stepsizes $\stepsize_i = i^{-\steppow}$ for some
  $\steppow \in (\half, 1)$, and let the iterates $x_k$ be generated by the
  dual averaging procedure~\eqref{eqn:variant-dual-averaging}.  Then
  \begin{equation*}
    \frac{1}{k^\steppow} \sum_{i=1}^k (x_i -x\opt) \cd \normal \left(0, 
      \sigma_\steppow^2(I-e_1 e_1^T)\right), 
    ~~ \mbox{where} ~~
    \sigma_\steppow^2 \defeq \frac{(1-\steppow)^2}{\steppow^2} \sum_{i=1}^\infty \stepsize_i^2.
  \end{equation*}
\end{observation}
\noindent
See Supplement Sec.~\ref{sec:proof-failure-dual-averaging} for a proof.
In this case, even the \emph{rate} of convergence is lost: denoting
$\wb{x}_k = \frac{1}{k}\sum_{i=1}^k x_i$, then asymptotic normality holds
for $\wb{x}_k - x\opt$, but $\wb{x}_k - x\opt$ is order $k^{\steppow-1} \gg
k^{-\half}$,

Our second observation applies to dual averaging with $z_k = \sum_{i =
  1}^k g_i$.
\begin{observation}
  \label{observation:failure-classical-dual-averaging}
  Let the stepsize sequence $\stepsize_k \propto k^{-\steppow}$ for some
  $\steppow \in \openright{0}{1}$. Then the classical dual
  averaging~\eqref{eqn:standard-dual-averaging} iterates satisfy
  \begin{equation*}
    \frac{1}{\sqrt{k}} \sum_{i = 1}^k (x_i - x\opt) \cd
    \normal \left(0, 2 (I-e_1 e_1^T)\right).
  \end{equation*}
\end{observation}
\noindent
See Supplement
Sec.~\ref{sec:proof-failure-classical-dual-averaging} for a proof.

We give a bit of intuition for the difficulty in
Observations~\ref{observation:failure-dual-averaging}
and~\ref{observation:failure-classical-dual-averaging}. We have that
$\sum_{i = 1}^k \stepsize_i g_i = (\sum_{i = 1}^k
\stepsize_i) \nabla f(x\opt) + \sum_{i = 1}^k \stepsize_i \noise_i$, where
$\noise_i \simiid \normal(0, I)$. But in projecting to the curved surface of
the ball $\{x : \ltwo{x} \le 1\}$, there is still sufficient noise in the
sum $\sum_{i = 1}^k \stepsize_i \noise_i$ to induce variance. In the case of
linear constraints $Ax \le b$, the vector $z_k = \sum_{i = 1}^k \stepsize_i
g_i$ eventually lies in the normal cone to the active face
$\{x : Ax = b\}$, so that projections force all iterates into the subspace
$\{x : Ax = b\}$, with no curvature for additional variance. Stochastic
gradient descent---which fails to even identify the active
constraints---similarly has sub-optimal rates for this problem.

\subsection{A Riemannian stochastic gradient procedure}
\label{sec:riemannian-fun}

The challenges we outline in Section~\ref{sec:slow-convergence} for
classical dual averaging and stochastic gradient methods necessitate
alternative algorithms for asymptotically optimal online procedures. To that
end, we develop an algorithm that alternates between dual averaging and a
stochastic gradient method on the manifold of the active constraints.  The
intuition is that we use dual averaging~\eqref{eqn:variant-dual-averaging}
to identify the optimal manifold, then use a Riemannian stochastic
gradient-like method~\cite{Bonnabel13, TripuraneniStJiReJo17} on the active
manifold. Letting $\M = \{x : f_i(x) = 0, i \in [\numactive]\}$ denote the
optimal manifold on which the solutions lie, two challenges arise in
the analysis of any such method. First, projections onto $\M$ are not
necessarily nonexpansive---a major component of most
analyses of stochastic gradient-based methods---so that showing convergence
of a pure Riemannian method is challenging.\footnote{Many papers on
  Riemannian stochastic gradient methods assume
  convergence, or that iterates remain in
  a small neighborhood of $x\opt$, as a condition; cf.\ \cite[Assumption
    2]{Bonnabel13,TripuraneniStJiReJo17}.} Even in noiseless settings,
gradient descent and
other first-order methods do not enjoy global convergence results for
minimization of convex $f : \R^n \to \R$
on Riemannian manifolds~\cite{AbsilMaSe09, AbsilBaGa07, Boumal14}.

To that end, we present Algorithm~\ref{alg:da-riemannian}, which is complex
and perhaps of more intellectual than practical interest,
but fulfills our desiderata of being (i) fully online, (ii) convergent with
probability 1, and (iii) asymptotically optimal.
To describe the algorithm and its convergence,
we require somewhat more notation.
For a closed set $\mc{M}$, let $\proj_{\mc{M}}(x)
= \argmin_{y \in \mc{M}}\{\norm{x - y}\}$ denote the Euclidean
projection of $x$ onto $\mc{M}$, with an arbitrary
rule for choosing the projecting if it is non-unique.
When the set $\mc{M} = \{x \in \R^n:
G(x) = 0\}$ for a continuously differentiable $G: \R^n \to
\R^l$, we let
$\grad G(x) = [\grad g_1(x) \cdots \grad g_l(x)] \in \R^{n \times
  l}$ and denote the tangent space to
$\mc{M}$ at $x$ by
\begin{equation*}
  \T_{\mc{M}}(x) \defeq \{v\in \R^n: \grad G(x)^T v = 0\},
\end{equation*}
and we define the orthogonal projector
\begin{equation*}
  \projt{x} = I - \grad G(x) (\grad G(x)^T \grad G(x))^{\dag}
  \grad G(x)^T
  \in \R^{n \times n}.
\end{equation*}

\begin{algorithm}[t]
  \caption{\label{alg:da-riemannian}
    Riemannian Stochastic Gradient with Dual Averaging}

  \begin{algorithmic}[1] 
    \State Initialize $k = 0$, $\rsgiter_0 \in \xdomain$, $\M_0 =
    \emptyset$. Input $q \in (0, 1)$, dual averaging
    times $\timesda$ with
    $1 \lesssim |\{i \in \timesda \mid i \le k\}| / k^\rho \lesssim 1$
    for some $\rho \in (0, 1)$, and stepsizes
    $\stepsize_k = \stepsize_0 k^{-\steppow}$ with $\steppow \in (\half, 1)$.
    Require that $q < \min\{\frac{1 - 2 \steppow}{(1 - \steppow) \rho},
    \frac{1}{2(1 - \rho) \steppow}\}$.

    \For{$k = 1, 2, \ldots$, $k \not \in \timesda$}

    
    \State
    Compute the manifold $\M_k$ that $x\supda_k$ identifies:
    \begin{align*}
      \M_k = \cap_{i \in \indset_k} \left\{x \in \R^n: f_i(x) = 0\right\}
      ~~\text{where}~~
      \indset_k = \left\{i \in [m]: f_i(x\supda_{\datime_k}) = 0\right\}.
    \end{align*}

    \State \label{state:riemann-sgd-manifold-update}
    Let $g_k = \grad f(\rsgiter_k) + \noise_k(\rsgiter_k)$
    and compute the iterate 
    \begin{equation*}
      \rsgiter^{\manifold}_{k+1} = \begin{cases}
	\proj_{\M_k}(\rsgiter_k - \stepsize_k \projmatr[\T_{\M_k}(\rsgiter_k)]
        g_k)
        & ~ \mbox{if}~ \M_k = \M_{k-1} \\
	x\supda_k & ~\text{otherwise}.
      \end{cases}
    \end{equation*}

    \State
    Let $\daaverage_k \defeq (\sum_{i \in \timesda}^{i \le k}
    \stepsize\supda_i)^{-1} \sum_{i \in \timesda}^{i \le k}
    \stepsize\supda_i x\supda_i$.

    \State \label{state:riemann-sgd-stay-X}
    Let $\ball_{k, 1} = \ball(\daaverage_k, \eps_k)$ and
    $\ball_{k, 3} = \ball(\daaverage_k, 3\eps_k)$
    for $\epsilon_k = (\sum_{i \in \timesda}^{i \le k} \stepsize\supda_i)^{-q}$.
    Compute
    \begin{equation*}
      \rsgiter_{k+1} = \begin{cases}
	\proj_{\xdomain}(\rsgiter^{\manifold}_{k+1})
        & \mbox{if}~ \proj_{\xdomain}(\rsgiter^{\manifold}_{k+1}) \in \ball_{k, 3} \\
	\argmin \left\{\norm{x} \mid x \in \M_k \cap \xdomain \cap \ball_{k, 1}
        \right\}
        &
        \mbox{if}~
        \proj_{\xdomain}(\rsgiter^{\manifold}_{k+1}) \not\in \ball_{k, 3},
	\M_k \cap \xdomain \cap \ball_{k, 1} \neq \emptyset \\
	\argmin \left\{\norm{x} \mid x \in \M_k \cap \xdomain\right\}
        & \text{otherwise}.
      \end{cases}
    \end{equation*}
    \EndFor
  \end{algorithmic}
  \label{algorithm:Rdg}
\end{algorithm}

With this notation established, we can describe
Algorithm~\ref{alg:da-riemannian}.  The algorithm alternates between
asymptotically infrequent iterates of dual averaging at iterates $k \in
\timesda$, constructing a sequence $x\supda_k$, and frequent iterates of
Riemannian stochastic gradient-like method that projects onto the active
constraints, the smooth manifold $\M_k = \{x : f_i(x) = 0 ~ \mbox{for}~ i
\in \indset_k\}$ where $\indset_k = \{i \in [m] \mid f_i(x\supda_k) = 0\}$
denotes the constraints dual averaging identifies.  The method takes a
stepsize sequence $\{\stepsize_k\}$ for the Riemannian stochastic gradient
method where $\stepsize_k = \stepsize_0 k^{-\steppow}$. For the dual
averaging iteration times $k \in \timesda$, we set the dual averaging
stepsizes via $\stepsize_k\supda = \stepsize_{t_k}$ where $t_k = |\{i \in
\timesda \mid i \le k\}|$, the same stepsize scaling as the Riemannian
method.  At each step $k \in \timesda$, the method updates the dual
averaging iterate via the update~\eqref{eqn:variant-dual-averaging} (with
$z_k = \sum_{i \in \timesda}^{i \le k} \stepsize_i\supda g_i$). Then for $k
\not\in \timesda$, the method performs a stochastic gradient step
(line~\ref{state:riemann-sgd-manifold-update}) but projects the stochastic
gradient $g_k$ onto the tangent space
of the active manifold $\M_k$. The final step of the
algorithm (line~\ref{state:riemann-sgd-stay-X}) guarantees that the iterates
$\rsgiter_k$ of the method stay near enough the dual averaging iterates,
which allows us to circumvent the difficulties of global convergence for
Riemannian methods.  As we demonstrate in the proof, this
asymptotically iterates a stochastic gradient method in the
tangent space $\T = \{v : \<\nabla f_i(x\opt), v\>
= 0, i \le \numactive\}$, and only updates
line~\ref{state:riemann-sgd-manifold-update} occur, as $\M_k = \M_{k-1}$ and
$\rsgiter_k^{\manifold} \in \xdomain$.

We prove the following theorem
in Supplement Sec.~\ref{sec:proof-asymptotic-normality-rsgd},
using the notation of Proposition~\ref{proposition:perturbation},
where $H\opt = \grad^2 f(x\opt) +
\sum_{i=1}^{\numactive}\lambda_i\opt\grad^2 f_i(x\opt)$
and $\projmatr$ is the projection onto
the tangent space $\T = \{v \in \R^n : v^T \grad f_i(x\opt) = 0,
i \in [\numactive]\}$.

\begin{theorem}
  \label{theorem:asymptotic-normality-rsgd}
  Let Assumptions~\ref{assumption:smoothness-of-objective},
  \ref{assumption:linear-independence-cq},
  \ref{assumption:restricted-strong-convexity}, and
  \ref{assumption:noise-vectors} hold.
  Then
  the iterates $\rsgiter_k$ of Algorithm~\ref{alg:da-riemannian}
  satisfy
  \begin{equation*}
    \frac{1}{\sqrt{k}}\sum_{i=1}^k (\rsgiter_i - x\opt)
    \cd
    \normal\left(0,
    \projmatr {H\opt}^{\dag} \projmatr \Sigma \projmatr {H\opt}^\dag
    \projmatr \right).
  \end{equation*}
\end{theorem}

\noindent
The extended Riemannian stochastic gradient method,
coupled with identification results that dual averaging supplies,
is asymptotically optimal.


\newcommand{\xtrue}{x^{\rm true}}

\section{Numerical experiments}
\label{sec:experiment}

In this section, we perform a small simulation study to compare dual
averaging~\eqref{eqn:variant-dual-averaging} with stochastic (and Riemannian)
gradient methods on nonnegative least squares
and ridge regression.  We take our observations $(a_i, b_i) \in \R^n \times
\R$ and use the squared loss $\loss(x; (a, b)) = \half (\< a, x\> - b)^2$.
Both problems are of the form~\eqref{eqn:problem}, where for nonnegative
least squares, we use the constraints $\xdomain = \R^n_+$, and for the ridge
regression problem, we set $\xdomain = \{x \in \R^n: \norm{x}^2 \le
\lambda\}$, where $\lambda > 0$.

Now we describe our experimental setting. To allow easier visualization, we
use dimension $n = 2$ and generate $b_i = \<a_i, \xtrue\> + \noise_i$ for
$a_i \simiid \normal(0, I_2)$ and $\noise_i \simiid \normal(0, 1)$.  For the
nonnegative least squares problem, we set $\xtrue = (1, -1)$, while for the
ridge regression problem, we set $\xtrue = (1, 1)$ and $\lambda=1$, giving
solutions $x\opt = (1, 0)$ and $x\opt = (\frac{1}{\sqrt{2}},
\frac{1}{\sqrt{2}})$, respectively.  For both problems, the unique solution
$x\opt$ lies on the boundary of the feasible set $\xdomain$. To fairly
compare the performance of the algorithms, we use the same parameters,
initializing at $x = 0$ and using stepsizes $\stepsize_k = k^{-\steppow}$
for $\steppow = 3/4$.  In each experiment, we run each method for $K$
iterations, and we perform $T$ independent replications.

\subsection{Constraint Identification}

Our first set of numerical results shows that the
stochastic gradient method
fails to identify active constraints, while dual averaging identifies
them. We present the results graphically in
Figure~\ref{fig:MI}. For each of the two plots, the
horizontal axis indexes the
iteration $k$ (over $K = 100$ iterations)
and the vertical axis represents the proportion
of the $T = 1000$ tests in which the iterate $x_k$ lies on the active
constraints. Both plots show that the dual averaging iterates (the solid
red curve) identify the constraints
(with 100\% accuracy by iteration 40),
while the stochastic gradient method
(the dotted blue curve) does not.

\begin{figure}[!htb]
\minipage{0.48\textwidth}
  \includegraphics[width=\linewidth]{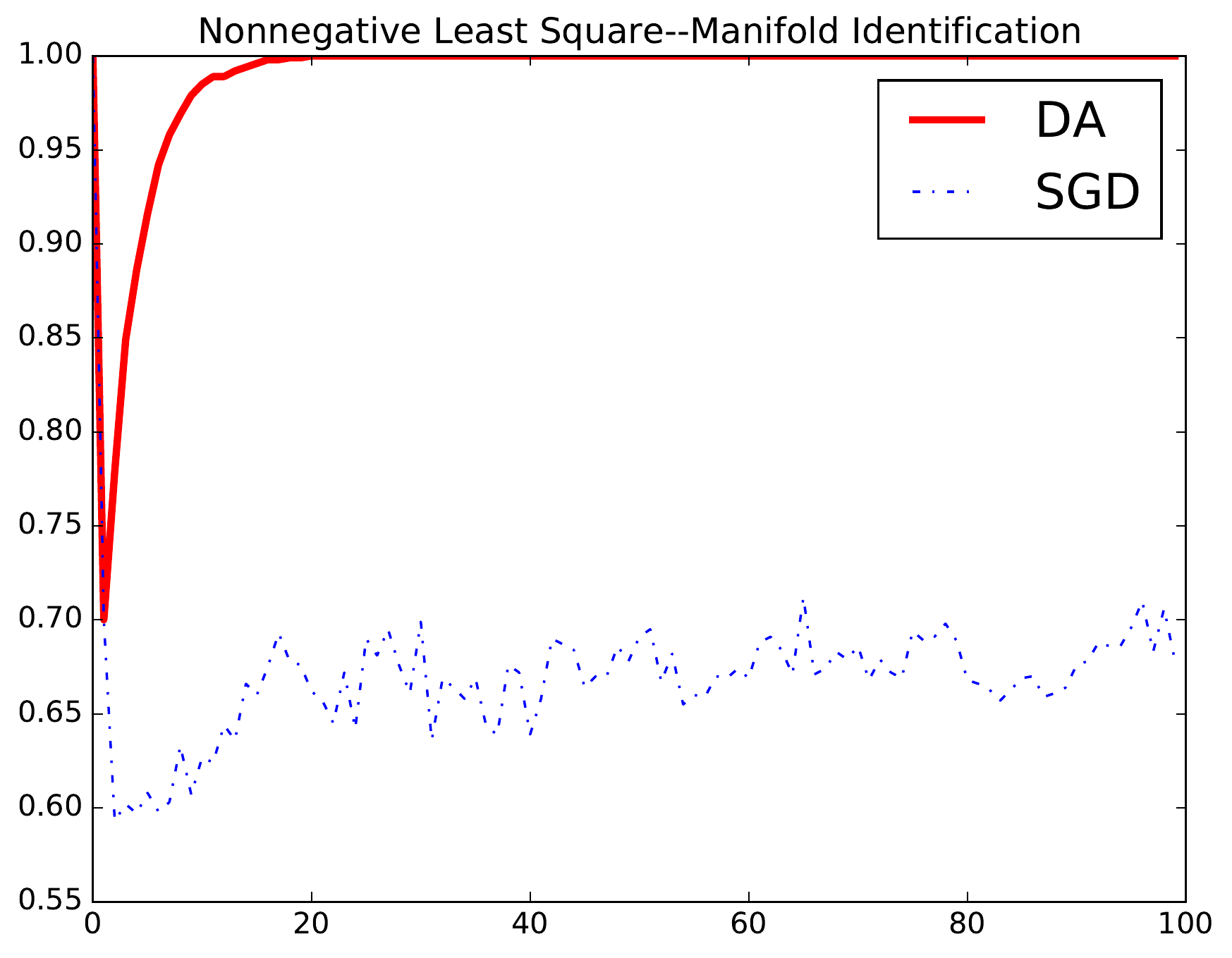}
\endminipage\hfill
\minipage{0.48\textwidth}%
  \includegraphics[width=\linewidth]{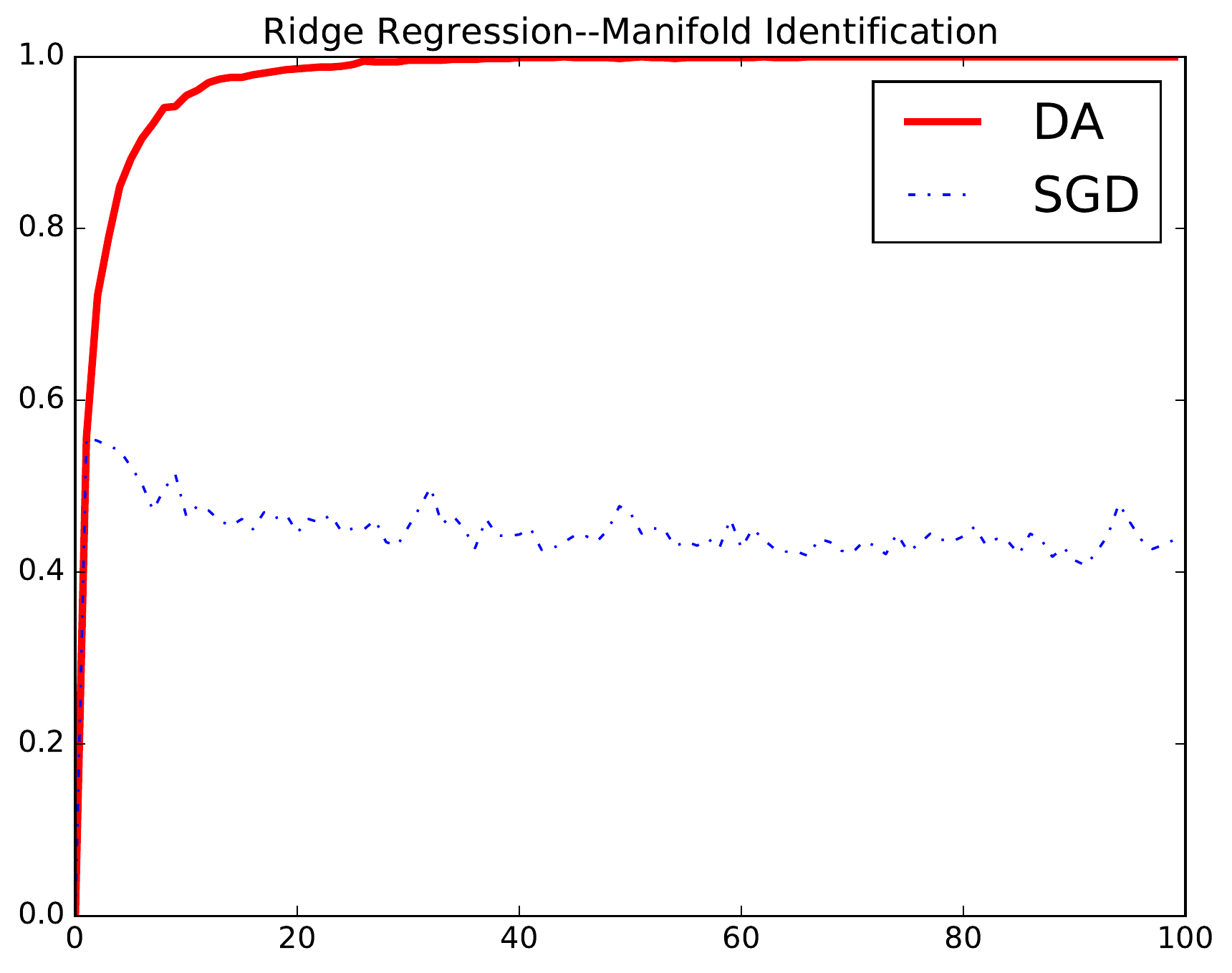}
\endminipage\hfill
\caption{\label{fig:MI}
  Success rate of manifold identification}
\end{figure}

\subsection{Accuracy}

Our second set of numerical results shows the improved performance of dual
averaging relative to projected stochastic gradient descent, and that
the manifold aware algorithm~\ref{alg:da-riemannian} exhibits better
behavior on nonlinear constraints, which we illustrate in
Figure~\ref{fig:final-iterate}.  Each of the red triangles (respectively,
blue circles or green diamonds) represents an averaged dual averaging
(resp., stochastic gradient or Riemannian method~\ref{alg:da-riemannian})
iterate $\bar{x}_K = \frac{1}{K}\sum_{i=1}^K x_i$ (we set $K = 100$) out of
$T = 20$ experiments. The dual averaging results are typically closer to
$x\opt$ (the black cross) and to the active constraints (the grey dotted
curve) than the stochastic gradient averages, while the right plot in
Fig.~\ref{fig:final-iterate} shows the improved performance of the
Riemannian method we outline in Alg.~\ref{alg:da-riemannian}.  The distance
of the dual averaging iterates to $x\opt$ is typically shorter along the
normals to the active constraints, leading to better accuracy estimating
$x\opt$.

\begin{figure}[!htb]
\minipage{0.48\textwidth}
  \includegraphics[width=\linewidth]{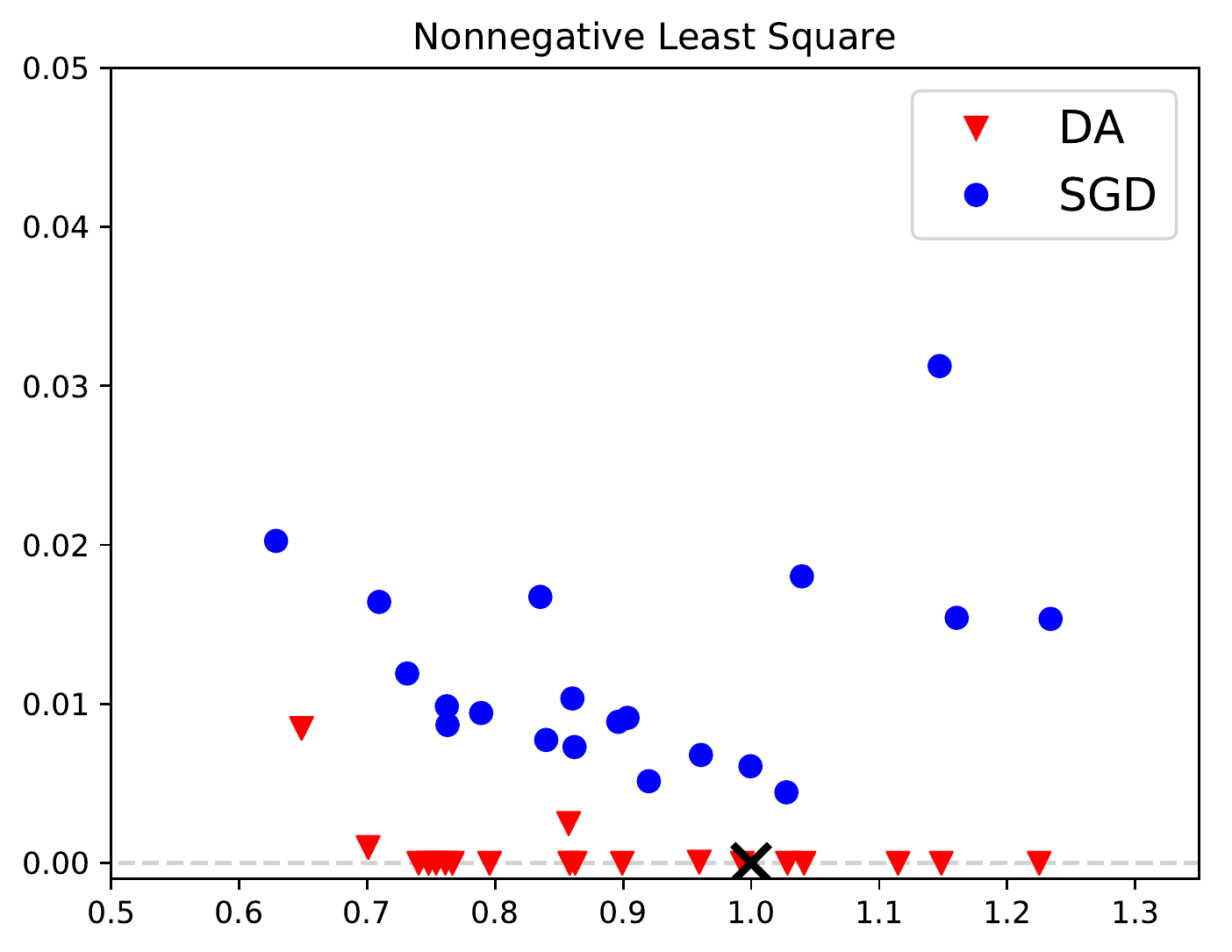}
\endminipage\hfill
\minipage{0.48\textwidth}%
  \includegraphics[width=\linewidth]{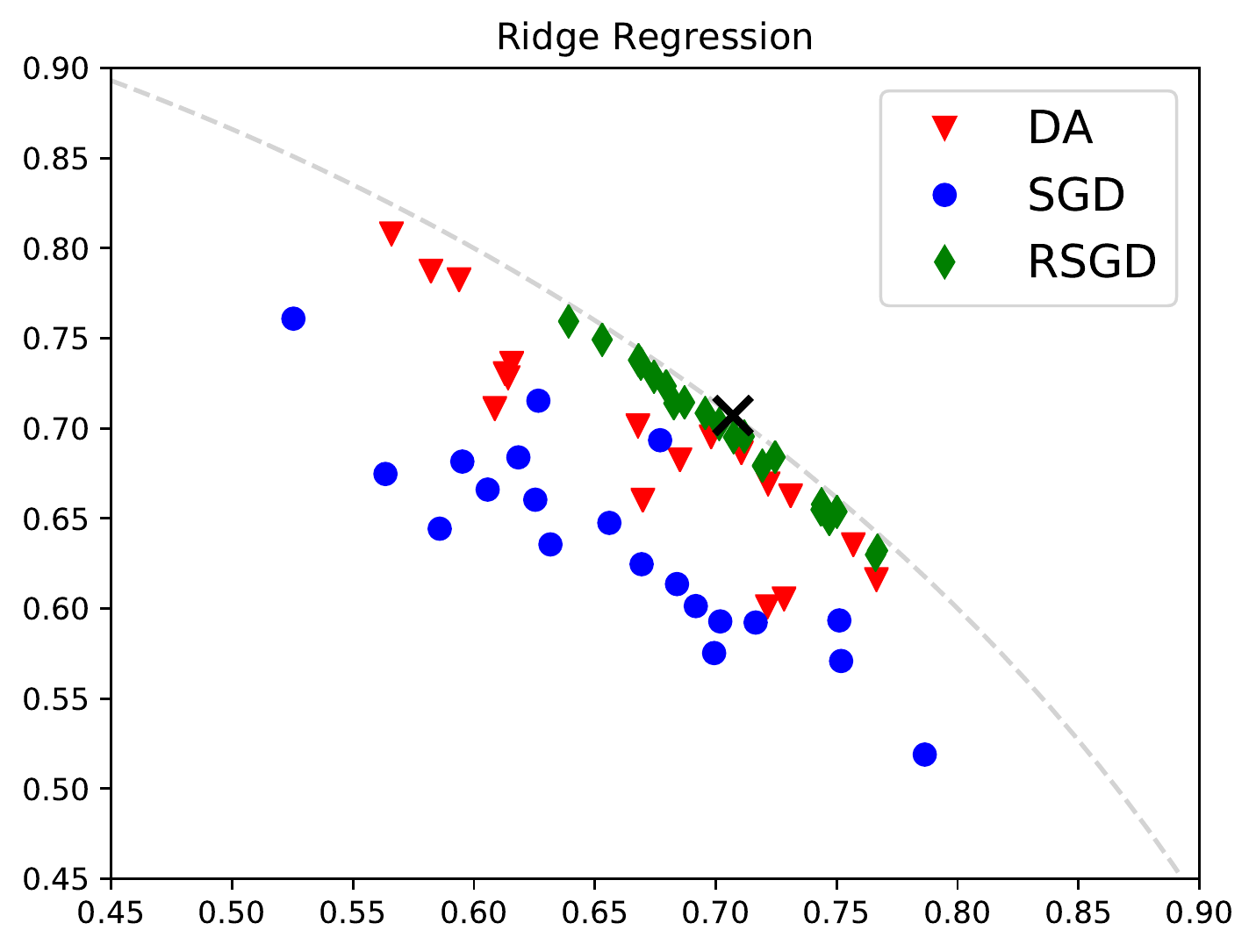}
\endminipage\hfill
\caption{The averaged iterates $\bar{x}_K = \frac{1}{K}\sum_{i=1}^K x_i$}
\label{fig:final-iterate}
\end{figure}

\subsection{Linear versus nonlinear constraints}

For our third set of numerical results, we investigate the asymptotic
variance of the variant dual averaging
method~\eqref{eqn:variant-dual-averaging} versus that of the Riemannian
method (Alg.~\ref{alg:da-riemannian}) and the optimal asymptotic variance
that Theorem~\ref{theorem:local-minimax-lower} provides.  For the
nonnegative least squares problem, the linear constraints have tangent set
$\T = \{t v\opt\}_{t \in \R}$, where $v\opt = (1, 0)$, while the ridge
problem has $\T = \{t v\opt\}_{t \in \R}$ where $v\opt = (1, -1)$. In each
case, we compute the variance of $\sqrt{k} \<v\opt, \wb{x}_k - x\opt\>$ for
$\wb{x}_k = \frac{1}{k} \sum_{i = 1}^k x_i$ for $k \le K = 10^4$ over $T =
1000$ independent trials.  We present the results in
Figure~\ref{fig:asymptotic-var}.  In each of the two plots, the red dashed
curve shows the variance the dual averaging iterates and the gray dotted
line shows the optimal asymptotic variance
(Thm.~\ref{theorem:local-minimax-lower}). In the left plot, we see that dual
averaging converges with the asymptotically optimal rate. In the right,
the Riemannian method (the solid green line) has asymptotically
optimal variance, while the dual averaging procedure has variance
between $3$ and $3.5$, which is suboptimal; these results
suggest the accuracy of our theoretical predictions.

\begin{figure}[!htb]
\minipage{0.48\textwidth}
  \includegraphics[width=\linewidth]{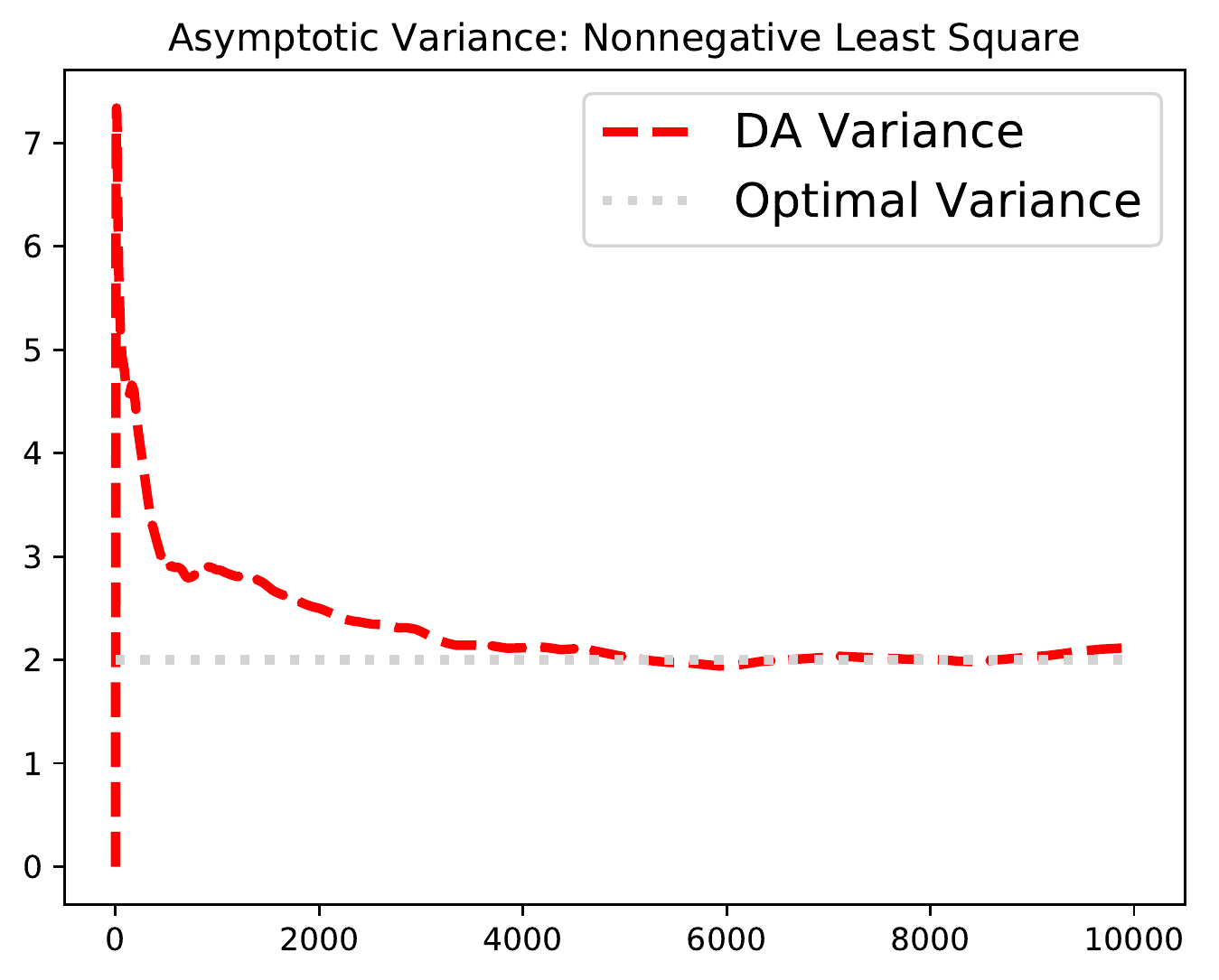}
\endminipage\hfill
\minipage{0.48\textwidth}%
  \includegraphics[width=\linewidth]{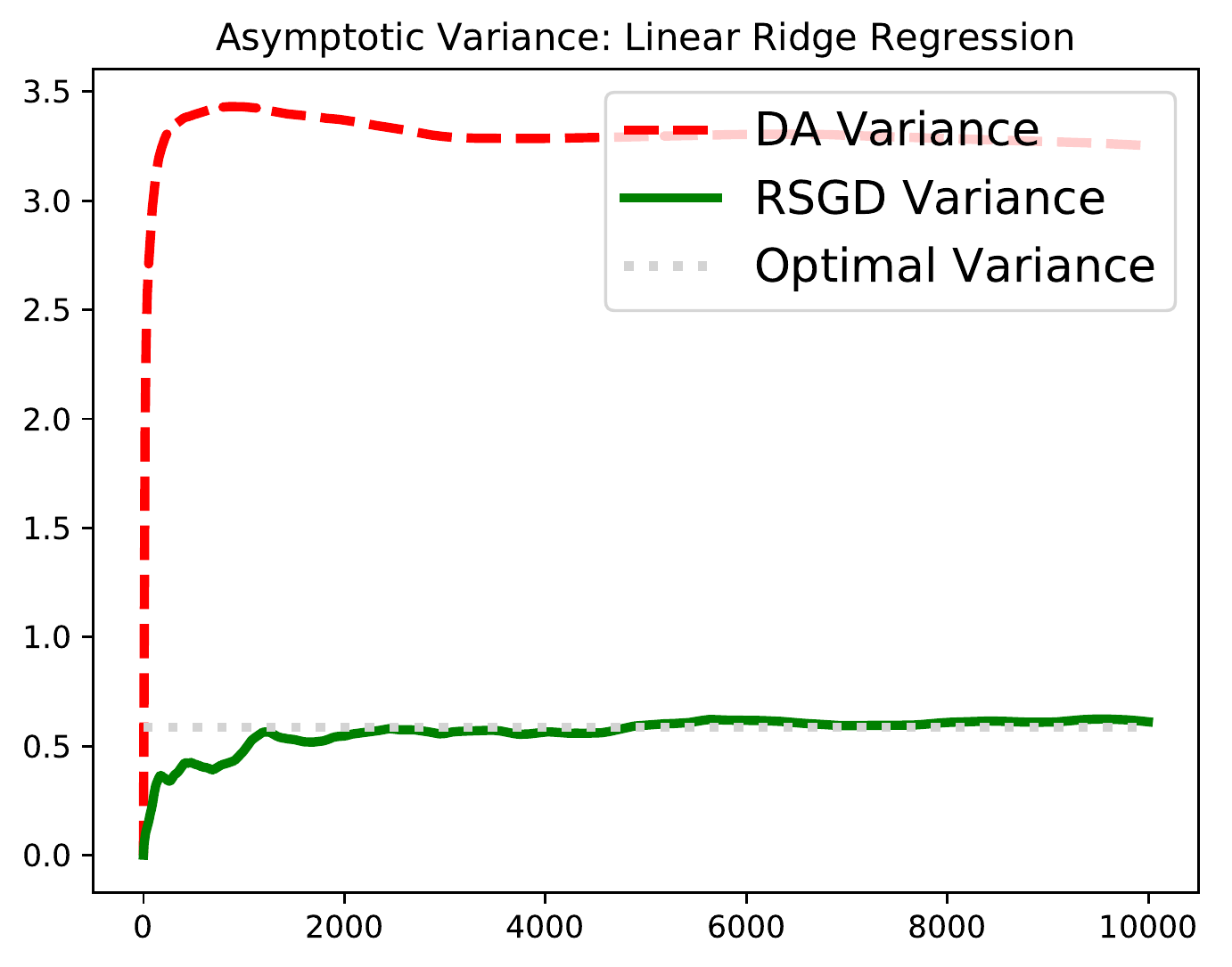}
\endminipage\hfill
\caption{Asymptotic variance under linear and nonlinear constraints}
\label{fig:asymptotic-var}
\end{figure}


\section{Proof of Proposition~\ref{proposition:perturbation}}
\label{sec:proof-perturbation}

This result is a consequence of \citet[Theorem 5.1]{Shapiro88} or
\citet[Theorem 2G.8]{DontchevRo14}. First,
consider the Lagrangian for the tilted problem~\eqref{eqn:tilted-problem},
\begin{equation*}
  \mc{L}_v(x, \lambda) = f_v(x) + \sum_{i = 1}^m \lambda_i f_i(x).
\end{equation*}
We perform a second-order Taylor approximation to
$\mc{L}_v(x, \lambda)$ around $x_0$, linearizing the active constraints
$f_i(x)$ for $i \in [\numactive]$, and minimizers of this quadratic
over linear constraints are $o(\norm{v})$-close to $x_v$.  We
make this precise.

Let $\Lambda_0 \subset \R^m_+$ denote the set of optimal Lagrange
multipliers for problem~\eqref{eqn:problem} (the tilted
problem~\eqref{eqn:tilted-problem} at $v = 0$), recalling that by
Assumption~\ref{assumption:linear-independence-cq}, this set is a compact
polyhedron (and is a singleton
under~\ref{assumption:linear-independence-cq}.i).  In either case of
Assumption~\ref{assumption:linear-independence-cq}, 
the set $\{\nabla^2_x \mc{L}_0(x_0, \lambda): \lambda \in \Lambda_0\}$ is a
singleton, so $H\opt = \nabla^2_x \mc{L}_0(x\opt, \lambda\opt) 
= \nabla^2 \mc{L}_0(x_0, \lambda)$ for any $\lambda \in \Lambda_0$. 
At $v = 0$, our assumptions imply
$\nabla^2_v\mc{L}_v(x_0, \lambda) = 0$ and 
$\nabla^2_{vx} \mc{L}_v(x_0, \lambda) = 0$. Define the quadratic
\begin{equation}
  \label{eqn:quad-approx-lagrange}
  \zeta_v(w)
  \defeq \half w^T H\opt w - v^T w
  = \half w^T \nabla^2_x \mc{L}_0(x_0, \lambda) w
  - v^T w,
\end{equation}
which approximates $\mc{L}_v(x_0 + w, \lambda) \approx f_0(x_0) +
\zeta_v(w)$ for $w, v$ small, because $\nabla_x
\mc{L}_0(x_0, \lambda) = 0$ for $\lambda \in \Lambda_0$.
For $\lambda \in \Lambda_0$, define the sets $I_0(\lambda) \defeq \{i
\in [\numactive] \mid \lambda_i = 0\}$ and $I_+(\lambda) \defeq \{i \in
[\numactive] \mid \lambda_i > 0\}$, and consider the
tangent cone
\begin{equation*}
  \overline{\tangentset}
  \defeq \bigcup_{\lambda \in \Lambda_0}
  \left\{w : w^T \nabla f_i(x_0) = 0
  ~ \mbox{for}~ i \in I_+(\lambda),
  ~ w^T \nabla f_i(x_0) \le 0
  ~ \mbox{for~} i \in I_0(\lambda) \right\}.
\end{equation*}
The minimizers of the quadratic function~\eqref{eqn:quad-approx-lagrange}
over $\overline{\tangentset}$ approximate those of the tilted
problem~\eqref{eqn:tilted-problem} as follows~\cite[Theorem
  2G.8]{DontchevRo14}: if for $v$ near 0 the function $\zeta_v(w)$ has a
unique minimizer $w_v$ over $\overline{\tangentset}$, then
\begin{equation}
  \label{eqn:gateaux-derive-tilt}
  \lim_{t \downarrow 0} \frac{x_{tv} - x_0}{t} = w_v.
\end{equation}
Moreover~\cite[Thm.~2G.8 and Def.~2.4 (semiderivative)]{DontchevRo14}, if
$w_v$ is linear in $v$, then $v \mapsto x_v$ is differentiable at $v = 0$
with $x_v = x_0 + w_v + o(\norm{v})$.  We consider the two cases of
Assumption~\ref{assumption:linear-independence-cq} to give the result.

\paragraph{Case I: Linearly independent constraints}
As noted following Assumption~\ref{assumption:linear-independence-cq},
the set $\Lambda_0 = \{\lambda\opt\}$, a singleton. Thus
$\overline{\tangentset} = \tangentset$ and following the quadratic
expansion~\eqref{eqn:quad-approx-lagrange}, we solve
\begin{equation*}
  \minimize_w ~ \half w^T H\opt w - v^T w
  ~~ \subjectto ~ w^T \nabla f_i(x_0) = 0,
  ~ i \in [\numactive].
\end{equation*}
This quadratic problem has solution $\projmat[\tangentset] {H\opt}^\dag
\projmat[\tangentset] v$, which is unique by
Assumption~\ref{assumption:restricted-strong-convexity}.
Expression~\eqref{eqn:gateaux-derive-tilt} gives the proposition in this
case.

\paragraph{Case II: Affine constraints}

In Assumption~\ref{assumption:linear-independence-cq}.ii,
the active $f_i$ are affine.  We claim that, though
$\Lambda_0$ may not be a singleton,
\begin{equation*}
  \overline{\tangentset} = \{w \mid Aw = 0\}.
\end{equation*}
To see this, let $u = -\nabla f_0(x_0)$, whence we know that $u = A^T
\lambda\opt = 0$ for some $\lambda\opt > 0$ by
Assumption~\ref{assumption:linear-independence-cq}.ii. Writing $A = [a_1 ~
  \cdots ~ a_{\numactive}]^T$, we see that for any $\lambda \in \Lambda_0$,
if we have
\begin{equation*}
  w \in \left\{w \mid a_i^T w = 0 ~ \mbox{for}~ i ~ \mbox{s.t.}~ \lambda_i > 0,
  a_i^T w \le 0 ~ \mbox{otherwise} \right\},
\end{equation*}
then $u = A^T\lambda$ and
$u^T w = \lambda^T A w = \sum_{i = 1}^{\numactive} \lambda_i a_i^T w = 0$,
because $a_i^T w = 0$ whenever $\lambda_i \neq 0$. But of course,
we know that $A^T \lambda\opt = u$, so that
\begin{equation*}
  0 = w^T u = w^T A^T \lambda\opt
  = \sum_{i = 1}^\numactive \lambda_i\opt a_i^T w
\end{equation*}
so each index $i$ satisfies $a_i^T w = 0$ as $a_i^T w \le 0$ and
$\lambda_i\opt > 0$.  The simplification $\wb{\tangentset} =
\tangentset$ as in Case I applies; the remainder of the proof
is identical.


\section{Proof of Theorem~\ref{theorem:local-minimax-lower}: 
  local minimax lower bounds}
\label{sec:proof-local-minimax-lower}

We briefly outline the approach. We divide the proof into two
parts: an analytic part studying properties of the perturbed solutions $x_u$
(Sec.~\ref{sec:perturbation-optimal-solution}), and a stochastic part
applying Le Cam's local asymptotic normality theory (Sec.~\ref{sec:LAN}). In
the first part, we investigate the perturbation properties of the solutions
$x_u$ as $u \to 0$ via the implicit function result of
Proposition~\ref{proposition:perturbation}.  We show that our
choice~\eqref{eqn:tilterriffic} of $P_u$ gives $f_u(x)
\approx f_0(x) + u^T \Sigma_{\pertfunc}(x - x_0)$ for an appropriate
$\Sigma_{\pertfunc}$, so that $x_u = x_0 + D u +
o(\norm{u})$ as $u \to 0$ for a matrix $D$ by
Proposition~\ref{proposition:perturbation}. This allows application of Le
Cam's local asymptotic normality
theory~\cite{LeCamYa00, VanDerVaartWe96, VanDerVaart98}; heuristically, we may
place a Gaussian prior on $u$ concentrated at rate $1/k$, so that
minimization in the problem~\eqref{eqn:stochastic-u-opt-problem} indexed by
$u$ is asymptotically equivalent to estimating the Gaussian shift $D u$. By
our construction of the tilting~\eqref{eqn:tilterriffic}, the vector $u$ is
asymptotically normally distributed (we make this precise in
Section~\ref{sec:LAN}), which allows us to apply standard normality
optimality guarantees. We unify our arguments in
Section~\ref{sec:finalize-local-minimax}.

\subsection{Perturbation of Optimal Solutions}
\label{sec:perturbation-optimal-solution}

We first consider optimal solutions to the problem
$\program_u$ defined in Eq.~\eqref{eqn:stochastic-u-opt-problem}.  We begin
with a lemma that describes the perturbation of $f_u$ from $f_0$.
\begin{lemma}
  \label{lemma:continuity}
  Let the conditions of Theorem~\ref{theorem:local-minimax-lower} hold.
  Then $(x, u) \mapsto f_u(x)$ is $\mc{C}^2$ near $u = 0$ and $x = x_0$, and
  \begin{equation*}
    f_u(x) = f_0(x) + u^T \covmat_{\pertfunc, \loss} (x - x_0)
    + c_u + o(\norm{x - x_0}^2 + \norm{u}^2),
  \end{equation*}
  where $\covmat_{\pertfunc, \loss} \defeq
  \E[\pertfunc(\statrv) (\nabla \loss(x_0;
    \statrv) - \grad f(x_0))^T]$ and $c_u$ depends only on $u$.
\end{lemma}
\noindent
The lemma consists of a number of applications of Lebesgue's dominated
convergence theorem; we defer proof to
Supplement Sec.~\ref{section:proof-of-lemma-continuity}.

Evidently, Proposition~\ref{proposition:perturbation} applies to the
minimizers $x_u$, as the problem $\program_u$ is asymptotically equivalent
to a linear tilt, exactly as in Eq.~\eqref{eqn:tilted-problem}.  Thus, it is
immediate that the minimizers $x_u$ of $f_u(x) = \int \loss(x; \statval)
dP_u(\statval)$ over $\xdomain$ satisfy
\begin{align}
  \sqrt{k} (x_{u/\sqrt{k}} - x_0)
  & \mathop{\rightarrow}_{k \uparrow \infty}
  -\projmat[\tangentset] \left(\nabla^2 f_0(x_0) + 
  \sum_{i=1}^\numactive \lambda\opt_i \nabla^2 f_i(x_0) \right)^\dag
  \projmat[\tangentset] \covmat_{\pertfunc, \loss}^T u,
  \label{eqn:asymptotic-u-solution}
\end{align}
where we recall that $\projmat[\tangentset]$ denotes projection onto the
tangent set~\eqref{eqn:critical-tangent} and $\lambda\opt$ are
optimal Lagrange multipliers for problem~\eqref{eqn:problem}.

\subsection{Local asymptotic normality}
\label{sec:LAN}

The tilts $P_u$ are a locally asymptotically
normal~\cite{LeCamYa00,VanDerVaartWe96} family of distributions indexed $u
\in \R^d$, which, when coupled with the differentiability
result~\eqref{eqn:asymptotic-u-solution}, allows us to apply the H\'ajek-Le
Cam local minimax theory.
We first recall definitions due to Le
Cam~\cite{LeCamYa00} that we use to develop our problems with
asymptotically Gaussian structure.
\begin{definition}
  \label{def:asymptotic-normality}
  Let $U \subset \R^d$ be an open set containing $0$.
  For each $k \in \N$ and $u \in U$, let $P_{k,u}$ be a probability
  measure on a measurable space $(\statdomain_k, \mc{F}_k)$,
  and let $\statrv^k$ be a sample from $P_{k,u}$.
  The sequence $\{\statdomain_k, \mc{F}_k, P_{k,u}\}_{u \in U}$ is
  \emph{locally asymptotically normal with precision $K \succeq 0$ (LAN)} if
  \begin{equation*}
    \log \frac{dP_{k,u}(\statrv^k)}{dP_{k,0}(\statrv^k)}
    = \<u, Z_k\> - \frac{1}{2} u^T K u
    + o_{P_0}(1)
  \end{equation*}
  where $Z_k \cd \normal(0, K)$ under the distribution $P_0$.
\end{definition}
\noindent
A second important definition is the regular
estimand~\cite{VanDerVaartWe96, VanDerVaart98}.
\begin{definition}
  \label{def:regular-parameter}
  Let $U \subset \R^d$ be a neighborhood of $0$ and
  $\kappa_k : U \to \R^{n}$.
  The sequence $\{\kappa_k\}_{k \in \N}$ is \emph{regular with derivative $D
    \in \R^{n \times d}$} if
  \begin{equation*}
    \sqrt{k} (\kappa_k(u) - \kappa_k(0)) \to D u
    ~~ \mbox{for~all~} u \in U.
  \end{equation*}
\end{definition}

\newcommand{\range}{\mathop{\rm range}}
\newcommand{\nullspace}{\mathop{\rm null}}

With these definitions, the following local asymptotic minimax result, a
variant of the H\'ajek-Le Cam minimax theorem, holds.
\begin{lemma}[Local minimax theorem, Theorem 3.11.5~\cite{VanDerVaartWe96}
  or Lemma 6.6.1 and Theorem 6.6.2~\cite{LeCamYa00}]
  \label{lemma:local-minimax}
  Let the sequence $\{\statdomain_k, \mc{F}_k, P_{k,u}\}_{u \in U}$ be
  locally asymptotically normal with precision $K$
  (Def.~\ref{def:asymptotic-normality}) and let $\kappa_k : U \to \R^{n'}$
  be regular with derivative $D$ (Def.~\ref{def:regular-parameter}). Let $L
  : \R^{n} \to \R_+$ be symmetric and quasi-convex. Then for any sequence
  $T_k : \statdomain_k \to \R^{n}$ of estimators,
  \begin{equation*}
    \sup_{U_0 \subset U,
      |U_0| < \infty}
    \liminf_{k \to \infty}
    \max_{u \in U_0}
    \E_{P_{u,k}}\left[L\big(\sqrt{k}(T_k(\statrv^k) - \kappa_k(u))\big)
      \right]
    \ge \E[L(Z)],
  \end{equation*}
  where $Z \sim \normal(0, D K^{-1} D^T)$ when $K \succ 0$.
  If $K$ is singular and
  $\range(D^T) \cap \nullspace(K) \neq \emptyset$, the result
  holds for
  $Z \sim \normal(0, D (K + \lambda I)^{-1} D^T)$ for any $\lambda > 0$.
\end{lemma}

\newcommand{\sumvec}{\mathsf{S}}

Eq.~\eqref{eqn:asymptotic-u-solution} shows that $\kappa_k(u)
\defeq\argmin_{x \in \mc{X}} f_{u/\sqrt{k}}(x)$ is regular
(Def.~\ref{def:regular-parameter}): recalling the definition of the Hessian
$H\opt = \nabla^2_x \mc{L}(x\opt, \lambda\opt)$ in the statement of the
theorem, the sequence is regular with derivative $\projmat[\tangentset]
{H\opt}^\dag \projmat[\tangentset] \covmat_{\pertfunc, \loss}^T$.  It remains
to establish the local asymptotic normality properties of $P_u$.
\begin{lemma}
  \label{lemma:local-asymptotic-normality}
  Let $P_u$ be as in expression~\eqref{eqn:tilterriffic}. Let $u \in
  \R^d$ and define $P_k = P^k_{u / \sqrt{k}}$, the $k$-fold product of
  $P_{u/\sqrt{k}}$. Let $\covmat_\pertfunc =
  \E_{P_0}[\pertfunc(\statrv) \pertfunc(\statrv)^T]$. Then
  \begin{equation*}
    \log \frac{dP_k(\statrv_1, \ldots, \statrv_k)}{dP_0(\statrv_1,
      \ldots, \statrv_k)}
    = -\frac{1}{\sqrt{k}} u^T \sum_{i = 1}^k \pertfunc(\statrv_i)
    - \half u^T \covmat_\pertfunc u + o_{P_0}(1).
  \end{equation*}
\end{lemma}
\noindent
See Supplement Sec.~\ref{sec:proof-LAN} for a proof.
In particular, we see that if $\mc{F}_k$ denotes the $\sigma$-algebra
on the product $\statdomain^k$, then
the sequence
\begin{equation*}
  \left\{\statdomain^k, \mc{F}_k, P_{u/\sqrt{k}}^k\right\}_{u \in \R^n}
\end{equation*}
is LAN with precision $\covmat_\pertfunc$ for
$\pertfunc$ with $\E_{P_0}[\pertfunc] = 0$ and $\E_{P_0}[\norm{\pertfunc}^2]
< \infty$.

\subsection{Finalizing the argument}
\label{sec:finalize-local-minimax}

Now that we have the regularity of the sequence $x_{u/\sqrt{k}}$ as
$k \to \infty$ (the convergence guarantee~\eqref{eqn:asymptotic-u-solution})
and the asymptotic normality of Lemma~\ref{lemma:local-asymptotic-normality},
we may apply Lemma~\ref{lemma:local-minimax}. Indeed, let
$P_{u,k} = P_{u/\sqrt{k}}^k$ be the distribution of an i.i.d.\ sample
$\statrv_i \simiid P_{u/\sqrt{k}}$ for $i = 1, \ldots, k$, and let
$\what{x}_k$ be an arbitrary estimator based on
$\statrv_{1:k}$. Lemma~\ref{lemma:local-minimax} implies
\begin{equation*}
  \sup_{U_0 \subset \R^d,
    |U_0| < \infty}
  \liminf_{k \to \infty} \max_{u \in U_0}
  \E_{P_{u/\sqrt{k}}^k}\left[
    L(\sqrt{k}(\what{x}_k - x_{u/\sqrt{k}}))\right]
  \ge \E[L(Z_\lambda)]
\end{equation*}
for any $\lambda > 0$, where
\begin{equation*}
  Z_\lambda \sim \normal\left(0, \projmat[\tangentset]
    {H\opt}^\dag \projmat[\tangentset] \covmat_{\pertfunc, \loss}^T
    (\covmat_\pertfunc + \lambda I)^{-1} \covmat_{\pertfunc, \loss}
    \projmat[\tangentset]
    {H\opt}^\dag \projmat[\tangentset] \right).
\end{equation*}
The theorem follows by taking $\lambda \downarrow 0$, noting that
for any two mean-zero random vectors $Z$ and $Y$,
we have
\begin{equation}
  \label{eqn:semidefinite-order-rvs}
  \E[YZ^T] \E[ZZ^T]^\dag \E[ZY^T]
  \preceq \E[YY^T],
\end{equation}
and that (by Anderson's lemma~\cite[Lemma 8.5]{VanDerVaart98}) if $\Sigma_1
\preceq \Sigma_2$ and $Z_i \sim \normal(0, \Sigma_i)$,
then $\E[L(Z_1)] \le \E[L(Z_2)]$.  To see
inequality~\eqref{eqn:semidefinite-order-rvs}, we may without loss of
generality assume that $\E[ZZ^T] \preceq I$, as by letting $\Sigma =
\E[ZZ^T]^\dag$, we have $\E[\Sigma^{1/2} ZZ^T \Sigma^{1/2}] = \Sigma^{1/2}
\Sigma^\dag \Sigma^{1/2} \preceq I$; to show
inequality~\eqref{eqn:semidefinite-order-rvs}, it is thus equivalent to show
that $\E[YZ^T] \E[ZY^T] \preceq I$ for all $Z$ such that $\E[ZZ^T] \preceq
I$.  To see this, let $v$ be arbitrary, and note
that by Cauchy-Schwarz we have
\begin{equation*}
  \ltwo{\E[v^T YZ]}^2
  = \sup_{\norm{u} \le 1}
  \E[v^T Y Z^T u]^2
  \le \E[(v^T Y)^2]
  \sup_{\norm{u} \le 1}
  \E[(u^T Z)^2]
  \le v^T \E[YY^T] v.
\end{equation*}


\section{Proofs of convergence for dual averaging}

Here we collect the major arguments for our proofs of the almost sure
convergence and finite time constraint identification for our
variant~\eqref{eqn:variant-dual-averaging} of dual averaging. We highlight
new results and techniques, deferring technical details.

\subsection{Proof of Theorem~\ref{theorem:as-convergence}:
  almost sure convergence}
\label{sec:proof-as-convergence}

First, we establish a few technical properties of the stepsize sequence.  We
begin with the following lemma, whose proof is immediate
when $\stepsize_k \propto k^{-\steppow}$ for $\steppow \in (\half, 1)$.
\begin{lemma}
  \label{lemma:big-diverge}
  For $\stepsize_k$ satisfying condition~\eqref{eqn:stepsizes},
  $\sum_{k=1}^\infty \frac{\stepsize_k}{\sum_{i=1}^k \stepsize_i} = \infty$.
\end{lemma}

Now we state a classical result that is useful
for showing the almost convergence of stochastic approximation algorithms.
\begin{lemma}[Robbins and Siegmund~\cite{RobbinsSi71}]
  \label{lemma:robbins-siegmund}
  Let $V_k, A_k, B_k, C_k$ be non-negative random variables adapted to
  a filtration $\mc{F}_k$. Assume that
  \begin{equation*}
    \E[V_{k + 1} \mid \mc{F}_k] \le (1 + A_k) V_k + B_k - C_k.
  \end{equation*}
  Then on the event $\{\sum_k A_k < \infty, \sum_k B_k < \infty\}$, there is
  a random variable $V_\infty < \infty$ such that $V_k \cas V_\infty$ and
  $\sum_k C_k < \infty$ a.s.
\end{lemma}

We use Lemma~\ref{lemma:robbins-siegmund} 
to show that the quantity
\begin{equation}
  \label{eqn:funny-remainder}
  R_k \defeq \<z_k + x_{k + 1}, x\opt - x_{k + 1}\>
  + \half \ltwo{x_{k + 1} - x\opt}^2
\end{equation}
converges a.s.\ to some random variable $R_\infty < \infty$, where
$z_k \defeq \sum_{i=1}^k \stepsize_i g_i$.  We can decompose $R_k$
as the sum of two nonnegative random variables,
\begin{equation*}
  R_k = G_k + V_k,
  ~~
  G_k = \<z_k + x_{k+1}, x\opt - x_{k+1}\> \ge 0
  ~~ \mbox{and} ~~
  V_k = \half \ltwo{x_{k + 1} - x\opt}^2.
\end{equation*}
Here we have $G_k \ge 0$ because $x_{k + 1}$ minimizes
$\<z_k, x\> + \half \ltwo{x}^2$ over $x \in \xdomain$, so that
$\<z_k + x_{k + 1}, y - x_{k + 1}\> \ge 0$ for all $y \in \xdomain$ (and
$x\opt \in \xdomain$ by definition), while $V_k \ge 0$ clearly.
Recall the definition (Assumption~\ref{assumption:noise-vectors})
of the filtration
\begin{equation*}
  \mc{F}_k \defeq \sigma(\noise_1, \ldots, \noise_k)
\end{equation*}
as the $\sigma$-field generated by the noise sequence through time $k$.
Then
we have the measurability $R_k, G_k, V_k \in \mc{F}_k$
and the following convergence.

\begin{lemma}
  \label{lemma:pair-convergence}
  Let $R_k$ be as in~\eqref{eqn:funny-remainder} and assume that
  $\sum_k \stepsize_k^2 < \infty$. Then for some finite random variable 
  $R_\infty$, we have $R_k \cas R_\infty$. Moreover,
  \begin{equation*}
    \sum_{i=1}^\infty \stepsize_i \left[f(x_i) - f(x^*)\right] < \infty
    ~~~ \mbox{with~probability}~ 1.
  \end{equation*}
\end{lemma}
\begin{proof}
  Let $\prox(x) = \half \ltwo{x}^2 + \setindic{x}{\xdomain}$ and define its
  conjugate $\prox^*(z) = \sup_{x \in \xdomain} \{\<z, x\> - \half
  \ltwo{x}^2\}$.  Then $\prox^*$ has $1$-Lipschitz continuous gradient with
  $\nabla \prox^*(z) = \argmax_{x \in \xdomain} \{\<z, x\> - \half
  \ltwo{x}^2\}$~\cite[Chapter X]{HiriartUrrutyLe93ab}, and
  \begin{align*}
    R_k 
    & = \<z_k, x\opt - x_{k + 1}\> + \half \ltwo{x\opt}^2
    - \half \ltwo{x_{k+1}}^2
    = \<z_k, x\opt\> + \half \ltwo{x\opt}^2 + \prox^*(-z_k).
  \end{align*}
  Using $\nabla \prox^*(-z_{k-1}) = x_k$ and the Lipschitz
  continuity of $\nabla \prox^*$, we have
  \begin{align*}
    \prox^*(-z_k)
    & \le \prox^*(-z_{k-1}) + \<\nabla \prox^*(-z_{k-1}), z_{k - 1} - z_k\>
      + \half \ltwo{z_k - z_{k - 1}}^2 \\
    & = \prox^*(-z_{k-1}) - \stepsize_k \<g_k, x_k\>
      + \frac{\stepsize_k^2}{2} \ltwo{g_k}^2.
  \end{align*}
  That is, we have for any $k$ that
  \begin{align*}
    R_k & \le \<z_k, x\opt\> + \half \ltwo{x\opt}^2
          + \prox^*(-z_{k - 1}) - \stepsize_k \<g_k, x_k\> + \frac{\stepsize_k^2}{2}
          \ltwo{g_k}^2 \\
        & = \underbrace{\<z_{k - 1} + x_k, x\opt - x_k\>
          + \half \ltwo{x_k - x\opt}^2}_{= R_{k-1}}
          - \stepsize_k \<g_k, x_k - x\opt\>
          + \frac{\stepsize_k^2}{2} \ltwo{g_k}^2.
  \end{align*}
  Taking conditional expectations and using that $\E[g_k \mid \mc{F}_{k-1}]
  = \nabla f(x_k)$ yields
  \begin{align*}
    \E[R_k \mid \mc{F}_{k-1}]
    & \le R_{k-1} - \stepsize_k \<\nabla f(x_k), x_k - x\opt\>
      + \frac{\stepsize_k^2}{2} \E[\ltwo{g_k}^2 \mid \mc{F}_{k-1}] \\
    & \stackrel{(i)}{\le} G_{k-1} + V_{k-1} - \stepsize_k \<\nabla f(x_k), x_k - x\opt\>
      + \frac{\stepsize_k^2}{2} (C \ltwo{x_k - x\opt}^2 + C) \\
    & \stackrel{(ii)}{\le}
      (1 + C \stepsize_k^2) \left[G_{k-1} + V_{k-1}\right]
      - \stepsize_k \<\nabla f(x_k), x_k - x\opt\>
      + C \stepsize_k^2
  \end{align*}
  where inequality $(i)$ follows by Assumption~\ref{assumption:noise-vectors}
  and the discussion (Eq.~\eqref{eqn:noise-level-gradient-norms}) immediately
  following Assumption~\ref{assumption:noise-level}, and inequality $(ii)$
  because $G_{k-1} \ge 0$ and $V_{k - 1} = \half \ltwo{x_k - x\opt}^2$.  In
  particular, we have
  \begin{equation*}
    \E[R_k \mid \mc{F}_{k-1}]
    \le (1 + C \stepsize_k^2) R_{k-1} - \stepsize_k \<\nabla f(x_k),
    x_k - x\opt\> + C \stepsize_k^2.
  \end{equation*}
  Because $f(x\opt) \ge f(x_k) + \<\nabla f(x_k), x\opt - x_k\>$,
  or $\<\nabla f(x_k), x_k - x\opt\> \ge f(x_k) - f(x\opt) \ge 0$,
  Lemma~\ref{lemma:robbins-siegmund} applies. Thus we must have
  $R_k \cas R_\infty$ for some finite random variable $R_\infty$, and 
  moreover 
   \begin{equation*}
    \sum_{i=1}^\infty \stepsize_i \left[f(x_i) - f(x^*)\right] \leq 
    \sum_{i=1}^\infty \stepsize_i \left<\nabla f(x_i), x_i - x^*\right> < \infty,
  \end{equation*}
 where we have used the standard first-order convexity inequality.
\end{proof}

%

With these lemmas as background, we finally provide the proof of
Theorem~\ref{theorem:as-convergence}, by showing that with $R_k$ defined as in
expression~\eqref{eqn:funny-remainder},
\begin{equation}
  \label{eqn:everyone-as-converges}
  R_k \cas 0
  ~~\mbox{so that} ~~
  x_k \cas x\opt.
\end{equation}
We introduce a bit of notation.  Let $A_k
= \sum_{i=1}^k \stepsize_i$, and recall that $z_k = \sum_{i=1}^k \stepsize_i
g_i$. Define $\bar{z}_k = \sum_{i=1}^k \stepsize_i \nabla f(x_i)$ to
be the weighted partial sum of the (non-noisy) gradients $\nabla f(x_i)$,
and we let $z\opt_k = A_k \nabla f(x\opt)$.

We first claim that the error sequence is asymptotically
negligible:
\begin{equation}
  \label{eqn:errors-negligible}
  \frac{1}{\sqrt{A_k}}\sum_{i=1}^k \stepsize_i \noise_i  \cas 0.
\end{equation}
To see the claim~\eqref{eqn:errors-negligible},
we use the following lemma.
\begin{lemma}[Dembo~\cite{Dembo16}, Exercise 5.3.35]
  \label{lemma:dembo-l2}
  Let $Z_k \in \R^n$ be a martingale adapted to $\mc{F}_k$
  and let $b_k > 0$ be a non-random sequence increasing to
  $\infty$. If
  $\sum_{k = 1}^\infty b_k^{-2} \E[\norm{Z_k - Z_{k-1}}^2 \mid \mc{F}_{k-1}]
  < \infty$, we have $b_k^{-1} Z_k \cas 0$.
\end{lemma}
\noindent
Since $\{\sum_{i=1}^k \stepsize_i \noise_i\}_{k=1}^\infty$ is a martingale 
difference sequence, Lemma~\ref{lemma:dembo-l2} shows that
to obtain the claim~\eqref{eqn:errors-negligible}
it is sufficient to show that
\begin{equation*}
  \sum_{k=1}^\infty
  \frac{1}{A_k} \E\left[\norm{\stepsize_k \noise_k}^2 \mid \mc{F}_{k-1}\right] < \infty. 
\end{equation*}
By Assumption~\ref{assumption:noise-vectors}, the left side of the preceding
display has upper bound $\frac{C}{A_1} \sum_{i=1}^\infty \stepsize_i^2
(1+\norm{x_i - x\opt}^2)$, so that showing $\sum_{i=1}^\infty
\stepsize_i^2(1+ \norm{x_i - x\opt}^2) < \infty$ proves the
claim~\eqref{eqn:errors-negligible}. With that in mind, recall
Lemma~\ref{lemma:restricted-strong-convexity}, which guarantees
an $\epsilon > 0$ such that $f(x) - f(x\opt) \ge
\epsilon (\norm{x - x\opt}^2 \wedge \norm{x - x\opt})$. Using
Lemma~\ref{lemma:pair-convergence}, we know that 
$M \defeq \sup_i \norm{x_i -  x\opt} \vee 1 < \infty$. Thus we have
\begin{equation*}
  f(x_i) - f(x\opt) \ge
  \epsilon \min\left\{
  \norm{x_i - x\opt} M / M,
  \norm{x_i - x\opt}^2\right\}
  \ge c \norm{x_i - x\opt}^2
\end{equation*}
where $c > 0$ is a random positive constant that depends on the bound
$M$. Combining this result with Lemma~\ref{lemma:pair-convergence}, we have
\begin{lemma}
  \label{lemma:sum-stepsize-square-distance}
  Under the conditions of Theorem~\ref{theorem:as-convergence}, we have
  \begin{equation*}
    \sum_{i = 1}^\infty \stepsize_i \norm{x_i - x\opt}^2
    \le \frac{1}{c} \sum_{i = 1}^\infty c \stepsize_i \norm{x_i - x\opt}^2
    \le \frac{1}{c} \sum_{i = 1}^\infty \stepsize_i [f(x_i) - f(x\opt)]
    < \infty.
  \end{equation*}
\end{lemma}
\noindent
Here the final inequality follows from Lemma~\ref{lemma:pair-convergence}.  Noting
that $\sum_{i=1}^\infty \stepsize_i^2 < \infty$ by
assumption~\eqref{eqn:stepsizes}, we obtain the
claim~\eqref{eqn:errors-negligible}. Moreover, this implies that
\begin{equation}
  \label{eqn:z-and-bar-z-converge}
  \frac{z_k - \bar{z}_k}{\sqrt{A_k}}
  = \frac{1}{\sqrt{A_k}} \sum_{i = 1}^k \stepsize_i \noise_i
  \cas 0.
\end{equation}

Now that we have the convergence guarantee~\eqref{eqn:z-and-bar-z-converge},
that $R_k \cas R_\infty < \infty$ (Lemma~\ref{lemma:pair-convergence}),
and that $\sum_{i = 1}^\infty \stepsize_i \norm{x_i - x\opt}^2 <\infty$ with
probability 1, we define
\begin{equation*}
  \Omega_0 \defeq \left\{\sum_{i=1}^\infty 
    \stepsize_i \norm{x_i - x\opt}^2 < \infty,
    R_k \to R_\infty < \infty, 
    \frac{z_k - \bar{z}_k}{\sqrt{A_k}} \rightarrow 0 \right\},
    ~~~ \P(\Omega_0) = 1.
\end{equation*}
On the set $\Omega_0$, using the Lipschitz
continuity Assumption~\ref{assumption:smoothness-of-objective}, we
may define
\begin{equation*}
  \sigma_\infty^2 \defeq \sum_{i=1}^\infty \stepsize_i \norm{\nabla f(x_i) - \nabla f(x^*)}^2 \leq 
  L^2 \sum_{i=1}^\infty \stepsize_i \norm{x_i - x^*}^2  < \infty. 
\end{equation*}
Then using Jenson's inequality and recalling
the definition $z\opt_k = A_k \nabla f(x\opt)$,
\begin{align*}
  \norm{\bar{z}_k - z\opt_k}^2
  =
  \normbigg{
  \sum_{i=1}^k \stepsize_i (\nabla f(x_i) - \nabla f(x\opt))}^2
  \leq A_k \sigma_\infty^2.
\end{align*}
Hence, we have $\norm{\bar{z}_k - z\opt_k} \leq \sqrt{A_k} \sigma_\infty$. 
Now, we see that on $\Omega_0$,
\begin{equation*}
  \infty > \sum_{i = 1}^\infty \stepsize_i \norm{x_i - x\opt}^2
  = \sum_{i = 1}^\infty \frac{\stepsize_i}{A_i} A_i \norm{x_i - x\opt}^2.
\end{equation*}
Lemma~\ref{lemma:big-diverge} (that $\sum_i \frac{\stepsize_i}{A_i} = \infty$)
implies there exists a subsequence $\{k_i\}$ with
\begin{equation*}
  \lim_{i\rightarrow \infty} A_{k_i}\norm{x_{k_i} - x\opt}^2 = 0,
\end{equation*}
and moreover,
\begin{equation*}
  \norm{\bar{z}_{k_i} - z\opt_{k_i}} \norm{x_{k_i} - x\opt}
  \le \sigma_\infty \sqrt{A_{k_i}} \norm{x_{k_i} - x\opt}
  \to 0.
\end{equation*}

Keep the subsequence $\{k_i\}$ fixed, and note that on
$\Omega_0$, we have that $R_{k_i - 1} \to R_\infty$. Let us expand
the terms in the definition of $R_k$ to see that we must have $R_\infty = 0$.
Indeed, we have
\begin{align*}
  R_{k-1} 
  & = \<z_{k - 1} + x\opt, x\opt - x_k\> - \half \norm{x_k - x\opt}^2 \\
  & \le \<z_{k - 1} + x\opt, x\opt - x_k\> \\
  & = \<z_{k - 1} - z\opt_{k-1}, x\opt - x_k\>
  + \<z\opt_{k-1}, x\opt - x_k\> + \<x\opt, x\opt - x_k\>
  \\
  & \le \norm{z_{k - 1} - z\opt_{k-1}} \norm{x\opt - x_k}
  + A_k \<\nabla f(x\opt), x\opt - x_k\>
  + \norm{x\opt} \norm{x\opt - x_k}.
\end{align*}
The optimality conditions for 
$x\opt$ imply $\<\nabla f(x\opt), x\opt - x_k\> \le 0$.
On the subequence $k_i$, we have
\begin{equation*}
  \limsup_{i \to \infty}
  \norm{z_{k_i - 1} - z\opt_{k_i - 1}} \norm{x\opt - x_{k_i}}
  \le \limsup_{i \to \infty}
  \sigma_\infty \sqrt{A_{k_i-1}} \norm{x\opt - x_{k_i}} = 0
\end{equation*}
and
$\limsup_{i \to \infty} \norm{x\opt} \norm{x\opt - x_{k_i}} = 0$.
In particular, that $R_k \ge 0$ implies
\begin{align*}
  0 \le \liminf_{i \to \infty} R_{k_i - 1}
  \le \limsup_{i\rightarrow \infty} R_{k_i - 1}
  = 0.
\end{align*}
Because $R_k \to R_\infty$ on $\Omega_0$, it must thus be the case
that $R_\infty = 0$.

\subsection{Manifold identification:
  Theorem~\ref{theorem:manifold-identification}}

\label{sec:proof-manifold-identification}

Recall that $z_k = \sum_{i = 1}^k \stepsize_i g_i$ is the weighted partial sum
of the noisy gradients, and let $A_k = \sum_{i=1}^k \stepsize_i$.
The following lemma is a nearly immediate
consequence of our previous results and, given the perturbation results in
Lemmas~\ref{lemma:perturbation-result-nonlinear}
and~\ref{lemma:perturbation-result-linear}, is the key to our finite
identification result.
\begin{lemma}
  \label{lemma:average-zk-converges}
  Under the conditions of the Theorem~\ref{theorem:manifold-identification},
  $\frac{1}{A_k} z_k \cas \nabla f(x\opt)$.
\end{lemma}
\begin{proof}
  We first remove the randomness of $\noise_i$. By Jenson's
  inequality,
  \begin{align*}
    \norm{\frac{z_k}{A_k} - \nabla f(x\opt)}^2
    & \le 2 
      \normbigg{A_k^{-1} \sum_{i=1}^k \stepsize_i (\nabla f(x_i) - \nabla f(x\opt))}^2
      + 2 \normbigg{A_k^{-1}\sum_{i = 1}^k \stepsize_i \noise_i}^2.
  \end{align*}
  The second term converges almost surely to zero by the almost sure
  convergence~\eqref{eqn:errors-negligible} in the proof of
  Theorem~\ref{theorem:as-convergence}. We thus focus on the first term.

  By Lemma~\ref{lemma:sum-stepsize-square-distance} in the proof
  of Theorem~\ref{theorem:as-convergence}
  and the Lipschitz Assumption~\ref{assumption:smoothness-of-objective},
  we know that $\sum_{i = 1}^\infty \stepsize_i \norm{\nabla
    f(x_i) - \nabla f(x\opt)}^2
  \le C \sum_{i = 1}^\infty \stepsize_i \norm{x_i - x\opt}^2
  < \infty$ with probability 1. Thus, by Jenson's inequality,
  \begin{align*}
    \frac{1}{A_k^2} 
      \normbigg{\sum_{i=1}^k \stepsize_i (\nabla f(x_i) - \nabla f(x\opt))}^2
    &\leq \frac{1}{A_k} \sum_{i=1}^\infty \stepsize_i \norm{\nabla f(x_i) - \nabla f(x\opt)}^2.
  \end{align*}
  Taking $A_k \to \infty$ gives the result.
\end{proof}

Applying Assumption~\ref{assumption:linear-independence-cq},
there exist $\lambda_i > 0$ and $\nu_i = 0$ such that
$\nabla f(x\opt) + \sum_{i = 1}^\numactive \lambda_i \nabla f_i(x\opt)
+ \sum_{i = \numactive + 1}^m \nu_i \nabla f_i(x\opt) = 0$.
Applying the standard KKT conditions, we immediately
see that $x\opt$ is an optimum of the convex problem
\begin{align*}
  \minimize_{x}~ &~ \<\nabla f(x\opt), x\>
  ~~ \subjectto
  f_i(x)\leq 0~~ \mbox{for~} i \in [m].
\end{align*}
The dual averaging update~\eqref{eqn:variant-dual-averaging}
chooses $x_{k + 1}$ via
\begin{equation*}
  x_{k + 1} = \argmin_x \left\{ \<\nabla f(x\opt), x\>
  + \<v_k, x\>
  + \frac{1}{2 A_k} \norm{x}^2
  \mid f_i(x) \le 0, i \in [m] \right\},
\end{equation*}
where $v_k = \frac{z_k}{A_k} - \nabla f(x\opt)$.
Theorem~\ref{theorem:as-convergence} guarantees that $x_k
\to x\opt$, while Lemma~\ref{lemma:average-zk-converges} shows that $A_k^{-1}
z_k - \nabla f(x\opt) \to 0$ with probability $1$. The perturbation
results (Lemmas~\ref{lemma:perturbation-result-nonlinear} and
\ref{lemma:perturbation-result-linear}) immediately yield the theorem.


\section{Discussion}

We have developed asymptotic theory for stochastic optimization
problems, showing a local asymptotic minimax lower bound and making precise
connections between tilt stability in optimization and the (statistical)
difficulty of solving risk minimization problems.  These optimal rates of
convergence are achievable by the classical M-estimator $\what{x}_k =
\argmin_{x \in \xdomain} \frac{1}{k} \sum_{i = 1}^k F(x; \statrv_i)$
(Corollary~\ref{corollary:shapiro-normality}) and approximate versions
thereof---e.g., from modern incremental gradient
methods~\cite{LeRouxScBa12, JohnsonZh13, DefazioBaLa14, LinMaHa15}. Our dual
averaging (lazy projected gradient) and Riemannian stochastic gradient
methods are also asymptotically optimal, though subtleties arise for
nonlinear constraint sets. There are open questions about whether
simpler methods---for example, methods that do not explicitly track the
active manifold---can achieve these rates, and developing finite
sample analogues of this theory remains an open question.



\ifdefined\useaos
\else
\setlength{\bibsep}{0.3em}
\fi
  
\bibliographystyle{abbrvnat}
\bibliography{bib}

\newpage

\ifdefined\useaos
\else
\appendix
\fi


\section{Local Asymptotic Normality proofs}
\label{appendix:proofs-asymptotic-normality}

In this appendix, we collect the proofs of the additional technical results
necessary for the proof of Theorem~\ref{theorem:local-minimax-lower}.

\subsection{Proof of lemma~\ref{lemma:continuity}}
\label{section:proof-of-lemma-continuity}

Recall the definitions $\pertfunc \in \pertfuncset$ of $\pertfunc :
\statdomain \to \R^d$ with $\E_{P_0}[\pertfunc(\statrv)] = 0$ and
$\E_{P_0}[\norms{\pertfunc(\statrv)}^2] < \infty$ and the $\mc{C}^3$
function $\nearlinear : \R \to [-1, 1]$ with $\nearlinear(t) = t$ for $t \in
[-\half, \half]$. Letting $C(u) = \int (1 + \nearlinear(u^T \pertfunc))
dP_0$ be the normalization constant for $P_u$, we define
\begin{equation*}
  G(x, u) \defeq \frac{1}{C(u)} \int \loss(x;\statval)
  (1 + \nearlinear(u^T \pertfunc(\statval))) dP_0(\statval)
  ~~ \mbox{and} ~~
  \wb{G}(x, u) \defeq C(u) G(x, u)
\end{equation*}
for notational convenience.
We first show that both $C(u)$ and the un-normalized function
$\wb{G}(x, u)$ are
$\mc{C}^2$ in a neighborhood of $(x_0, 0)$.

In this case, we have $C(u) = 1 + \int \nearlinear(u^T \pertfunc(\statval))
dP_0(\statval)$, and a standard application~\cite[Thm.~16.8]{Billingsley86}
of the dominated convergence theorem, coupled with the assumption that $g$
is $\mc{C}^3$ with $\nearlinear' \ge 0$, guarantees
\begin{equation}
  \label{eqn:C-u-derivs}
  \nabla_u C(u) = \E_{P_0}[\nearlinear'(u^T \pertfunc) \pertfunc]
  ~~~ \mbox{and} ~~~
  \nabla^2_u C(u) = \E_{P_0}[
    \nearlinear''(u^T \pertfunc) \pertfunc \pertfunc^T]
\end{equation}
both of which are continuous in $u$ because $\nearlinear'''$ is bounded by
assumption, whence $C(u) = 1 + o(\norm{u}^2)$ as $\nabla_u C(0) = 0$ and
$\nabla^2_u C(0) = 0$.  Now we consider the un-normalized function
$\wb{G}$. We have
\begin{align*}
  \nabla_x \wb{G}(x, u)
  & = \int (1 + \nearlinear(u^T \pertfunc(\statval))) \nabla \loss(x; \statval)
  dP_0(\statval) ~~~ \mbox{and} \\
  \nabla^2_x \wb{G}(x, u)
  & = \int (1 + \nearlinear(u^T \pertfunc(\statval)))
  \nabla^2 \loss(x; \statval) dP_0(\statval),
\end{align*}
again by standard application of the dominated convergence
theorem~\cite[Theorem 16.8]{Billingsley86} because $g$ is bounded in $[-1, 1]$
and we have the remainder guarantee of
Assumption~\ref{assumption:regularity-gradients}.  We may calculate
derivatives with respect to $u$ similarly, obtaining
\begin{align*}
  \nabla_u \wb{G}(x, u)
  & = \int \loss(x; \statval) \nearlinear'(u^T \pertfunc(\statval)) \pertfunc(\statval)
  dP_0(\statval)
  ~~ \mbox{and} \\
  \nabla^2_u \wb{G}(x, u)
  & = \int \loss(x; \statval) \nearlinear''(u^T \pertfunc(\statval)) \pertfunc(\statval) \pertfunc(\statval)^T
  dP_0(\statval)
\end{align*}
in a neighborhood of $(x_0, 0)$, because
$\int |\loss(x; \statval)| \norm{\pertfunc(\statval)}^2 dP_0(\statval) < \infty$ for $x$
near $x_0$ by Assumption~\ref{assumption:regularity-gradients}.
Lastly, a completely similar calculation yields
\begin{equation*}
  \nabla^2_{x,u} \wb{G}(x, u)
  = \int \nearlinear'(u^T \pertfunc(\statval))
  \nabla \loss(x; \statval) \pertfunc(\statval)^T dP_0(\statval).
\end{equation*}

With these equalities in place, we may now compute derivatives. We have that
$\nearlinear''(0) = 0$ and $\nearlinear'''(0) = 0$, because $\nearlinear(t)
= t$ for $-\half \le t \le \half$, so $\nabla C(0) = 0$ and $\nabla^2 C(0) =
0$. Moreover, again using $\nearlinear'(0) = 1$ we have $\nabla^2_{x,u}
\wb{G}(x_0, 0) = \cov_{P_0}(\grad \loss(x_0; \statrv), \pertfunc(\statrv)) =
\covmat_{\pertfunc, \loss}^T$.
We then have
\begin{align*}
  \lefteqn{f_u(x)
    = G(x, u)} \\
  & ~~ = G(x, 0)
  + \nabla_u \wb{G}(x_0, 0)^T u
  + u^T \nabla^2_{u,x} \wb{G}(x_0, 0) (x - x_0)
  + o(\norm{u}^2 + \norm{x - x_0}^2) \\
  & ~~ = f_0(x) + u^T \covmat_{\pertfunc,\loss}(x - x_0) + c_u
  + o(\norm{u}^2 + \norm{x - x_0}^2),
\end{align*}
where we have used that $G(x, 0) = f_0(x)$ and $\nabla^2_u \wb{G}(x_0, 0) = 0$
and that $\norm{u} \norm{v} \le \half \norm{u}^2 + \half \norm{v}^2$
for all $u, v$ and $c_u$ depends only on $u$.

\subsection{Proof of Lemma~\ref{lemma:local-asymptotic-normality}}
\label{sec:proof-LAN}

Let $C_u = 1 + \int \nearlinear(u^T
\pertfunc(\statval)) dP_0(\statval)$.  We begin by expanding the log
likelihood ratio, which gives immediately that
\begin{align*}
  \log \frac{dP_k(\statrv_1, \ldots, \statrv_n)}{dP_0(\statrv_1,
    \ldots, \statrv_n)}
  & = k \log C_{u / \sqrt{k}}
  + \sum_{i = 1}^k \log (1 + \nearlinear( u^T \pertfunc(\statrv_i) / \sqrt{k})).
\end{align*}
By expression~\eqref{eqn:C-u-derivs} for $C_u$, we have $\nabla_u C_0 = 0$
and $\nabla_u^2 C_0 = 0$, so that $C_u = 1 + o(\norm{u}^2)$. Thus
\begin{equation*}
  k \log C_{u/\sqrt{k}}
  = k \log \left(1 + o(\norm{u}^2 / k)\right)
  = k \cdot o(\norm{u}^2 / k) \to 0
\end{equation*}
as $k \to \infty$, so for the remainder of the
proof, we ignore the term $k \log C_{u/\sqrt{k}}$.

Noting that $\E_{P_0}[|u^T \pertfunc(\statrv_i)|^2] \le \ltwo{u}^2 \E_{P_0}[\ltwo{\pertfunc(\statrv_i)}^2]
< \infty$,
it is a standard result~\cite[e.g.][Lemma 3]{Owen90} that
$\max_{1 \le i \le k} u^T \pertfunc(\statrv_i) / \sqrt{k} \cas 0$ under $P_0$.
Thus, we may without loss of generality assume that
$\nearlinear(u^T \pertfunc(\statrv_i) / \sqrt{k}) = u^T \pertfunc(\statrv_i) / \sqrt{k}$, and performing
a Taylor expansion of the logarithm, we obtain
\begin{equation*}
  \log(1 + u^T \pertfunc(\statrv_i) / \sqrt{k})
  = \frac{u^T \pertfunc(\statrv_i)}{\sqrt{k}}
  - \frac{u^T \pertfunc(\statrv_i) \pertfunc(\statrv_i)^T u}{2 k}
  + C_i \frac{|u^T \pertfunc(\statrv_i)|^3}{k^{3/2}}
\end{equation*}
for a some $C_i \in [-1, 1]$, because we may assume that
$|u^T \pertfunc(\statrv_i) / \sqrt{k}| \le \epsilon$ for any $\epsilon > 0$. In particular,
we find that for large enough $k$, we have
\begin{align*}
  \lefteqn{\sum_{i = 1}^k \log\left(1 + \nearlinear\left(\frac{u^T \pertfunc(\statrv_i)}{\sqrt{k}}
    \right)\right)} \\
  & = \frac{1}{\sqrt{k}} \sum_{i = 1}^k u^T \pertfunc(\statrv_i)
  - \frac{1}{2} u^T \bigg(\frac{1}{k} \sum_{i = 1}^k \pertfunc(\statrv_i) \pertfunc(\statrv_i)^T \bigg)
  u
  + \frac{1}{k^{3/2}} \sum_{i = 1}^k C_i |u^T \pertfunc(\statrv_i)|^3.
\end{align*}
By H\"older's inequality, the final term satisfies
\begin{align*}
  \frac{1}{k^{3/2}} \sum_{i = 1}^k C_i |u^T \pertfunc(\statrv_i)|^3
  & \le \frac{1}{k} \sum_{i = 1}^k u^T  \pertfunc(\statrv_i) \pertfunc(\statrv_i)^T u
  \cdot \max_{1 \le i \le k} \frac{|\pertfunc(\statrv_i)^T u|}{\sqrt{k}}
  = o_{P_0}(1).
\end{align*}
Thus, we have
\begin{equation*}
  \log \frac{dP_k(\statrv_1, \ldots, \statrv_k)}{dP_0(\statrv_1,
    \ldots, \statrv_k)}
  = \frac{1}{\sqrt{k}} \sum_{i = 1}^k u^T \pertfunc(\statrv_i)
  - \half u^T \covmat_k u
  + o_{P_0}(1),
\end{equation*}
where $\Sigma_k = \frac{1}{k} \sum_{i = 1}^k \pertfunc(\statrv_i) \pertfunc(\statrv_i)^T \cas \covmat_\pertfunc$, as
desired.


\section{Proofs of linear perturbation results}

In this section, we collect the proofs of our technical results on
perturbation of optimization problems.

\subsection{Proof of Lemma~\ref{lemma:perturbation-result-nonlinear}}
\label{sec:proof-perturbation-nonlinear}

For shorthand, let $x_k = x\opt_{v_k, \delta_k}$, where we have by
assumption that $x_k \to x\opt$.  The continuity of the functions $f_i$ then
implies that for $i \in \{\numactive + 1, \ldots, m\}$, there exists some
$\epsilon > 0$ such that $f_i(x_k) \le -\epsilon$ for all sufficiently large
$k$. Following the definition of $x_k$ as the unique minimizer of the convex
problem~\eqref{eqn:nonlinear-constraints-prob}, the KKT conditions for the
problem guarantee the existence of $\lambda_k \ge 0$ such that
\begin{align*}
  & g + v_k + \delta_k(x_k - x_0) + \sum_{i=1}^\numactive \lambda_{k, i}
  \nabla f_i(x_k) = 0,
  ~~~ \lambda_{k, i} f_i(x_k) = 0 ~\text{for} ~ 1\leq i\leq m, 
\end{align*}
where $\lambda_{k, i}$ denotes the $i$th coordinate of $\lambda_k$.  We
now argue that for large $k$, we have $\lambda_{k,i} > 0$ for
$i \in [\numactive]$, which by complementary slackness~\cite[Chapter
5]{BoydVa04} implies that $f_i(x_k) = 0$.

For the sake of contradiction, suppose that we do not eventually have
$\lambda_{k, i} > 0$ for all $i$, and without loss of generality, assume that
$\lambda_{k, 1} = 0$ infinitely often; we assume (again w.l.o.g.) that
$\lambda_{k, 1} = 0$ for all $k$. For $i = 0, \ldots, \numactive$, define
the modified Lagrange multipliers $\wt{\lambda}_{k,i} = \lambda_{k, i} /
\sqrt{\sum_{i = 0}^\numactive \lambda_{k,i}^2}$, where $\lambda_{k,0} = 1$.
As $f_i(x_k\opt) < 0$ eventually for $i > \numactive$, we see that
(for large $k$) we have $\lambda_{k, i} = 0$ for $i > \numactive$, so
the KKT gradient condition becomes
\begin{equation*}
  \wt{\lambda}_{k, 0} g + \wt{\lambda}_{k, 0}(v_k + \delta_k(x_k-x_0))
  + \sum_{i=1}^\numactive \wt{\lambda}_{k, i} \nabla f_i (x_k) = 0.
\end{equation*}
The subsequence $\{\wt{\lambda}_k\}$ has an accumulation point
$\wt{\lambda}$, because it lies on the surface of the sphere. Using
our assumption that $x_k \to x\opt$ and that the $\nabla f_i$ are
continuous, we may pass to a limit in the preceding display to obtain
\begin{equation*}
  \wt{\lambda}_0 g + \sum_{i=1}^\numactive \wt{\lambda}_i \nabla f_i(x\opt) = 0.
\end{equation*}
By assumption, we have $\wt{\lambda}_1 = 0$ because $\lambda_{k,1} = 0$
along the subsequence. By the constraint
qualification~\ref{assumption:linear-independence-cq}.\ref{item:linear-independence-cq},
the vectors $\{\nabla f_i(x\opt)\}_{i = 1}^\numactive$ are linearly
independent, so that the weights $\lambda\opt > 0$ satisfying $g = -\sum_{i =
  1}^\numactive \lambda\opt_i \nabla f_i(x\opt)$ are unique.
If $\wt{\lambda}_0 > 0$, we divide by it to obtain
$g = -\sum_{i = 2}^\numactive \wt{\lambda}_i / \wt{\lambda}_0
\nabla f_i(x\opt)$, a contradiction. On the other hand,
if $\wt{\lambda}_0 = 0$, then using $\wt{\lambda}_1 = 0$ we have
$\sum_{i = 2}^\numactive \wt{\lambda}_i \nabla
f_i(x\opt) = 0$, where $\wt{\lambda} \neq 0$, which contradicts
the linear independence of the gradients $\nabla f_i(x\opt)$.

\subsection{Proof of Lemma~\ref{lemma:perturbation-result-linear}}
\label{sec:proof-perturbation-linear}


We begin our development by stating the following lemma, which helps reduce
our discussion to the case where $A$ has independent rows.  The result is
essentially a variant of Carath\'eodory's theorem, so we defer the proof to
Section~\ref{sec:proof-special-treatment-linear}.

\begin{lemma}
  \label{lemma:special-treatment-for-linear-constraints}
  Let $A \in \R^{m\times n}$, and define
  the set
  \begin{equation*}
    S^0 \defeq
    \left\{x\in \R^n\mid x=A^T \lambda: \lambda > 0, \lambda\in \R^m\right\}. 
  \end{equation*}
  Denote the vectors $\{a_i\}_{i \in [n]}$ to be the columns of $A$. 
  Then for any $x\in S^0$
  and any index $i_0 \in [m]$, one of the following two cases occurs:
  \begin{enumerate}[(i)]
  \item There exists a set $T \subset [m]$ and $\lambda_i \ge 0$ such
    that the collection of vectors
    $a_{i_0} \cup \{a_i\}_{i\in T}$ is linearly independent, and
    \begin{equation*}
      x = \lambda_0 a_{i_0} + \sum_{i\in T} \lambda_i a_i ,~\lambda_{0} > 0, 
      \text{and}~\lambda_i > 0~\mbox{for}~ i\in T.
    \end{equation*}
  \item There exists a set $T \subset [m]$ and $\lambda_i \ge 0$ such
    that
    \begin{equation*}
      a_{i_0} + \sum_{i\in T} \lambda_i a_i = 0.      
    \end{equation*}
  \end{enumerate}
\end{lemma}
\noindent
With this lemma established, we can prove the main perturbation result of
Lemma~\ref{lemma:perturbation-result-linear}. 

First, we note that if $A$ has full row rank, the conclusion of
Lemma~\ref{lemma:perturbation-result-linear} follows immediately from
Lemma~\ref{lemma:perturbation-result-nonlinear}. Thus, it remains to
consider the case that $A$ does not have full row rank.  Let $x_k$ be
shorthand for the optimum of the perturbed
problem~\eqref{eqn:perturbed-linear-optimization} with vector $v = v_k$ and
scalar $\delta = \delta_k$, while $x\opt$ denotes an optimum for the
unperturbed problem, where $A x\opt = b$.  We show that for any fixed row
$i_0 \in [m]$ of $A = [a_1 ~ \cdots ~ a_m]^T$, there exists $K \in \N$ such
that $k \ge K$ implies $a_{i_0}^T x_k = b_{i_0}$. As the number of rows of
$A$ is finite, this will imply the result.

By assumption on the linear
program~\eqref{eqn:perturbed-linear-optimization} with $v = 0, \delta = 0$,
there exists $\lambda\opt > 0$ such that $g + A^T \lambda\opt = 0$.
Applying Lemma~\ref{lemma:special-treatment-for-linear-constraints}, we have
two possible cases to consider on the independence structure of $A$. In the
first case, there is a set $I_0 \subset [m]$ with $i_0 \in I_0$ such that
$\{a_i\}_{i \in I_0}$ are linearly independent, and there are $\mu_i > 0$,
$i \in i_0$, such that
\begin{equation*}
  g + \sum_{i \in I_0} \mu_i a_i
  = g + A_{I_0}^T \mu = 0,
\end{equation*}
where $A_I$ denotes the sub-matrix of $A$ whose rows are those indexed by $I$.
In this case, by considering the perturbed optimization problem
\begin{equation*}
  \minimize_x ~ \<g + v_k, x\> + \frac{\delta_k}{2} \norm{x - x_0}^2
  ~~ \mbox{s.t.} ~ A_{I_0}x \le b_{I_0},
  ~ A_{I_0^c} x \le b_{I_0^c}, ~
  C x \le d,
\end{equation*}
we may apply Lemma~\ref{lemma:special-treatment-for-linear-constraints} to
$A_{I_0}$, which has independent rows, to obtain that $A_{I_0} x_k =
b_{I_0}$ for large enough $k$.  In the other case of
Lemma~\ref{lemma:special-treatment-for-linear-constraints}, there is a
subset $I_0 \subset [m]$ with $i_0 \in I_0$ and $\mu_i \ge 0$, $i \in I_0$,
such that
\begin{equation*}
  a_{i_0} + \sum_{i \in I_0 \setminus i_0} \mu_i a_i = 0.
\end{equation*}
Let $x$ be an arbitrary feasible point for the
problem~\eqref{eqn:perturbed-linear-optimization}.  Taking the inner product
of the preceding equality with $x - x\opt$, we have
\begin{equation*}
  0 = \<a_{i_0}, x - x\opt_0\> + \sum_{i \in I_0 \setminus i_0} \mu_i \<a_i, x - x\opt\>
  = \underbrace{\<a_{i_0}, x\> - b_{i_0}}_{\le 0}
  + \sum_{i \in I_0 \setminus i_0} \underbrace{\mu_i (\<a_i, x\> - b_i)}_{\le 0},
\end{equation*}
so that each of the terms $\<a_{i_0}, x\> - b_{i_0}$ and
$\mu_i (\<a_i, x\> - b_i)$ must be zero. That is,
$\<a_{i_0}, x\> = b_{i_0}$, and so it certainly must be the case
that $\langle a_{i_0}, x_k\rangle = b_{i_0}$ as $x_k$ is feasible.

\subsubsection{Proof of
  Lemma~\ref{lemma:special-treatment-for-linear-constraints}}
\label{sec:proof-special-treatment-linear}

The proof of the claim is similar to the standard proof of Carath\'eodory's
theorem on convex hulls~\cite[Chapter III]{HiriartUrrutyLe93ab}. Fixing the
index $i_0$, define the index set
\begin{equation*}
  I(x) \defeq \argmin_{T \subset [m] \setminus \{i_0\}}
  \left\{\card(T) \mid
  \exists \lambda_0 > 0, \lambda \in \R^m_+
  ~ \mbox{s.t.} ~
  x = a_{i_0} \lambda_0 + A \lambda,
  \lambda_T > 0,
  \lambda_{T^c} = 0
  \right\}
\end{equation*}
where $\lambda_T$ denotes the sub-vector of $\lambda$ indexed by $T$.

We must then have one of the following two cases:
\begin{enumerate}[(i)]
\item The vectors $a_{i_0} \cup \{a_i\}_{i \in I(x)}$ are linearly
  independent. This is case (i) of the conclusion of the lemma.
\item The vectors $a_{i_0} \cup \{a_i\}_{i \in I(x)}$ are linearly
  dependent. We claim that in this case, the vectors $\bigcup_{i\in I(x)}
  \{a_i\}$ are linearly independent. If not, there must be a vector $\mu \neq
  0$, where $\mu_i = 0$ for $i \not \in I(x) \cup \{i_0\}$, such that $\sum_{i
    \in I(x)} \mu_i a_i = 0$.  But then considering the quantity $\lambda + t
  \mu$ for $t \in \R$, we have $A (\lambda + t \mu) = A \lambda$ and we may
  choose $t$ such that $\lambda + t \mu \ge 0$ but
  \begin{equation*}
    \card(\{i : \lambda_i +
    t\mu_i \neq 0\}) < \card(\{i : \lambda_i \neq 0\}),
  \end{equation*}
  contradicting the definition of $I(x)$.
  Combined with the fact that $a_{i_0} \cup \{a_i\}_{i \in I(x)}$ are linearly
  dependent, we thus must have $\mu \neq 0$ such that
  \begin{equation*}
    a_{i_0} + \sum_{i\in I(x)} \mu_i a_i = 0.
  \end{equation*}
  Assume that some $\mu_i < 0$ (as otherwise this is exactly case (ii) of the
  lemma).  Let $\lambda_0 > 0, \lambda \in \R^m_+$ be minimizing values in the
  definition of $I(x)$ above.
  Setting $t = \min_{i \in I(x) : \mu_i < 0} |\lambda_i / \mu_i|$, we have
  $\lambda_i + t \mu_i \ge 0$ for all $i$, while
  $\lambda_i + t \mu_i = 0$ for some $i \in I(x)$.
  Then we have
  \begin{equation*}
    x = a_{i_0} \lambda_0 + A \lambda
    = a_{i_0} (t + \lambda_0) + A (\lambda + t \mu),
  \end{equation*}
  a contradiction to the definition of $\lambda_0, \lambda$ in the definition
  of $I(x)$.
\end{enumerate}

\section{The failure of dual averaging}

In this appendix, we collect the proofs of
Observations~\ref{observation:failure-dual-averaging}
and~\ref{observation:failure-classical-dual-averaging}.

\subsection{Proof of Observation~\ref{observation:failure-dual-averaging}}
\label{sec:proof-failure-dual-averaging}

Define $A_k = \sum_{i=1}^k \stepsize_i$ to be the partial sum of the
stepsizes, and recall that
$z_k = \sum_{i=1}^k \stepsize_i g_i = A_k x\opt+ \sum_{i=1}^k \stepsize_i
\noise_i$.
Theorem~\ref{theorem:manifold-identification} shows that
our variant of dual averaging identifies the active constraints in finite time
with probability one. Define by convention that $z/\norm{z} = 0$ when 
$z= 0$. This implies that, 
\begin{equation*}
  \frac{1}{k^\steppow} \sum_{i=1}^k\left(x_{i+1} - \frac{z_i}{\norm{z_i}}\right) 
  \cas 0,
\end{equation*}
because for large enough $i$, we will have 
$x_{i+1} = z_i / \norm{z_i}$. This implies that
The observation will thus follow if we show that
\begin{equation}
  \label{eqn:convergence-normalized-zs}
  \frac{1}{k^\steppow} \sum_{i=1}^k \left(\frac{z_i}{\norm{z_i}} - x\opt\right)
  \cd
  \normal \left(0, \sigma^2 (I- e_1 e_1^T)\right).
\end{equation}
To demonstrate the observation (via the
convergence~\eqref{eqn:convergence-normalized-zs}), then, we provide the
following three technical lemmas, whose proofs we present at the end of this
section in sections~\ref{sec:proof-main-term},
\ref{sec:proof-converge-to-zero}, and
\ref{sec:proof-convergence-of-first-coordinate}, respectively.
\begin{lemma}
  \label{lemma:main-term}
  Under the conditions of the observation, we have
  \begin{equation*}
    \frac{1}{k^\steppow} \sum_{i=1}^k \left(\frac{z_i}{A_i} - x\opt\right) 
    \cd
    \normal \left(0, \sigma^2 I\right).
  \end{equation*}
\end{lemma}
\begin{lemma}
  \label{lemma:converge-to-zero}
  Under the conditions of the observation, we have
  \begin{equation*}
    \frac{1}{k^\steppow} \sum_{i=1}^k \left(\frac{z_i}{\norm{z_i}} - x\opt\right)
    \left(\frac{A_i - \norm{z_i}}{A_i}\right) \cp 0.
  \end{equation*}
\end{lemma}
\begin{lemma}
  \label{lemma:convergence-of-first-coordinate}
  Let $[x]_j$ denote the $j$th coordinate of the vector $x$. Under the
  conditions of the observation, we have
  \begin{equation*}
    \left[\frac{1}{k^\steppow}
      \sum_{i=1}^k \left(\frac{z_i}{A_i} - x\opt\right)
      + \frac{1}{k^\steppow} \sum_{i = 1}^k \frac{A_i - \norm{z_i}}{A_i}
      x\opt
    \right]_1
    \cas 0.
  \end{equation*}
\end{lemma}

The above three lemmas give an almost immediate proof of the
theorem. Expanding the differences $z_i / \norm{z_i} - x\opt$, we have
\begin{align*}
  \lefteqn{\frac{1}{k^\steppow} \sum_{i=1}^k \left(\frac{z_i}{\norm{z_i}} - x\opt\right)
  } \\
  & = \underbrace{
    \frac{1}{k^\steppow} \sum_{i=1}^k \left(\frac{z_i}{\norm{z_i}} - x\opt\right)
    \left(\frac{A_i - \norm{z_i}}{A_i}\right)}_{\eqdef T_1}
  + \underbrace{
    \frac{1}{k^\steppow} \left(\sum_{i=1}^k \frac{A_i- \norm{z_i}}{A_i}\right) x\opt}_{
    \eqdef T_2}
  + \underbrace{\frac{1}{k^\steppow} \sum_{i=1}^k 
    \left(\frac{z_i}{A_i} - x\opt\right)}_{\eqdef T_3}.
\end{align*}
The first term $T_1$ converges to zero in probability
(Lemma~\ref{lemma:converge-to-zero}),
the final $n - 1$ coordinates of the term $T_2$ are zero because $x\opt = e_1$, and the
last $n - 1$ coordinates of term $T_3$ are asymptotically
$\normal(0, \sigma^2 I_{n-1 \times n-1})$ (Lemma~\ref{lemma:main-term}).
Moreover, the first coordinate of $T_2 + T_3$ converges almost surely
to zero by Lemma~\ref{lemma:convergence-of-first-coordinate}.
An application of Slutsky's theorem gives the observation.

\subsubsection{Proof of Lemma~\ref{lemma:main-term}}
\label{sec:proof-main-term}

By assumption, the vector $k^{-\steppow} \sum_{i = 1}^k
(\frac{z_i}{A_i} - x\opt)$ has mean zero. We thus need only show that
the (asymptotic) covariance of the normalized sum is
$\sigma^2 (I - e_1 e_1^T)$. By rearranging the summation, we
have
\begin{align*}
  \frac{1}{k^\steppow}\sum_{i=1}^k \left(\frac{z_i}{A_i} - x\opt \right)
  &= \frac{1}{k^\steppow}\sum_{i=1}^k \frac{\sum_{j=1}^i \stepsize_j \noise_j}{A_i} 
    = \sum_{i=1}^k \left[\frac{1}{k^\steppow}\sum_{j = i}^k \frac{1}{A_j}\right]
    \stepsize_i \noise_i.
\end{align*}
For each $i, k \in \N$, define the normalized partial sums
$B_{i, k} = \frac{1}{k^\steppow}\sum_{j = i}^k \frac{1}{A_j}$,
where $B_{i, k} = 0$ for $i > k$.
We claim that
\begin{equation}
  \sum_{i = 1}^k \stepsize_i^2 B_{i,k}^2 \to \sigma^2
  \label{eqn:limiting-B-sums}
\end{equation}
as $k \to \infty$, which evidently implies the lemma.

To make the claim~\eqref{eqn:limiting-B-sums} formal, we provide the
following technical lemma.
\begin{lemma}
  \label{lemma:convergence-of-B-sequence}
  The set $\left\{B_{i, k}\right\}_{i \leq k}$ has a uniform
  upper bound $\sup_{i, k} B_{i,k} < \infty$, and
  for any fixed $N \in \N$, we have
  \begin{equation*}
    \lim_{k\rightarrow \infty} B_{N, k} = b^* \defeq \frac{1-\steppow}{\steppow}. 
  \end{equation*}
\end{lemma}
\begin{proof}
  By definition of the Riemann integral, we obtain for any
  $\rho \in (0, 1)$ that
  \begin{equation*}
    \frac{1}{k^\rho} \sum_{i=1}^k i^{\rho-1} = 
    \frac{1}{k} \sum_{i=1}^k \left(\frac{i}{k}\right)^{\rho-1} 
    \mathop{\longrightarrow}_{k\uparrow\infty}
    \int_{0}^1 x^{\rho-1} dx = \frac{1}{\rho}.
  \end{equation*}
  Thus $k^{-\steppow} \sum_{i = 1}^k i^{\steppow - 1} \to \steppow^{-1}$ and
  $k^{1 - \steppow} \sum_{i = 1}^k i^{-\steppow} \to (1 - \steppow)^{-1}$.
  Rewriting the second limit, we thus have $k^{\steppow - 1} A_k \to c^* \defeq
  \frac{1}{1 - \steppow}$. This immediately implies that
  \begin{equation*}
    \frac{1}{k^\steppow} \sum_{i=1}^k \left(\frac{1}{A_i} - \frac{1}{c^* i^{1-\steppow}}\right)
    = \frac{1}{k^\steppow} \sum_{i=1}^k \left(\frac{c^* i^{1-\steppow} - A_i}{A_i c^* 
        i^{1-\steppow}}\right)
    = \frac{1}{k^\steppow}
    \sum_{i = 1}^k \frac{1}{A_i c^*} \left(c^* - \frac{A_i}{i^{1 - \steppow}}\right)
    \to 0
  \end{equation*}
  as $k \to \infty$, because $A_i \gtrsim i^{1 - \steppow}$ and
  the term inside the summation is therefore $o(i^{\steppow - 1})$.
  Finally, applying the first Riemann integral approximation above,
  we have
  $\frac{1}{k^\steppow} \sum_{i=1}^k \frac{1}{c^* i^{1-\steppow}} 
  \to \frac{1}{c^* \steppow} = \frac{1-\steppow}{\steppow}$
  as $k \to \infty$. Thus,
  we have
  $\frac{1}{k^\steppow} \sum_{i = 1}^k \frac{1}{A_i} = (1 + o(1)) \frac{1}{k^\steppow}
  \sum_{i = 1}^k (c^* i^{1 - \steppow}) = (1 + o(1)) \frac{1 - \steppow}{\steppow}$,
  and noting that we may ignore the first $N$ terms in the summation
  gives the limit $\lim_k B_{N,k} = b^*$.
  
  The claim of uniform boundedness is immediate because
  $A_j \gtrsim j^{1 - \steppow}$.
\end{proof}

Now we give the claim~\eqref{eqn:limiting-B-sums}. Let $C \geq \sup_{i,k}B_{i,k}$,
where $C < \infty$.
Fix $m \in \N$, and note that for any $k \ge m$, we have
\begin{align*}
  \left|\sum_{j=1}^k \stepsize_j^2 B_{j,k}^2 - \sum_{j=1}^\infty \stepsize_j^2 (b^*)^2 \right|
  & \leq (b^*)^2 \sum_{j=k+1}^\infty \stepsize_j^2 + (C + b^*)^2
    \sum_{j=m+1}^k \stepsize_j^2
    + \sum_{j=1}^m \stepsize_j^2 |B_{j,k}^2 - (b^*)^2| \\
  &\leq (C + b^*)^2 \sum_{j=m+1}^\infty \stepsize_j^2 + 
    (C + b^*) \sum_{j=1}^m \left|B_{j,k} - b^*\right|.
\end{align*}
Taking $k \to \infty$, the final term above tends to zero. As $m \in \N$ is
arbitrary and $\sum_{j = 1}^\infty \stepsize_j^2 < \infty$, this implies the
claim~\eqref{eqn:limiting-B-sums}.

\subsubsection{Proof of Lemma~\ref{lemma:converge-to-zero}}
\label{sec:proof-converge-to-zero}

Because $\sum_{i=1}^\infty \stepsize_i^2 < \infty$, the
standard Kolmogorov three series theorem implies
$S_\infty \defeq \sum_{i=1}^\infty \stepsize_i \noise_i$ exists and converges
almost surely.  In particular, with probability 1 we have the partial sum
bound $\sup_k \norms{\sum_{i=1}^k \stepsize_i \noise_i} < \infty$.  Let
$R = \sup_k \norms{\sum_{i = 1}^k \stepsize_i \noise_i}$ denote this
(random) upper bound.  By the triangle inequality, we then have
\begin{equation*}
  |A_i - \norm{z_i}| = |\norm{A_i x\opt} - \norm{z_i}| \leq \norm{A_i x\opt - z_i} 
  \leq R.
\end{equation*}
Therefore, if we denote $\Delta_i = \norm{x\opt - z_i/\norm{z_i}}$, we have
\begin{equation*}
  \left|\frac{1}{k^\steppow} \sum_{i=1}^k \left(\frac{z_i}{\norm{z_i}} - x\opt\right)
    \left(\frac{A_i - \norm{z_i}}{A_i}\right)\right| 
  \leq  \frac{R}{k^\steppow} \sum_{i=1}^k \frac{\Delta_i}{A_i}
\end{equation*}
where $\Delta_i = 0$ eventually.
In particular, $k^{-\steppow} \sum_{i = 1}^k \frac{\Delta_i}{A_i} \cas 0$,
giving the result.

\subsubsection{Proof of Lemma~\ref{lemma:convergence-of-first-coordinate}}
\label{sec:proof-convergence-of-first-coordinate}

Define the weighted sums $\wt{\noise}_k = \sum_{i=1}^k \stepsize_i \noise_i$,
which converge a.s.\ to a Gaussian vector
$\wt{\noise}_\infty = \sum_{i = 1}^\infty \stepsize_i \noise_i$ with variance
$\sum_i \stepsize_i^2$.  Then the first coordinate of the sum in the lemma is
\begin{equation*}
  \frac{1}{k^\beta} \sum_{i=1}^k \left[\frac{A_i +  \wt{\noise}_{i,1}
      - \norm{z_i}}{A_i} \right].
\end{equation*}
If we can show that the numerator terms in the preceding sum
converge to zero, then the fact that $A_i \gtrsim i^{1 - \steppow}$ will give the result.
Indeed, we know that $A_i + \wt{\noise}_{i,1} \cas \infty$, and thus we have
\begin{align*}
  A_i + \wt{\noise}_{i,1} - \norm{z_i}
  & = A_i + \wt{\noise}_{i,1}
  - \sqrt{(A_i + \wt{\noise}_{1,i})^2 + \sum_{j = 2}^n \wt{\noise}_{i,j}^2} \\
  & = A_i + \wt{\noise}_{i,1}
  - A_i + \wt{\noise}_{i,1}
  - \frac{\sum_{j = 2}^n \wt{\noise}_{i,j}^2}{2 \sqrt{
      A_i + \wt{\noise}_{i,1}}} (1 + o(1))  
    = -\frac{\sum_{j = 2}^n \wt{\noise}_{\infty,j}^2}{
    2 \sqrt{A_i}} (1 + o(1))
\end{align*}
almost surely as $i \to \infty$. Certainly the final term converges
to zero with probability 1, which gives the desired convergence result.

\subsection{Proof of
  Observation~\ref{observation:failure-classical-dual-averaging}}
\label{sec:proof-failure-classical-dual-averaging}

Recall that $x\opt = e_1$, the first standard basis vector.
By definition~\eqref{eqn:standard-dual-averaging} of the dual averaging
sequence $x_k$ and the fact that $g_k = -e_1 - \noise_k$ for
$\noise_k \simiid \normal(0, I)$, if we take $\wb{\noise}_k = \frac{1}{k}
\sum_{i = 1}^k \noise_i$ to be the average of the noise sequence, we have
\begin{equation*}
  x_{k + 1}
  = \begin{cases}
    \stepsize_k \left(k e_1 + k \wb{\noise}_k \right) & \mbox{if}~
    \stepsize_k \norm{k e_1 + k \wb{\noise}_k} \le 1 \\
    \frac{e_1 + \wb{\noise}_k}{\norm{e_1 + \wb{\noise}_k}}
    & \mbox{otherwise}.
  \end{cases}
\end{equation*}
As $\stepsize_k \gtrsim k^{-\steppow}$ for some $\steppow < 0$, with
probability 1 it is eventually the case that
$\norms{k e_1 + k \wb{\noise}_k} \ge \stepsize_k^{-1}$. Consequently, it is no loss
of generality to study the convergence of the iterates
\begin{equation*}
  x_{k} = \frac{x\opt + \wb{\noise}_k}{
    \norm{x\opt + \wb{\noise}_k}}.
\end{equation*}
To prove observation~\ref{observation:failure-classical-dual-averaging}, 
we invoke the three technical lemmas.
The first lemma shows that the variance of the first component
of $\sqrt{k}(\wb{x}_k - x\opt)$ is zero. The second and third lemmas
give the covergence of the latter $n - 1$ components of
$\sqrt{k}(\wb{x}_k - x\opt)$ to a $\normal(0, 2 I_{n-1 \times n-1})$ distribution.
\begin{lemma}
  \label{lemma:failure-convergence-to-zero}
  Under the conditions of the observation, the first coordinate
  of $\sqrt{k} (\wb{x}_k - x\opt)$ converges almost surely to 0.
\end{lemma}
\begin{proof}
  Let $\wb{\noise}_{k, j}$ denote the $j$th coordinate of 
  $\wb{\noise}_k = \frac{1}{k}\sum_{i=1}^k \noise_i$. The
  first coordinate of $x_k - x\opt$ is
  \begin{equation*}
    \frac{1+ \wb{\noise}_{k, 1}}{\sqrt{(1+\wb{\noise}_{k, 1})^2 + \sum_{j \neq 1} 
        \wb{\noise}_{k, j}^2 }} - 1 = 
    \frac{\sum_{j\neq 1} \wb{\noise}_{k, j}^2}
    {\sqrt{(1+\wb{\noise}_{k, 1})^2 + \sum_{j \neq 1}
        \wb{\noise}_{k, j}^2 } (1+\wb{\noise}_{k, 1} + 
      \sqrt{(1+\wb{\noise}_{k, 1})^2 + \sum_{j \neq 1}
        \wb{\noise}_{k, j}^2 })}.
  \end{equation*}
  Evidently, the denominator converges almost surely to $2$, and
  as $k^{2/3} \wb{\noise}_{k,j}^2 = k^{-2/3} \sum_{i = 1}^k \noise_{i,j}^2
  \cas 0$ by the CLT and law of the iterated logarithm, whence
  \begin{equation*}
    k^{2/3} [\wb{x}_k - x\opt]_j
    = \left(\half + o(1)\right) k^{2/3} \sum_{j > 1} \wb{\noise}_{k,j}^2
    \cas 0
  \end{equation*}
  as desired.
\end{proof}

\begin{lemma}
  \label{lemma:asymptotic-variance}
  Let $\noise_i \simiid \normal(0, I)$. Then
  \begin{equation*}
    \frac{1}{\sqrt{k}} \sum_{i=1}^k \frac{1}{i}\sum_{j=1}^{i} \noise_j 
    \cd \normal \left(0, 2I \right)
  \end{equation*}
\end{lemma}
\begin{proof}
  As a linear combination of the independent gaussian random vectors, 
  we know that the average on the left side of the lemma is mean zero and Gaussian.
  Rearranging the sum, we obtain
  \begin{align*}
    \frac{1}{\sqrt{k}} \sum_{i=1}^k \sum_{j=1}^{i} \frac{1}{i}\noise_j = 
    \frac{1}{\sqrt{k}} \sum_{j=1}^k \left(\sum_{i=j+1}^{k} \frac{1}{i}\right)\noise_j,
  \end{align*}
  a sum of independent Gaussian vectors weighted by approximately
  $\log k - \log j$. The coordinates of $\noise_j$ are independent, so
  we need compute the variance only of single components. To that end,
  we note the following equality, which follows by tedious algebraic manipulation:
  \begin{equation*}
    \sigma_k^2 \defeq \frac{1}{k} \sum_{j = 1}^k \left(\sum_{i = j + 1}^k
      \frac{1}{i}\right)^2
    = 2 - \frac{1}{k} \sum_{l = 1}^k \frac{1}{l}
    - \frac{1}{k} \left(\sum_{l = 1}^k \frac{1}{l}\right)^2.
  \end{equation*}
  The last two terms are $O(k^{-1} \log^2 k)$, which gives the lemma.
\end{proof}


\begin{lemma}
  \label{lemma:failure-convergence-to-zero-two}
  Define the error $\delta_k = x_k - (x\opt + \wb{\noise}_k)$.
  Then for each coordinate $j \ge 2$, we have
  \begin{equation*}
    \frac{1}{\sqrt{k}} \sum_{i=1}^k\delta_{i, j} \cas 0.
  \end{equation*}
\end{lemma}
\begin{proof}
  For suitably large $k$, we have
  \begin{align*}
    \delta_k
    & = 
      \frac{x\opt + \wb{\noise}_k}{\norm{x\opt + \wb{\noise}_k}}
      - (x\opt + \wb{\noise}_k)
      = (x\opt + \wb{\noise}_k)
      \left(\frac{1}{\norm{x\opt + \wb{\noise}_k}} - 1\right).
  \end{align*}
  Letting $\wb{\noise}_{k,j}$ be the $j$th coordinate of $\wb{\noise}_k$ as previously,
  the triangle inequality, coupled with Young's inequality, implies for $j \ge 2$ that
  \begin{equation*}
    \left|\delta_{k, j}\right| \leq \left|\wb{\noise}_{k, j}\right| \cdot 
    \frac{\norm{\wb{\noise}_k}}{\norm{x\opt + \wb{\noise}_k}}
    \leq \half \norm{\wb{\noise}_k}^2 
    + \frac{\norm{\wb{\noise}_k}^2}{2 \norm{x\opt + \wb{\noise}_k}^2}.
  \end{equation*}
  As in the proof of lemma~\ref{lemma:failure-convergence-to-zero}, 
  we have $k^{2/3} \norm{\wb{\noise}_k}^2 \cas 0$, implying
  $k^{2/3} |\delta_{k,j}| \cas 0$, whence the lemma follows.
\end{proof}

Combining
Lemmas~\ref{lemma:failure-convergence-to-zero}--\ref{lemma:failure-convergence-to-zero-two},
an application of Slutsky's theorem (or the continuous mapping theorem) yields
the observation.


\section{A generic asymptotic normality result}
\label{sec:generalization-of-Polyak-Juditsky}

In this section, we give the technical tool that serves as the keystone for
the proofs of both Theorems~\ref{theorem:asymptotic-normality} and
Theorem~\ref{theorem:asymptotic-normality-rsgd}.  In short, the result is a
generalizes Polyak and Juditsky's results~\cite{PolyakJu92} on
asymptotic normality in averaged stochastic gradient methods to an arbitrary
subspace of the finite dimensional space $\R^n$ with slightly relaxed
assumptions on the martingale structure the original
results~\cite{PolyakJu92} require.

To better elaborate the technical result, we first describe the
abstract setting along with the
assumptions we require. Let $\T$ be a subspace of $\R^n$. Denote $\projmatr$
to be the projection operator onto the vector space $\T$, so that
if $\T = \{x \in \R^n \mid Ax = 0\}$ for a full row rank $A \in \R^{m \times n}$,
we have
\begin{equation*}
  \projmatr(x) \defeq \argmin_{y \in \T} \norm{y - x}
  = (I - A^T (AA^T)^{-1} A) x.
\end{equation*}
With slight abuse of notation, we also use $\projmatr \in \R^{n \times n}$
to denote the matrix representation of this projection.
As in Assumption~\ref{assumption:noise-vectors},
we assume that $\noise_k$ have martingale difference structure
adapted to $\mc{F}_k$.

Now, let $\{x_k\}_{k\in \N}, \{\notquaderror_k\}_{k \in \N},
\{\offerror_k\}_{k\in \N}, \{\Delta_k\}_{k \in \N}$ be sequences of
vectors in $\R^n$ adapted to the filtration $\mc{F}_k$,
that is, $x_k,
\notquaderror_k, \offerror_k, \Delta_k \in \filt_{k-1}$,
where $\Delta_k = x_k - x\opt$ for a vector $x\opt \in \R^n$.
Assume that for a matrix $H \in \R^{n \times n}$,
we have the recursion
\begin{equation}
  \label{eqn:main-iterate}
  \Delta_{k+1} = \Delta_k - \stepsize_k \projmatr H \projmatr \Delta_k
  + \stepsize_k \projmatr (\noise_k(x_k) + \notquaderror_k)
  + \offerror_k ~~\text{for $k\in \N$},
\end{equation}
where $\Delta_0 \in \T$ and $\offerror_k \in \T$ for all $k$.
We now enumerate assumptions, which are a simplified variant
of Assumptions~\ref{assumption:restricted-strong-convexity}
and \ref{assumption:noise-vectors}, that are sufficient
for the asymptotic normality of the iteration~\eqref{eqn:main-iterate}.

\begin{assumption}
  \label{assumption:growth-and-noise}
  There exists $c > 0$ such that for all $w \in \T$, 
  \begin{equation*}
    w^T H w\ge c \norm{w}^2,
  \end{equation*}
  and there exist constants $0 < \epsilon, C < \infty$ such that
  for all $k \in \N$ and $x$ such that $\norm{x - x\opt} \le \epsilon$,
  \begin{equation*}
    \E\left[\normbig{\noisezero_k}^2 \mid \filt_{k-1}\right] \le C
    ~~~ \mbox{and} ~~~
    \E \left[\normbig{\noisier_k(x)}^2 \mid \filt_{k-1}\right]
    \le C\norm{x-x\opt}^2. 
  \end{equation*}
  Moreover, for some $\Sigma \succeq 0$,  
  \begin{equation*}
    \frac{1}{\sqrt{k}} \sum_{i=1}^{k} \noise_i^{(0)}
    \cd \normal(0, \Sigma).
  \end{equation*}
\end{assumption}

\begin{assumption}
  \label{assumption:convergence-as}
  The sequence $\{\notquaderror_k\}_{k \in \N}$ satisfies
  $\frac{1}{\sqrt{k}}\sum_{i=1}^k  \norm{\projmatr \notquaderror_i}
  \cas 0$, there exists a random variable $T < \infty$ such that
  $\offerror_k = 0$ for $k \ge T$, and the iterates
  $x_k$ satisfy
  \begin{equation*}
    x_k \cas x\opt
    ~~~ \mbox{and} ~~~
    \frac{1}{\sqrt{k}} \sum_{i = 1}^k
    \norm{x_i - x\opt}^2 \cas 0.
  \end{equation*}
\end{assumption}

With Assumptions~\ref{assumption:growth-and-noise} and
\ref{assumption:convergence-as} in place, the following
generalization of \citeauthor{PolyakJu92}'s result~\cite{PolyakJu92}
follows. There are some technicalities because of the subspace
$\T$ and existence of $\offerror_k$, and so we include the proof in
Section~\ref{sec:proof-generalization-polyak-juditsky}.

\begin{proposition}
  \label{proposition:generalization-of-Polyak-Juditsky}
  Let Assumptions~\ref{assumption:growth-and-noise} and
  \ref{assumption:convergence-as} hold, where $\Delta_k = x_k - x\opt$
  satisfies the recursion~\eqref{eqn:main-iterate}. Then
  \begin{equation*}
    \frac{1}{\sqrt{k}}\sum_{i=1}^k \Delta_i \cd
    \normal\left(0, (\projmatr H \projmatr)^\dag \projmatr
    \Sigma \projmatr (\projmatr H \projmatr)^\dag\right).
  \end{equation*}
\end{proposition}

\section{Proof of Theorem~\ref{theorem:asymptotic-normality}}
\label{sec:proof-asymptotic-normality}

In this appendix, we provide the proof of
Theorem~\ref{theorem:asymptotic-normality}.  We prove the theorem in two
main steps, which show that the iterates obey a recursion analogous
to~\eqref{eqn:main-iterate}, which allows us to apply
Proposition~\ref{proposition:generalization-of-Polyak-Juditsky} to prove the
asymptotic normality.  First, we perform a few preliminary calculations that
make rigorous our heuristic that eventually, the (modified) dual averaging
iteration behaves eventually like stochastic gradient descent restricted to
the subspace $\{x : Ax = b\}$.  Second, we provide a convergence rate
argument to show that the conditions of
Proposition~\ref{proposition:generalization-of-Polyak-Juditsky} hold for an
alternative sequence, arguing that this is sufficient to demonstrate the
result for the true sequence of dual averaging iterates.  We defer the
proofs of more technical lemmas.

We begin by introducing a small amount of additional
notation to make our arguments cleaner.
Throughout, as in the statement of the theorem,
we let $H = \nabla^2 f(x\opt)$ and $\projmat = I - A^T (AA^T)^\dag A$.
We abuse notation and define the projected error
\begin{equation*}
  \Delta_k \defeq \projmat (x_k - x\opt).
\end{equation*}
Theorem~\ref{theorem:manifold-identification} shows that
$\Delta_k = x_k - x\opt$ for all large enough $k$ with probability 1,
as dual averaging identifies the active set $\{Ax = b\}$, so that
if we can prove the asymptotic normality of $\Delta_k$ we have proved
the asymptotic normality of $x_k - x\opt$.

By definition of the dual
averaging iteration as $x_{k + 1} = \argmin_{x \in \xdomain} \{\<z_k, x\> +
\half \norm{x}^2\}$, the KKT conditions for the optimizing $x_{k + 1}$ imply
there exist $\lambda_k \ge 0, \mu_k \ge 0$ such that
\begin{equation}
  \label{eqn:single-step-linear-update}
  x_{k + 1} + z_k + A^T \lambda_k + C^T \mu_k =
  x_{k+1} + \sum_{i=1}^k \stepsize_i g_i + A^T \lambda_k + C^T \mu_k = 0,
\end{equation}
where $\lambda_k^T (A x_{k + 1} - b) = 0$ and
$\mu_k^T (C x_{k + 1} - d) = 0$.
Based on the single term optimality~\eqref{eqn:single-step-linear-update},
the next lemma gives a concrete recursive form for the
projected error $\Delta_k$ based on two higher order error sequences.
\begin{lemma}
  \label{lemma:recursion-for-error}
  Define the error sequences
  \begin{equation}
    \label{eqn:extra-error-sequences}
    \begin{split}
      \notquaderror_k & \defeq
      \nabla f(x_k) - \nabla f(x\opt) - \nabla^2 f(x\opt) (x_k - x\opt)
      ~~~ \mbox{and} \\
      \offerror_k & \defeq \projmat C^T (\mu_{k-1} - \mu_k)
      - \stepsize_k \projmat \nabla^2 f(x\opt) (I - \projmat) (x_k - x\opt),
    \end{split}
  \end{equation}
  where $\notquaderror_k$ is the non-quadratic error in $f$ near $x\opt$ and
  $\offerror_k$ is the off-manifold error.  Then
  \begin{equation}
    \label{eqn:recursion-for-error}
    \Delta_{k + 1}
    = (I - \stepsize_k \projmat \nabla^2 f(x\opt) \projmat)
    \Delta_k
    - \stepsize_k \projmat \xi_k
    - \stepsize_k \projmat \notquaderror_k
    + \offerror_k.
  \end{equation}
\end{lemma}

The recursion~\eqref{eqn:recursion-for-error} evidently has the same form as
the generic recursion~\eqref{eqn:main-iterate}, so that the asymptotic
normality of $\wb{\Delta}_k$ will follow if we can demonstrate that the
conditions of Assumption~\ref{assumption:convergence-as} hold, allowing us
to apply Proposition~\ref{proposition:generalization-of-Polyak-Juditsky}.

To that end, we state the following lemma, whose proof is somewhat
delicate but
tracks arguments of other researchers (cf.~\cite[Appendix A.3]{DuchiChRe15}
or~\cite[Proof of Theorem 2, Part 4]{PolyakJu92}). We defer the proof to
Appendix~\ref{sec:proof-convergence-rate-residual}.
\begin{lemma}
  \label{lemma:convergence-rate-of-residual}
  Let the conditions of Theorem~\ref{theorem:asymptotic-normality} hold, and
  assume that $\stepsize_k \propto k^{-\beta}$ for some $\beta \in (\half,
  1)$. Then
  \begin{equation*}
    \frac{1}{\sqrt{k}} \sum_{i=1}^k \norm{x_i - x\opt}^2
    \cas 0.
  \end{equation*}
\end{lemma}

We now argue that the recursion~\eqref{eqn:recursion-for-error} and these
results imply the conditions of
Proposition~\ref{proposition:generalization-of-Polyak-Juditsky}.  To see
this, note that Assumption~\ref{assumption:growth-and-noise} holds by the
conditions of Theorem~\ref{theorem:asymptotic-normality} immediately.  To
verify Assumption~\ref{assumption:convergence-as}, all we need to check is
that $\frac{1}{\sqrt{k}} \sum_{i = 1}^k \norms{\projmat \notquaderror_i}
\cas 0$, where we recall that $\notquaderror_i = \nabla f(x_i) - \nabla
f(x\opt) - \nabla^2 f(x\opt)(x_i - x\opt)$.  There exist $\epsilon > 0$ and
$c < \infty$ such that when $\norm{x_i - x\opt} \le \epsilon$, we have
$\norm{\notquaderror_i} \le c \norm{x_i - x\opt}^2$, so that
\begin{equation*}
  \frac{1}{\sqrt{k}} \sum_{i = 1}^k \norm{\projmat \notquaderror_i}
  \le \frac{1}{\sqrt{k}} \sum_{i = 1}^k
  \norm{\projmat \notquaderror_i} \indic{\norm{x_i - x\opt} > \epsilon}
  + \frac{c}{\sqrt{k}} \sum_{i=1}^k \norm{x_i - x\opt}^2.
\end{equation*}
The latter term converges almost surely to zero by
Lemma~\ref{lemma:convergence-rate-of-residual}, while the former converges
a.s.\ to as $x_i \cas x\opt$.
Finally, Proposition~\ref{proposition:generalization-of-Polyak-Juditsky}
shows that
\begin{equation*}
  \frac{1}{\sqrt{k}} \sum_{i = 1}^k
  \projmat (x_i - x\opt) \cd
  \normal\left(0, (\projmat H \projmat)^\dag \projmat
  \Sigma \projmat (\projmat H \projmat)^\dag \right).
\end{equation*}
As $\projmat (x_i - x\opt) = x_i - x\opt$ for all large enough $i$
and $\projmat^2 = \projmat$ because $\projmat$ is a projection matrix,
we have
$(\projmat P \projmat)^\dag \projmat \Sigma \projmat
(\projmat P \projmat)^\dag
= \projmat P^\dag \projmat \Sigma \projmat P^\dag \projmat$,
which gives the theorem.

\subsection{Proof of Lemma~\ref{lemma:recursion-for-error}}
\label{sec:proof-recursion-for-error}

By using $z_k - z_{k - 1} = \stepsize_k g_k = \stepsize_k (\nabla f(x_k) +
\noise_k)$, the optimality conditions~\eqref{eqn:single-step-linear-update}
imply
\begin{equation*}
  x_{k + 1} = x_k - \stepsize_k g_k + A^T (\lambda_{k - 1} - \lambda_k)
  + C^T (\mu_{k - 1} - \mu_k).
\end{equation*}
Of course, $\projmat A^T = A^T - A^T(AA^T)^\dag AA^T = 0$, so
subtracting $x\opt$ from each side of the above equality
we have
\begin{align*}
  \projmat (x_{k+1} - x\opt) &= \projmat (x_k - x\opt)
  - \stepsize_k \projmat g_k 
  + \projmat C^T (\mu_{k-1} - \mu_k) \\
  & = \projmat (x_k - x\opt) - \stepsize_k \projmat (\nabla f(x_k) + \xi_k)
  + \projmat C^T (\mu_{k-1} - \mu_k) \\
  & =  \projmat (x_k - x\opt) - \stepsize_k \projmat (\nabla f(x\opt) + 
  \nabla^2 f(x\opt)(x_k - x\opt) + \notquaderror_k
  + \xi_k)
  + \projmat C^T (\mu_{k-1} - \mu_k),
\end{align*} 
where we have used the definition of $\notquaderror_k$.
Recognizing that $\nabla f(x\opt) \in \image(A^T)$
because $x\opt$ is optimal,
we further have $\projmat \nabla f(x\opt) = 0$. Substituting
the definition of $\offerror_k$ and noting that
$\projmat^2 = \projmat$, we thus obtain
\begin{align*}
  \Delta_{k + 1} & = \Delta_k
  - \stepsize_k \projmat \nabla^2 f(x\opt) (x_k - x\opt)
  - \stepsize_k \projmat(\xi_k + \notquaderror_k)
  + \projmat C^T (\mu_{k - 1} - \mu_k) \\
  & = \Delta_k
  - \stepsize_k \projmat \nabla^2 f(x\opt) \projmat (x_k - x\opt)
  - \stepsize_k \projmat(\xi_k + \notquaderror_k) \\
  & \qquad ~ 
  - \stepsize_k \projmat \nabla^2 f(x\opt) (I - \projmat) (x_k - x\opt)
  + \projmat C^T (\mu_{k - 1} - \mu_k) \\
  & = \Delta_k
  - \stepsize_k \projmat \nabla^2 f(x\opt) \projmat \Delta_k
  - \stepsize_k \projmat(\xi_k + \notquaderror_k)
  + \offerror_k,
\end{align*}
which is our desired result.

\subsection{Proof of Lemma~\ref{lemma:convergence-rate-of-residual}}
\label{sec:proof-convergence-rate-residual}

\newcommand{\trueerror}{\delta}

Define $\proj$ to be the projection onto the affine set $\{y : Ay = b, Cy
\le d\}$, that is,
\begin{equation*}
  \proj(x) \defeq
  \argmin_{y} \left\{\norm{x - y}^2
  :  Ay = b, Cy \le d\right\}.
\end{equation*}
To demonstrate the result, for each $t \in \N$ we construct an alternative
projected version of $x_k$, which we denote by $x_k^t$, arguing the
convergence of this sequence.

For purely technical reasons, we also require a slightly
different noise sequence definition, where we recall from
Assumption~\ref{assumption:noise-vectors} that the oracle returns noise
\emph{functions} $\noise_k$ with $\E[\noise_k \mid \mc{F}_{k-1}] = 0$, and
$\noise_k(x) = \noisezero_k + \noisier_k(x)$.  For each $t \in \N$ we then
define
\begin{equation*}
  x_{k+1}^t =
  \begin{cases}
    x_{k+1} = x_k^t - \stepsize_k(\nabla f(x_k^t) + \noise_k) + 
    A^T (\lambda_{k-1} - \lambda_k) 
    + C^T (\mu_{k-1} - \mu_k) & \mbox{if}~
    k < t \\
    \proj\big(x_k^t - \stepsize_k (\projmat (\nabla f(x_k^t) + 
    \noise_k (x_k^t)))\big)
    & \mbox{if}~ k \ge t.
  \end{cases}
\end{equation*}
Note that we have used the noise $\noise_k(x_k^t)$ in the definition of
$x_k^t$, which means that the noisy stochastic gradients are computed at the
points $x_k^t$ for the preceding sequence.  In the proof that follows, we also
define the (true and unprojected) errors
\begin{equation}
  \label{eqn:true-error}
  \trueerror_k = x_k - x\opt
  ~~ \mbox{and} ~~
  \trueerror_k^t = x_k^t - x\opt.
\end{equation}

We now state two lemmas, the first of which is more or less standard,
that demonstrate our desired convergence.
\begin{lemma}
  \label{lemma:as-convergence-corrected-sequence}
  For any $t \in \N$, we have
  \begin{equation*}
    \trueerror_k^t \cas 0 ~ \mbox{as~} k \to \infty
    ~~ \mbox{and} ~~
    \sup_k \E\left[\norm{\trueerror_k^t}^2\right] < \infty.
  \end{equation*}
\end{lemma}
\noindent
See Section~\ref{sec:proof-as-convergence-corrected-sequence} for
a proof of the result.

\begin{lemma}
  \label{lemma:convergence-rate-sequence-da}
  Let the conditions of Theorem~\ref{theorem:asymptotic-normality} hold. 
  Then for any $t \in \N$,
  \begin{equation*}
    \frac{1}{\sqrt{k}} \sum_{i=1}^k \norm{\trueerror_i^t}^2 \cas 0.
  \end{equation*}
\end{lemma}
\noindent
See Section~\ref{sec:proof-convergence-rate-sequence-da}
for a proof of the result.

Now, let $T < \infty$
be the random (finite)
manifold identification time
that Theorem~\ref{theorem:manifold-identification} guarantees,
that is, the $T$ such that
for $k \geq T$, we have $A x_k = b$, $C x_k < d$.
Then by definition, we have that $x_k = x_k^T$ for all $k \in \N$,
and as $T < \infty$ with probability 1, we have that
the conclusions of Lemmas~\ref{lemma:as-convergence-corrected-sequence}
and~\ref{lemma:convergence-rate-sequence-da} apply to
$\trueerror_k = x_k - x\opt$.

\subsubsection{Proof of Lemma~\ref{lemma:as-convergence-corrected-sequence}}
\label{sec:proof-as-convergence-corrected-sequence}

Since the projection operator is non-expansive, we have for all $k > t$ that
\begin{align*}
  \half \norm{\trueerror_{k+1}^t}^2 &=  \half\norm{\proj(x_k^t - \stepsize_k \projmat g_k^t) - x\opt}^2 
   \leq \half \norm{\trueerror_k^t - \stepsize_k \projmat g_k^t}^2 \\
  & = \half\norm{\trueerror_k^t}^2 + \half \alpha_k^2 \norm{\projmat 
    (\nabla f(x_k^t) + \noise_k (x_k^t))}^2 - \alpha_k\left<
  \trueerror_k^t, \projmat (\nabla f(x_k^t) + \noise_k (x_k^t))\right>\\
  &\le \half\norm{\trueerror_k^t}^2 +  \half \alpha_k^2 \left(\norm{ 
    \nabla f(x_k^t)}^2 + \norm{\noise_k (x_k^t)}^2 + 
  \left<\nabla f(x_k^t), 2\noise_k (x_k^t)\right>\right) - \alpha_k\left<
  \trueerror_k^t, \nabla f(x_k^t) + \noise_k (x_k^t)\right>
\end{align*}
where in the last equation, we used the fact that projection operator is non-expansive
and the fact that $\projmat \trueerror_k^t = \trueerror_k^t$, as both $x_k^t$ and $x\opt$, 
by definition, lie on the manifold $Ax = b$. 

Notably, we still have $\E[\noise_k(x_k^t) \mid \mc{F}_{k - 1}] = 0$
for the $\sigma$-fields $\mc{F}_k = \sigma(\noise_1, \ldots,
\noise_k)$, and $\noise_k(x_k^t)  \in \mc{F}_k$ as well. Thus
we obtain
\begin{equation*}
  \half \E \left[\norm{\trueerror_{k+1}^t}^2 \mid \mc{F}_{k-1}\right]
  \le \half \norm{\trueerror_k^t}^2
  + \half \stepsize_k^2 \left(\norm{ \nabla f(x_k^t)}^2
  + \E \left[\norm{\noise_k}^2 \mid \mc{F}_{k-1}\right]\right)
  - \stepsize_k 
  \<\nabla f(\proj(x_k^t)), \trueerror_k^t\>.
\end{equation*}
Applying the definition of the noise
sequence $\noise_k(x_k^t) = \noisezero_k + \noisier_k(x_k^t)$,
we have
\begin{equation*}
  \E[\norm{\noise_k}^2 \mid \mc{F}_{k-1}]
  \le 2 \E[\norms{\noisezero_k}^2 \mid \mc{F}_{k-1}]
  + 2 \E[\norms{\noisier_k(x_k^t)}^2 \mid \mc{F}_{k-1}]
  \le 2 C\left(1 + \norm{x_k^t - x\opt}^2\right)
\end{equation*}
by Assumption~\ref{assumption:noise-vectors}. Noting that
$\norm{\nabla f(x)}^2 \le 2 \norm{\nabla f(x) - \nabla f(x\opt)}^2
+ 2 \norm{\nabla f(x\opt)}^2$, we obtain
\begin{equation}
  \label{eqn:one-step-true-error}
  \E \left[ \norm{\trueerror_{k+1}^t}^2 \mid \mc{F}_{k-1}\right]
  \leq (1 + C \stepsize_k^2)
  \norm{\trueerror_k^t}^2
  + \stepsize_k^2 C -
  \stepsize_k \<\nabla f(\proj(x_k^t)), \trueerror_k^t\> 
\end{equation}
whenever $k > t$.

Now we apply the almost supermartingale convergence theorem of Robbins and
Siegmund (Lemma~\ref{lemma:robbins-siegmund}). Because
\begin{equation*}
  \<\nabla f(x_k^t), \trueerror_k^t\>
  = \<\nabla f(x_k^t), x_k^t - x\opt\>
  \ge f(x_k^t) - f(x\opt) \ge 0,
\end{equation*}
using that $\proj(x_k^t) = x_k^t$ for $k > t$ we obtain that
$\norm{\trueerror_k^t}^2$ converges with probability one
to some random variable
$V_\infty < \infty$ and that
\begin{equation*}
  \sum_{k=1}^\infty \stepsize_k \left<\nabla f(\proj(x_k^t)),
  \trueerror_k^t\right> < \infty.
\end{equation*}
Therefore, by Lemma~\ref{lemma:restricted-strong-convexity}, we have 
$\sum_{k=1}^\infty \stepsize_k \indic{\norm{\trueerror_k^t} > \epsilon} < \infty$
for any $\epsilon > 0$.
Since $\trueerror_k^t$ converges and $\sum_{k=1}^\infty \stepsize_k = \infty$, we
know that $V_\infty = 0$ and $\trueerror_k^t \cas 0$ as $k \to \infty$.

For the second statement of the lemma, 
let $k > t$, and define $V_k = \norm{\trueerror_k^t}^2$ for shorthand.
Then in a single step, we have that
\begin{equation*}
  \E[V_{k+1}]
  \le (1 + C \stepsize_k^2) \E[V_k]
  + C \stepsize_k^2.
\end{equation*}
Thus, we find that
\begin{equation*}
  \E[V_{k+1}]
  \le \prod_{i = 1}^k (1 + C \stepsize_i^2) \E[V_1]
  + C \sum_{i = 1}^k
  \stepsize_i^2 \prod_{j = i + 1}^k (1 + C \stepsize_j^2).
\end{equation*}
Of course, we have $\prod_{k = 1}^\infty (1 + C \stepsize_k^2)
< \infty$ whenever $\sum_k \stepsize_k^2 < \infty$, which gives
the result.

\subsubsection{Proof of Lemma~\ref{lemma:convergence-rate-sequence-da}}
\label{sec:proof-convergence-rate-sequence-da}

Hiding the dependence on the index $t$, define the event $\event_{i, l} =
\{\norms{\Delta_j^t} \le \epsilon ~ \mbox{for~} j = i, \ldots, l\}$;
throughout, the remainder of the proof we assume that all indices are $>
t$, so that the iterates $x_k^t$ are simply projected stochastic gradient
descent with projections to the set $\{x : Ax = b, C x \le d\}$. Then we
have that $\event_{i, l} \in \mc{F}_l = \sigma(\noise_1, \ldots,
\noise_{l-1})$.  Moreover, by
Lemma~\ref{lemma:as-convergence-corrected-sequence}, we know that with
probability one
\begin{equation*}
  \event_{i/2, i}
  ~~ \mbox{holds~eventually}.
\end{equation*}
Now, we claim that
\begin{equation}
  \label{eqn:sum-indics-finite}
  \lim_{k \to \infty}
  \sum_{i = 1}^k \E\left[\frac{1}{\sqrt{i}} \norm{\trueerror_i^t}^{2}
    \indic{\event_{i/2, i-1}} \right] < \infty
\end{equation}
Deferring the proof of Eq.~\eqref{eqn:sum-indics-finite} to the sequel,
note that it almost immediately implies the lemma. First,
the monotone convergence theorem implies that
$\sum_{i = 1}^\infty \frac{1}{\sqrt{i}} \norms{\trueerror_i^t}^2
\indic{\event_{i/2,i}} < \infty$ with probability 1. As
$\event_{i/2,i}$ happens eventually,
we thus obtain that
$\sum_{i = 1}^\infty \frac{1}{\sqrt{i}} \norms{\trueerror_i^t}^2
< \infty$ with probability 1 as well. The Kronecker lemma
then implies the Lemma~\ref{lemma:convergence-rate-sequence-da}.

We return to demonstrate the deferred
claim Eq.~\eqref{eqn:sum-indics-finite}.
By the one-step bound of inequality~\eqref{eqn:one-step-true-error} and
Lemma~\ref{lemma:restricted-strong-convexity} we find that for any $l \le
i$,
\begin{align*}
  \E\left[\norm{\trueerror_{i+1}^t}^2
    \indic{\event_{l, i}} \mid \mc{F}_{i-1}\right]
  \le (1 - c \stepsize_i + C \stepsize_i^2) \norm{\trueerror_i^t}^2
  \indic{\event_{l, i}}
  + C \stepsize_i^2 \indic{\event_{l, i}}
\end{align*}
for some constants $c, C$, and for large enough $i$, we have
\begin{equation*}
  (1 - c \stepsize_i + C \stepsize_i^2)
  \le \exp(-c \stepsize_i).
\end{equation*}
In particular, for suitably large $i$ we have
\begin{align*}
  \E\left[\norm{\trueerror_{i+1}^t}^2 \indic{\event_{i/2, i}}\right]
  & \le \exp(-c \stepsize_i)
  \E\left[\norm{\trueerror_{i}^t}^2 \indic{\event_{i/2, i}}\right]
  + C \stepsize_i^2 \P(\event_{i/2, i}) \\
  & \le \exp(-c \stepsize_i)
  \E\left[\norm{\trueerror_{i}^t}^2 \indic{\event_{i/2, i-1}}\right]
  + C \stepsize_i^2 \\
  & \le \exp\left(-c \sum_{j = i/2}^i \stepsize_j\right)
  \E\left[\norms{\trueerror_{i/2}^t}^2\right]
  + C \sum_{l = i/2}^i \stepsize_l^2 \exp\bigg(-c \sum_{j = l}^{i-1}
  \stepsize_j \bigg).
\end{align*}
Noting that the final expectation is upper bounded by some $C < \infty$
by Lemma~\ref{lemma:as-convergence-corrected-sequence}
and that for $\stepsize_i \propto i^{-\steppow}$ we have
$\sum_{j = i/2}^i \stepsize_j \ge c i^{1 - \steppow}$ for some constant $c$,
we obtain that for all suitably large $i$,
\begin{equation*}
  \E\left[\norm{\trueerror_{i+1}^t}^2 \indic{\event_{i/2, i}}\right]
  \le C \exp\left(-c i^{1 - \steppow}\right)
  + C \sum_{l = i/2}^i l^{-2 \steppow}
  \exp\left(-c (i^{1 - \steppow} - l^{1 - \steppow})\right).
\end{equation*}
The next lemma provides a convergence guarantee for 
this sum.
\begin{lemma}[Asi and Duchi~\cite{AsiDu18}, Lemma A.7]
  \label{lemma:funny-zero-sum}
  Let $\steppow \in (\half, 1)$. For any $c > 0$ and $\rho > 1$, there
  exists a constant $C < \infty$ such that
  \begin{equation*}
    \sum_{k = 1}^t k^{-\rho \steppow} \exp\left(-c(t^{1 - \steppow}
    - k^{1 - \steppow})\right) 
    \le C \frac{\log t}{t^{(\rho - 1) \steppow}}
    ~~~ \mbox{for~all~} t \in \N.
  \end{equation*}
\end{lemma}

Substituting the estimate of Lemma~\ref{lemma:funny-zero-sum} into
the previeous inequality yields that
\begin{align*}
  \E\bigg[\sum_{i = 1}^k \frac{1}{\sqrt{i}}
    \norm{\trueerror_{i}^t}^{2}
    \indic{\event_{i/2, i-1}}\bigg]
  & \le C \sum_{i = 1}^k \frac{1}{\sqrt{i}} \exp(-c i^{1 - \steppow})
  + C \sum_{i = 1}^k \frac{\log i}{i^{\steppow + 1/2}}.
\end{align*}
Noting that $\steppow > \half$ gives the
limit~\eqref{eqn:sum-indics-finite}.

\section{Proof of Theorem~\ref{theorem:asymptotic-normality-rsgd}}
\label{sec:proof-asymptotic-normality-rsgd}

In this section, we provide the main part of the proof of Theorem
\ref{theorem:asymptotic-normality-rsgd}.  The rough roadmap is as follows:
we argue that eventually, Algorithm~\ref{alg:da-riemannian} is eventually
more or less performing the stochastic gradient method on the tangent space
$\T$, which allows us to develop a recursion analogous to the
iteration~\eqref{eqn:main-iterate} and apply the generic asymptotic
normality result in
Proposition~\ref{proposition:generalization-of-Polyak-Juditsky}.
To that end, in Section~\ref{sec:preliminary} we collect preliminary
calculus and smoothness results on the active manifold,
while in Section~\ref{sec:convergence-rates-rsgd} we provide
convergence rate guarantees. Finally, in Section~\ref{sec:rsgd-recursion}
we develop the recursion for Algorithm~\ref{alg:da-riemannian}
that shows it follows the form~\eqref{eqn:main-iterate}, allowing us
to apply Proposition~\ref{proposition:generalization-of-Polyak-Juditsky}.

\subsection{Preliminaries, Notation and Conventions}
\label{sec:preliminary}


Throughout the proof, we let $F: \R^n \to \R^{\numactive}$ be $F(x) =
[f_1(x), f_2(x), \ldots, f_{\numactive}(x)]^T$ and define $\M$ to be the
active manifold
\begin{equation*}
  \M \defeq \left\{x \in \R^n: F(x) = 0\right\}
  ~~~ \mbox{and} ~~~
  \M_{\X} \defeq \M \cap \X.
\end{equation*}
Then we have the tangent space
\begin{equation*}
  \T_{\M}(x) = \left\{v \in \R^n: \langle \grad F(x), v\rangle = 0\right\}.
\end{equation*}
and orthogonal projection
\begin{equation*}
  \projt{x} = I - \grad F(x) (\grad F(x)^T \grad F(x))^{\dag} \grad F(x)^T
\end{equation*}
so that $\projt{x}$ is the projection matrix onto $\T_{\M}(x)$ when
$x \in \M$.  Recall
the notation
\begin{equation*}
  H\opt = \Big(\grad^2 f + \sum_{i=1}^{\numactive} 
  \lambda_i\opt \grad^2 f_i\Big)(x\opt)
  ~~\text{and}~~\projmatr = \Pi_{\T_{\M}}(x\opt).
\end{equation*}

We provide a collection of analytical results that allow more
careful Taylor-like expansions throughout the proof.  We present the proofs
of these five results in Section~\ref{sec:proof-Calculus-rsgd} for
completeness.

\begin{lemma}
\label{lemma:stability-of-projection}
  There exist $0 < \eps, C < \infty$ such that $x\in \M \cap \ball(x\opt,
  \eps)$ implies
  \begin{equation*}
    \opnorm{\projt{x} - \projt{x\opt}} \le C\norm{x-x\opt}.
  \end{equation*}
\end{lemma}

\begin{lemma}
  \label{lemma:manifold-second-order-result}
  There exist $0 < \eps, C < \infty$ such that $x\in \M \cap \ball(x\opt,
  \eps)$ implies
  \begin{equation*}
    \norm{\projmatr(x-x\opt) - (x - x\opt)} \le C \norm{x-x\opt}^2. 
  \end{equation*}
\end{lemma}

\begin{lemma}
  \label{lemma:second-order-projection-gradient}
  There exist $0 < \eps, C < \infty$ such that $x\in \M \cap \ball(x\opt,
  \eps)$ implies
  \begin{equation*}
    \norm{\projt{x} \grad f(x) - \projmatr H\opt \projmatr(x-x\opt)}
    \le C\norm{x-x\opt}^2.
  \end{equation*}
\end{lemma}

\begin{lemma}
  \label{lemma:restricted-strong-convexity-near-rsgd}
  There exist $c, \eps > 0$ such that $x\in \M \cap
  \ball(x\opt, \eps)$ implies
  \begin{equation*}
    \langle x - x\opt, \projt{x}\grad f(x) \rangle \ge c \norm{x-x\opt}^2. 
  \end{equation*}
\end{lemma}

\begin{lemma}
  \label{lemma:second-order-control-of-projection}
  There exist $0 < \eps, C < \infty$ such that $x\in \M \cap
  \ball(x\opt, \eps)$ and $v \in \T_\M(x) \cap \ball(0, \eps)$ imply
  \begin{equation*}
    \norm{\proj_{\M}(x+v) - (x+v)} \le C\norm{v}^2. 
  \end{equation*}
\end{lemma}

Throughout the proof of Theorem~\ref{theorem:asymptotic-normality-rsgd}, we
let $\eps > 0$ be a constant small enough that each of the statements of
Lemma~\ref{lemma:stability-of-projection} through
Lemma~\ref{lemma:second-order-control-of-projection} hold for some constants
$0 < c, C < \infty$. We also without mention assume this $\eps > 0$ is
small enough that
\begin{equation}
  \label{eqn:eps-inactive}
  f_i(x) < 0~\text{for all $x \in \ball(x\opt, \eps)$}~\text{and}~i >
  \numactive.
\end{equation}

Throughout the entire proof, we use $c, C$ to denote finite constants
depending only on $\X$, $\eps$, and the stepsize sequence $\stepsize_i$,
whose values may change from line to line.  We use the letter $N$ to denote
large but non-random integer that may depend solely on $\X$ and $\eps$, and
use the letter $T$ to denote large but (possibly) random integer.

\newcommand{\rerror}{e}

\subsection{Convergence rates in
  Theorem~\ref{theorem:asymptotic-normality-rsgd}}
\label{sec:convergence-rates-rsgd}

To follow the roadmap we outline at the beginning of the
section, we require a number of convergence rate and
identification results, which we do in this section.

We begin by proving a convergence rate for the slowly updated dual averaging
sequence $\{x_k\supda\}_{k \in \N}$.  Recall $\wb{x}_k\supda$ is the
weighted average of the actual updated iterates through time $k$, while
$\{\rsgiter_k\}_{k\in \N}$ are the iterates Algorithm~\ref{algorithm:Rdg}
generates.  Lemma~\ref{lemma:pair-convergence} implies $\sum_{i \in
  \timesda}^\infty \stepsize\supda_i [f(x_i\supda) - f(x\opt)] < \infty$, so
that $(\sum_{i \in \timesda}^{i \le k} \stepsize \supda_i)
[f(\wb{x}\supda_k) - f(x\opt)] \cas 0$ by the Kronecker lemma. Using
Theorem~\ref{theorem:as-convergence} that $x\supda_k \cas x\opt$ and that
near $x\opt$, we have $f(x) - f(x\opt) \ge c \norm{x - x\opt}^2$ (recall
Lemma~\ref{lemma:restricted-strong-convexity}), we have
the following convergence lemma.
\begin{lemma}
  \label{lemma:eventually-opt-in-ball}
  Let the conditions of Theorem~\ref{theorem:asymptotic-normality-rsgd}.
  Then with probability 1,
  \begin{equation*}
    \lim_{k \to \infty}
    \left(\sum_{i \in \timesda, i \le k} \stepsize_i\supda\right)
    \cdot \norm{x\opt - \wb{x}\supda_k}^2
    = 0
    ~~~ \mbox{and} ~~~
    \lim_{k \to \infty}
    \epsilon_k^{-\frac{1}{2q}}
    \norm{x\opt - \wb{x}_k\supda} = 0.
  \end{equation*}
\end{lemma}
\noindent
Thus, most iterates are eventually near $x\opt$, and (in fact)
much nearer than $\epsilon_k = (\sum_{i \in \timesda}^{i \le k}
\stepsize_i\supda)^{-q}$.

Before providing convergence rate guarantees for
the iterates $\rsgiter_k$, we collect a few
additional facts on the stepsize sequences and other
constants associated with the rates of convergence we prove.
Let $\steppow_1$ satisfy
$2 q \rho (1 - \steppow) < \steppow_1 < 2 \steppow - 1$,
which is possible by our choice of $q$ in Alg.~\ref{alg:da-riemannian}.
Define $\gamma_k = k^{-\steppow_1}$. Then the stepsize
sequence $\stepsize_k = \stepsize_0 k^{-\steppow}$,
so that
$\eps_k = (\sum_{i \in \timesda}^{i \le k} \stepsize_i\supda)^{-q}
\asymp (\sum_{i = 1}^{k^{\rho}} \stepsize_i)^{-q}
\asymp k^{q\rho(\steppow - 1)} \to 0$ as $k \to \infty$.
Consequently, we have
\begin{equation}
  \label{eqn:def-eps-gamma-k}
  \begin{split}
    \frac{\eps_{k}}{\eps_{k+1}} = 1 + O(1/k),
    ~~~
    & \frac{\gamma_k}{\gamma_{k+1}} = 1 + O(1/k),
    ~~~
    \lim_{k \to \infty} \frac{\gamma_k^{1/2}}{\eps_k} = 0, \\
    \sum \frac{\stepsize_k^2}{\eps_k} < \infty,
    &
    ~~~\mbox{and} ~~~
    \sum \frac{\stepsize_k^2}{\gamma_k} < \infty.
  \end{split}
\end{equation} 
The conditions~\eqref{eqn:def-eps-gamma-k} are all that we require on
the stepsize sequences to prove
Theorem~\ref{theorem:asymptotic-normality-rsgd}; our conditions
on $\eps_k$, $\stepsize_k$, and $\timesda$ are simply sufficient for them.


With Lemma~\ref{lemma:eventually-opt-in-ball}
in place, we have the probability one convergence
of the iterates; we can build off of this to provide slightly stronger
rate of convergence guarantees. To do so, define
the errors
\begin{equation*}
  \Delta_k = \rsgiter_k - x\opt,
  ~~~ \mbox{and} ~~~
  \Delta^{\manifold}_k = \rsgiter^{\manifold}_k - x\opt.
\end{equation*}
We then have the following lemma,
whose proof we provide in
Section~\ref{sec:convergence-rate-rsgd-proof}.

\begin{lemma}
  \label{lemma:convergence-rate-sequence-rsgd}
  Let the conditions of Theorem~\ref{theorem:asymptotic-normality-rsgd}
  hold. Then
  \begin{equation}
    \label{eqn:convergence-rate-quick-convergence}
    \begin{split}
      & \gamma_k^{-1} \norm{\Delta_k}^2 \cas 0,
      ~~~
      \gamma_k^{-1} \normbig{\Delta^{\manifold}_k}^2 \cas 0, \\
      & \stepsize_k \projt{\rsgiter_k}\noise_k \cas 0,
      ~~~\text{and}~~~
      \stepsize_k \projt{\rsgiter_k}g_k \cas 0.
    \end{split}
  \end{equation}
  Additionally, we have
  \begin{equation}
    \label{eqn:average-quad-quick-convergence}
    \frac{1}{\sqrt{k}}\sum_{i=1}^k \norm{\Delta_i}^2 \cas 0
    ~~\text{and}~~
    \frac{1}{\sqrt{k}} \sum_{i=1}^k \stepsize_i
    \norm{\proj_{\M_i}(\rsgiter_i)g_i}^2 \cas 0.
  \end{equation}
\end{lemma}

\subsection{The basic recursion in
  Theorem~\ref{theorem:asymptotic-normality-rsgd}}
\label{sec:rsgd-recursion}

The key result of Lemma~\ref{lemma:convergence-rate-sequence-rsgd} is that
it allows us to show that eventually the iterates $\rsgiter_k$ of
Alg.~\ref{algorithm:Rdg} satisfy the simple recursion $\rsgiter_{k+1} =
\proj_{\M} (\rsgiter_k - \stepsize_k \projt{\rsgiter_k}g_k)$.
This is relatively easy with
Lemmas~\ref{lemma:eventually-opt-in-ball} and
\ref{lemma:convergence-rate-sequence-rsgd}: we have
\begin{lemma}
  \label{lemma:ultimate-rsgd-recursion}
  There exists a random integer $T$, finite with probability
  one, such that  
  \begin{enumerate}[(i)]
  \item \label{item:manifold-identified}
    For all $k \ge T$, we have $\M_k = \M$.
  \item \label{item:iterates-close}
    For all $k \ge T$, we have 
    \begin{equation*}
      \rsgiter_k\in \ball(x\opt, \eps) \cap \M_{\X},
      ~~~
      \rsgiter^{\manifold}_k \in \ball(x\opt, \eps) \cap \M_{\X},
      ~~~\text{and}~~~
      \stepsize_k \projt{\rsgiter_k}g_k \in \ball(0, \eps).
    \end{equation*}
  \item \label{item:essential-recursion} For all $k \ge T$,
    \begin{equation*}
      \rsgiter_{k+1} =
      \proj_{\M} \left(\rsgiter_k - \stepsize_k \projt{\rsgiter_k}g_k\right).
    \end{equation*}
  \end{enumerate}
\end{lemma}
\begin{proof}
  The manifold identification guarantee of
  Theorem~\ref{theorem:manifold-identification} implies that the dual
  averaging iterates eventually lie on $\M$, and
  Lemma~\ref{lemma:convergence-rate-sequence-rsgd} implies that for some
  finite $T_1 < \infty$, $k \ge T_1$ implies
  \begin{equation}
    \label{eqn:T-1-is-good}
    \M_k = \M,
    ~~~\rsgiter_k \in \ball(x\opt, \eps),
    ~~~\rsgiter^{\manifold}_k \in \ball(x\opt, \eps),
    ~~~\text{and}~~~
    \stepsize_k \projt{\rsgiter_k}g_k \in \ball(0, \eps).
  \end{equation} 
  This shows item~\eqref{item:manifold-identified}.  The triangle inequality
  and the fact that $\proj_{\X}$ is nonexpansive imply
  \begin{align*}
    \eps_{k}^{-1} \norms{\proj_{\X}({\rsgiter^{\manifold}_k}) - \wb{x}_k\supda}
    & \le  \eps_k^{-1}\left(\norms{\rsgiter^{\manifold}_k - x\opt}
    +\norms{\wb{x}_k\supda - x\opt}\right). 
  \end{align*}
  Since $\gamma_k^{1/2}/\eps_k \to 0$ by Eq.~\eqref{eqn:def-eps-gamma-k},
  Lemmas~\ref{lemma:eventually-opt-in-ball} and
  \ref{lemma:convergence-rate-sequence-rsgd}
  (Eq.~\eqref{eqn:convergence-rate-quick-convergence}) thus imply
  \begin{align*}
    \eps_k^{-1}
    \norm{\proj_{\X}({\rsgiter^{\manifold}_k}) - \wb{x}_k\supda}
    \cas 0,
  \end{align*}
  indicating that for some finite $T_2 < \infty$, we have
  $\proj_{\X}(\rsgiter^{\manifold}_k) \in \ball_{k, 1}$ for $k \ge T_2$. Let
  $T = \max\{T_1, T_2\}$.

  It remains to show that items~\eqref{item:iterates-close}
  and~\eqref{item:essential-recursion} hold for the $T < \infty$ we have
  just constructed. Recall the definition of $\rsgiter_k,
  \rsgiter^{\manifold}_k$ from Algorithm~\ref{algorithm:Rdg}.
  Then Eq.~\eqref{eqn:T-1-is-good} and the definition of $\rsgiter^{\manifold}_k$
  imply that $\rsgiter_k^{\manifold} \in \ball(x\opt, \eps) \cap \M$
  for $k \ge T$, and as
  $\eps$ is small enough that $f_i(x) < 0$ for $\norm{x - x\opt} \le \eps$
  and $i > \numactive$ (recall assumption~\eqref{eqn:eps-inactive}),
  we have that $k \ge T$ implies
  \begin{equation*}
    \rsgiter^{\manifold}_k\in \M_{\X}.
  \end{equation*}
  Coupled with the containments~\eqref{eqn:T-1-is-good} we have
  item~\eqref{item:iterates-close}. Finally,
  the definition of $\rsgiter_k$ in Alg.~\ref{algorithm:Rdg}
  gives that the recursion in item~\eqref{item:essential-recursion} holds.
\end{proof}

Our third step is to prove asymptotic normality for the averaged projected
errors $\wt{\Delta}_k \defeq \projmatr \Delta_k$.
\begin{lemma}
  \label{lemma:asymptotic-normality-projection-rsgd}
  Let the conditions of Theorem~\ref{theorem:asymptotic-normality-rsgd}
  hold. Then
  \begin{equation*}
    \frac{1}{\sqrt{k}}\sum_{i=1}^k \wt{\Delta}_i
    \cd
    \normal\left(0,
    (\projmatr P\projmatr)^\dag \Sigma (\projmatr P\projmatr)^\dag\right).
  \end{equation*}
\end{lemma}
\begin{proof}
  We pick $T$ (random) sufficiently large that the conclusions of
  Lemma~\ref{lemma:ultimate-rsgd-recursion} hold, and in particular, that
  $\rsgiter_{k+1} = \proj_{\M} (\rsgiter_k - \stepsize_k
  \projt{\rsgiter_k}g_k)$ for all $k$. This is sufficient to develop a
  recursion of the form~\eqref{eqn:main-iterate}.  By
  Lemma~\ref{lemma:second-order-control-of-projection},
  when $k \ge T$, we have the expansion
  \begin{equation*}
    \rsgiter_{k+1}
    = \proj_{\M}(\rsgiter_k - \stepsize_k \projt{\rsgiter_k} g_k)
    = \rsgiter_k - \stepsize_k \projt{\rsgiter_k} g_k
    + \stepsize_k e_k, 
  \end{equation*}
  where the error satisfies $\norms{e_k} \le C \stepsize_k \norms{
    \projt{\rsgiter_k}g_k}^2$ for some constant $C < \infty$. If we
  subtract $x\opt$ and project both sides of this equality onto the linear
  subspace $\T_{\M}(x\opt)$, for $\wt{\Delta}_k = \projmatr \Delta_k$,
  $\wt{e}_k = \projmatr e_k$, and $k \ge T$ we
  have
  \begin{equation*}
    \wt{\Delta}_{k+1}
    = \wt{\Delta}_k - \stepsize_k \projmatr  \projt{\rsgiter_k} g_k
    + \stepsize_k \wt{e}_k.
  \end{equation*}
  By Lemma~\ref{lemma:second-order-projection-gradient}, for
  $k \ge T$ we have the expansion
  \begin{align*}
    \lefteqn{\projt{\rsgiter_k} g_k
      = \projt{\rsgiter_k} \grad f(\rsgiter_k)
      + \projt{\rsgiter_k} \noise_k}
    \\
    &~~= \projmatr H\opt \projmatr  \wt{\Delta}_k +
    \projmatr \noisezero_k + 
    \left[(\projt{\rsgiter_k} - \projmatr) \noisezero_k + 
      \projt{\rsgiter_k} \noisier_k(\rsgiter_k)\right] + \rerror_k'
  \end{align*}
  for an error term $\rerror_k'$ satisfying $\norm{\rerror_k'} \le
  C\norm{\rsgiter_k - x\opt}^2$.
  Substituting this into the preceding display,
  we see that if we define
  \begin{equation*}
    \notquaderror_k \defeq \rerror_k + \rerror_k'
    ~~ \mbox{and} ~~
    \noise_k^{\manifold}(x)
    \defeq (\projt{x} - \projmatr) \noisezero_k
    + \projt{x} \noisier_k(x),
  \end{equation*}
  for $k \ge T$ we obtain
  \begin{equation*}
    \wt{\Delta}_{k+1} =
    \wt{\Delta}_k - \stepsize_k \projmatr H\opt \projmatr \wt{\Delta}_k
    - \stepsize_k \projmatr \noisezero_k -
    \stepsize_k \projmatr \noise_k^{\manifold}(\rsgiter_k)
    + \projmatr \notquaderror_k.
  \end{equation*}
  Evidently, we may introduce a random variable $\offerror_k$ with
  $\offerror_k = 0$ for $k \ge T$ such that
  for all $k$,
  \begin{equation}
    \wt{\Delta}_{k+1} =
    \wt{\Delta}_k - \stepsize_k \projmatr H\opt \projmatr \wt{\Delta}_k
    - \stepsize_k \projmatr \noisezero_k -
    \stepsize_k \projmatr \noise_k^{\manifold}(\rsgiter_k)
    + \projmatr \notquaderror_k
    + \projmatr \offerror_k.
    \label{eqn:fourth-step-expansion}
  \end{equation}

  By inspection, the recursion~\eqref{eqn:fourth-step-expansion}
  is of the form~\eqref{eqn:main-iterate}, so we apply
  Proposition~\ref{proposition:generalization-of-Polyak-Juditsky} to the
  iterates $\wt{\Delta}_k$. It remains to verify
  Assumptions~\ref{assumption:growth-and-noise} and
  \ref{assumption:convergence-as}.

  \paragraph{Verifying Assumption~\ref{assumption:growth-and-noise}}
  The growth condition of
  Assumption~\ref{assumption:growth-and-noise} is immediate by
  Assumption~\ref{assumption:restricted-strong-convexity}.  To see that the
  conditions on the noise sequence of
  Assumption~\ref{assumption:growth-and-noise} hold, note that both terms
  $\projmatr \noisezero_k$ and $\noise_k^{\manifold}$ are martingale difference
  sequences. Assumption~\ref{assumption:noise-vectors}
  immediately implies that
  $\sup_k \E[\norms{\noisezero_k}^2 \mid \mc{F}_{k-1}] < \infty$ and
  $k^{-1/2} \sum_{i = 1}^k \projmatr \noisezero_i \cd \normal(0,
  \projmatr \Sigma \projmatr)$.
  Finally, we note that for all $x\in \R^n$,
  \begin{align*}
    \lefteqn{\E\left[\normbig{
          \noise_k^{\manifold}(x)}^2 \mid \filt_{k-1}\right]
      \indic{\norm{x - x\opt} \le \eps}} \nonumber \\
    & \stackrel{(i)}{\le}
    2 \left(\E \left[\normbig{(\projt{x} - \projmatr) \noisezero_k}^2
      \mid \filt_{k-1}\right]^2	
    +  \E \left[\normbig{\projt{x} \noisier_k(x)}^2 \mid \filt_{k-1}\right]^2
    \right)
    \indic{\norm{x - x\opt} \le \eps} \nonumber \\
    &~\le 2 \left(\opnorm{\projt{x} - \projmatr}^2
    \E\left[\normbig{\noisezero_k}^2 \mid \filt_{k-1}\right]^2
    + \E \left[\normbig{\noisier_k(x)}_2^2 \mid \filt_{k-1}\right]^2\right)
    \indic{\norm{x - x\opt} \le \eps}  \nonumber \\
    & \stackrel{(ii)}{\le} C \norm{x - x\opt}^2 \indic{\norm{x - x\opt} \le \eps}, 
  \end{align*}
  where $(i)$ follows because $(a + b)^2 \le 2a^2 + 2b^2$ for all $a, b \in \R$
  and the
  triangle inequality, and $(ii)$ follows
  from Lemma~\ref{lemma:stability-of-projection} and
  Assumption~\ref{assumption:noise-vectors}.
  This gives the last condition of 
  Assumption~\ref{assumption:growth-and-noise}.

  \paragraph{Verifying Assumption~\ref{assumption:convergence-as}}
  By the triangle inequality, for some constant $C > 0$, the error terms
  $\notquaderror_k$ satisfy
  \begin{align*}
    \frac{1}{\sqrt{k}}
    \sum_{i = 1}^k \norm{\notquaderror_i}
    & \le \frac{1}{\sqrt{k}}
    \sum_{i = 1}^k \left[\norm{\rerror_i} + \norm{\rerror_i'}\right] \\
    & \le \frac{1}{\sqrt{k}}
    \sum_{i = 1}^{T-1} \left[\norm{\rerror_i} + \norm{\rerror_i'}\right]
    +
    \frac{1}{\sqrt{k}}
    \sum_{i = T}^k C \left[\norm{\rsgiter_i - x\opt}^2
      + \stepsize_i \norm{\projt{\rsgiter_i} g_i}^2 \right]
    \cas 0,
  \end{align*}
  convergence is a consequence of
  Lemma~\ref{lemma:convergence-rate-sequence-rsgd} and that $T < \infty$
  with probability 1.
  That $\offerror_k = 0$ for all large $k$ is immediate by the definition
  of the random time $T$ (recall Lemma~\ref{lemma:ultimate-rsgd-recursion}),
  and finally, we obtain that $\frac{1}{\sqrt{k}}
  \sum_{i = 1}^k \norm{\rsgiter_k - x\opt}^2 \cas 0$ by
  Lemma~\ref{lemma:convergence-rate-sequence-rsgd}.

  With the conditions of
  Proposition~\ref{proposition:generalization-of-Polyak-Juditsky} verified,
  we obtain from the recursion~\eqref{eqn:fourth-step-expansion}
  that
  \begin{equation*}
    \frac{1}{\sqrt{k}} \sum_{i = 1}^k \wt{\Delta}_i
    \cd \normal\left(0, (\projmatr H\opt \projmatr)^\dag \projmatr
    \Sigma \projmatr (\projmatr H\opt \projmatr)^\dag \right).
  \end{equation*}
  Lemma~\ref{lemma:asymptotic-normality-projection-rsgd}
  follows by noting that $\projmatr (\projmatr H\opt \projmatr)^\dag
  = (\projmatr H\opt \projmatr)^\dag$.
\end{proof}

Finally, we translate the asymptotic normality of
$\wt{\Delta}_i = \projmatr \Delta_i$
Lemma~\ref{lemma:asymptotic-normality-projection-rsgd} provides to the
unprojected error sequence $\{\Delta_i\}_{i \in \N}$. Let $T < \infty$ be the
random integer such that the conclusions of
Lemma~\ref{lemma:ultimate-rsgd-recursion}
hold. Lemma~\ref{lemma:stability-of-projection} shows that for $k
\ge T$,
\begin{equation*}
  \norms{\Delta_k - \wt{\Delta}_k}
  = \norm{(\rsgiter_k - x\opt) - \projmatr(\rsgiter_k - x\opt)} 
  \le C \norm{\rsgiter_k - x\opt}^2 = C \norm{\Delta_k}^2. 
\end{equation*}
The triangle inequality implies that
\begin{align*}
  \frac{1}{\sqrt{k}}\norm{\sum_{i=1}^k \Delta_i - \sum_{i=1}^k\wt{\Delta}_i}
  & \le \frac{1}{\sqrt{k}} \sum_{i=1}^T
  \norms{\Delta_i- \wt{\Delta}_i} + \frac{1}{\sqrt{k}}
  \sum_{i=T+1}^k \norms{\Delta_i - \wt{\Delta}_i} \\
  &\le \frac{1}{\sqrt{k}} \sum_{i=1}^T \norms{\Delta_i - \wt{\Delta}_i}
  + \frac{C}{\sqrt{k}}
  \sum_{i=1}^k \norm{\Delta_i}^2
  \cas 0,
\end{align*}
where the final convergence guarantee follows from
Lemma~\ref{lemma:convergence-rate-sequence-rsgd} and $T < \infty$.
Slutsky's lemma thus gives the theorem,
as $\frac{1}{\sqrt{k}} \sum_{i = 1}^k
(\Delta_i - \wt{\Delta}_i) \cas 0$.

\section{Proof of Lemma~\ref{lemma:convergence-rate-sequence-rsgd}}
\label{sec:convergence-rate-rsgd-proof}

\newcommand{\bR}{R}

The proof of Lemma~\ref{lemma:convergence-rate-sequence-rsgd} is complex,
and as in our proof of Theorem~\ref{theorem:asymptotic-normality}
(Lemma~\ref{lemma:convergence-rate-of-residual},
Section~\ref{sec:proof-convergence-rate-residual}), we develop a purely
technical auxiliary sequence that has the convergence guarantees we desire,
arguing that this sequence eventually is identical to the true iteration
$\rsgiter_k$.  First, we introduce the
auxiliary sequences $\{\rsgiter_k^t\}_{k\in \N}$ in
Algorithm~\ref{algorithm:Rsgd-fake}. Of course
Algorithm~\ref{algorithm:Rsgd-fake} is not implementable, as it requires
knowledge of the unknown active manifold $\M$. The iterates
$\{\rsgiter_k^t\}_{k\in \N}$, however, are well defined, and they serve as a
mathematical tool to prove
Lemma~\ref{lemma:convergence-rate-sequence-rsgd}.

\begin{algorithm}[htp]
  \caption{Dual Averaging + Riemannian Stochastic Gradient $(t)$ Algorithm}
  \begin{algorithmic}[1] 
    \State Compute the first $t$ iterates as as in
    Algorithm~\ref{alg:da-riemannian}, that is,
    \begin{equation*}
      \rsgiter_k^t = \rsgiter_k
      ~~ \mbox{for}~~ k \le t.
    \end{equation*}

    \State
    Let $g_k^t = \grad f(\rsgiter_k^t) + \noise_k^t$,
    where $\noise_k^t = \noisezero_k + \noisier_k(\rsgiter_k^t)$. 
    Compute the iterate
    \begin{equation*}
      \rsgiter_{k+1}^{\manifold,t} = \proj_{\M}(\rsgiter_k^t - \stepsize_k \projt{\rsgiter_k^t} g_k^t).
    \end{equation*}

    \State Let $\ball_{k, 1} = \ball(\wb{x}_k\supda, \eps_k)$ and $\ball_{k,
      3} = \ball(\wb{x}_k\supda, 3\eps_k)$.

    \State \label{state:all-fake-cases}
    Compute the next iterate
    \begin{equation*}
      \rsgiter_{k+1}^t
      = \begin{cases}
        \proj_{\X}(\rsgiter_{k+1}^{\manifold,t})&\text{if $\proj_{\X}(\rsgiter_{k+1}^{\manifold,t}) \in \ball_{k, 3}$} 
        \\
        \argmin \left\{\norm{x} \mid x \in\M_{\X} \cap \ball_{k, 1}\right\}
        &\text{if}~
        \proj_{\X}(\rsgiter_{k+1}^{\manifold,t}) \not\in \ball_{k, 3}
        ~ \mbox{and} ~
	\M_{\X} \cap \ball_{k, 1} \neq \emptyset
        \\
        \argmin \left\{\norm{x} \mid x \in\M_{\X}\right\}
        &
        \text{otherwise}.
      \end{cases}
    \end{equation*}
  \end{algorithmic}
  \label{algorithm:Rsgd-fake}
\end{algorithm}

Now, we introduce the two error sequences
\begin{equation*}
  \Delta_k^t \defeq \rsgiter_k^t - x\opt
  ~~\text{and}~~
  \Delta^{\manifold,t}_k
  \defeq \rsgiter^{\manifold,t}_k - x\opt. 
\end{equation*}
We show that the guarantees of
Lemma~\ref{lemma:convergence-rate-sequence-rsgd} hold
for each of the artificial sequences $\{\rsgiter_k^t\}_{k \in \N}$,
for all $t \in \N$, which we transfer to $\{\rsgiter_k\}_{k \in \N}$
shortly. We enumerate the results in turn; each tacitly
assumes the conditions of Lemma~\ref{lemma:convergence-rate-sequence-rsgd}
hold.
\begin{lemma}
  \label{lemma:fake-error-cas}
  For all $t \in \N$,
  $\gamma_k^{-1} \norms{\Delta^t_k}^2 \cas 0$.
\end{lemma}
\noindent
See Section~\ref{sec:proof-fake-error-cas} for the proof.
We now demonstrate convergence of the gradient and noise terms.
\begin{lemma}
  \label{lemma:fake-projected-cas}
  For all $t \in \N$,
  \begin{equation*}
    \frac{\stepsize_k}{\sqrt{\gamma_k}} \projt{\rsgiter_k^t}\noise_k^t \cas 0
    ~~~ \mbox{and} ~~~
    \frac{\stepsize_k}{\sqrt{\gamma_k}} \projt{\rsgiter_k^t}g_k^t \cas 0.
  \end{equation*}
\end{lemma}
\noindent
See Section~\ref{sec:proof-fake-projected-cas} for the proof.
We can also provide a rate of convergence estimate for
the manifold-based errors $\Delta_k^{\manifold, t} = \rsgiter_k^{\manifold,t}
- x\opt$.
\begin{lemma}
  \label{lemma:fake-manifold-error-cas}
  For all $t \in \N$, $\gamma_k^{-1} \norms{\Delta_k^{\manifold, t}}^2
  \cas 0$.
\end{lemma}
\noindent
See Section~\ref{sec:proof-fake-manifold-error-cas} for the proof.
Our final lemma provides an alternative convergence rate guarantee
for the errors and projected stochastic gradients; we provide
the proof in Section~\ref{sec:proof-fake-sequence-rsgd-cas}.
\begin{lemma}
  \label{lemma:fake-sequence-rsgd-cas}
  For any fixed $t \in \N$,
  \begin{equation*}
    \frac{1}{\sqrt{k}}\sum_{i=1}^k \norm{\Delta_i^t}^2 \cas 0
    ~~~ \mbox{and} ~~~
    \frac{1}{\sqrt{k}}
    \sum_{i=1}^k \stepsize_i \norm{\projt{\rsgiter_i^t}g_i^t}^2
    \cas 0.
  \end{equation*}
\end{lemma}

With these lemmas, we can finish the proof of
Lemma~\ref{lemma:convergence-rate-sequence-rsgd}.  Let $T < \infty$ be the
random integer the manifold identification result of
Theorem~\ref{theorem:manifold-identification} guarantees,
so that for all $k \ge T$,
$k \in \timesda$, we have $x_k\supda \in \M$. Then
\begin{equation}
  \label{eqn:finite-time-fake-equal-true-rsgd}
  \rsgiter_k = \rsgiter_k^{T},
  ~~\Delta_k^{\manifold} = \Delta_k^{\manifold, T},
  ~~\text{and}~~
  \Delta_k = \Delta_k^T
  ~~~ \mbox{for~all~} k \in \N.
\end{equation}

Let us now demonstrate the convergence
guarantees~\eqref{eqn:convergence-rate-quick-convergence}.  To show the
first assertion of Eq.~\eqref{eqn:convergence-rate-quick-convergence}, we
note that Lemma~\ref{lemma:fake-error-cas} coupled with the
equalities~\eqref{eqn:finite-time-fake-equal-true-rsgd} guarantees that
$\norm{\Delta_k} \cas 0$.  For the second assertion of
Eq.~\eqref{eqn:convergence-rate-quick-convergence}, we use the
equalities~\eqref{eqn:finite-time-fake-equal-true-rsgd} to obtain
\begin{equation*}
  \left\{\lim_k \gamma_k^{-1} \norms{\Delta_k^{\manifold}}^2
  = 0 \right\}
  \supset \bigcap_{t =1}^{\infty}
  \left\{\lim_k \gamma_k^{-1} \norm{\Delta_k^{\manifold, t}}^2 = 0\right\}.
\end{equation*}
By Lemma~\ref{lemma:fake-manifold-error-cas}, the right hand event
occurs with probability 1. The two remaining assertions of
Eq.~\eqref{eqn:convergence-rate-quick-convergence}
follow similarly, where we replace application
of Lemma~\ref{lemma:fake-manifold-error-cas} with
Lemma~\ref{lemma:fake-projected-cas}.

The proof of the assertions~\eqref{eqn:average-quad-quick-convergence} is
immediate given the
equivalence~\eqref{eqn:finite-time-fake-equal-true-rsgd}.

\subsection{Proof of Lemma~\ref{lemma:fake-error-cas}}
\label{sec:proof-fake-error-cas}

We show that the iterates $\rsgiter_k^t$ behave similarly to a
stochastic gradient iteration, exhibiting both
contractive properties as well as an almost supermartingale behavior
for $\norms{\Delta^{t}_k}^2$. We first define
events sufficient to guarantee the iterates
do not leave appropriate neighborhoods of $x\opt$.
For $k \in \N$, define
the $\mc{F}_{k-1}$-measurable events
\begin{equation}
  \begin{split}
    \label{eqn:event-projections-fine}
    & \event_{k,1} \defeq \{x\opt \in \ball_{k,1}\},
    ~~~
    \event_{k,2} \defeq \left\{\stepsize_{k-1}
    \norms{\projt{\rsgiter_{k-1}^t} g_{k-1}^t} \le \eps / 2 \right\},
    \\
    & \event_k \defeq \event_{k-1,1} \cap \event_{k,1},
    ~~~ \mbox{and} ~~~
    \event_k' \defeq \event_{k-1} \cap \event_{k,2}
    = \event_{k-2,1} \cap \event_{k-1,1} \cap \event_{k,2}.
  \end{split}
\end{equation}
The key is that on
these events, the iterates $\rsgiter^t_k$ exhibit
non-expansivity:
\begin{lemma}
  \label{lemma:logic-lemma}
  Fix $t \in \N$ and let $k > t$ be such that $2\eps_k \le \eps$.
  Then on event $\event_{k, 1}$,
  \begin{equation*}
    \norm{\Delta_{k+1}^t} \le 4 \eps_k,
    ~~~
    \norm{\rsgiter^t_{k+1} - x\opt} \le
    \norms{\rsgiter^{\manifold,t}_{k+1} - x\opt},
    ~~ \mbox{and} ~~
    \rsgiter^t_{k+1} \in \M_{\X} \cap \ball_{k,3}.
  \end{equation*}
\end{lemma}

Deferring the proof of Lemma~\ref{lemma:logic-lemma} temporarily
(see Sec.~\ref{sec:proof-logic-lemma}),
we use it to prove Lemma~\ref{lemma:fake-error-cas}.
Define the radius $\bR_k = \sup_{x \in \ball_{k, 3}}\norm{x - x\opt} +
2\norm{x\opt}$. Then by definition of $\rsgiter_{k+1}^t$, $k >
t$ implies
\begin{equation*}
  \norm{\Delta_{k+1}^t} = \normbig{\rsgiter^{t}_{k+1} - x\opt} \le \bR_k.
\end{equation*}
Let $k > t$ satisfy $8\eps_k \le \eps$. By 
Lemma~\ref{lemma:logic-lemma}, we have 
\begin{align}
  \norm{\Delta_{k+1}^t}^2 & \le
  \normbig{\Delta_{k+1}^{\manifold, t}}^2
  \indic{\event_{k+1}'}
  + \bR_k^2  \indic{(\event_{k+1}')^c} \nonumber \\
  &\le \normbig{
    \Pi_{\M}(\rsgiter_k^t - \stepsize_k \projt{\rsgiter_k^t}g_k^t) - x\opt}^2  
  \indic{\event_{k+1}'} + \bR_k^2  \indic{(\event_{k+1}')^c},
  \label{eqn:first-step-contrative-bound-deterministic}
\end{align}
where we recall the definition~\eqref{eqn:event-projections-fine}
of the events $\event_k$ and $\event_k'$.

On event $\event_{k+1}'$, $\rsgiter_k^t \in \M$ by
Lemma~\ref{lemma:logic-lemma}. Using
Lemma~\ref{lemma:second-order-control-of-projection} and the
triangle
inequality, we obtain
\begin{equation}
  \label{eqn:step-one-convergence-fake-rsgd}
  \normbig{\Pi_{\M}(\rsgiter_k^t - \stepsize_k \projt{\rsgiter_k^t}g_k^t) - x\opt}
  \le \normbig{\Delta_k^t - \stepsize_k \projt{\rsgiter_k^t} g_k^t}
  + C\stepsize_k^2 \normbig{\projt{\rsgiter_k^t}g_k^t}^2.
\end{equation}
Now, by definition~\eqref{eqn:event-projections-fine} of $\event_{k+1}'$, we
know that on $\event_{k+1}'$,
\begin{equation*}
  \stepsize_k \normbig{\projt{\rsgiter_k^t}g_k^t} \le \eps.
\end{equation*}
By Lemma~\ref{lemma:logic-lemma} and triangle inequality, we
know on event $\event_{k+1}'$,
\begin{equation*}
  \norm{\Delta_k^t - \stepsize_k \projt{\rsgiter_k^t}g_k^t}
  \le \norm{\Delta_k^t} + \stepsize_k \normbig{\projt{\rsgiter_k^t}g_k^t}
  \le 4\eps_k + \eps/2 \le \eps.
\end{equation*}

For any $0 \le a, b \le \eps$, if $\eps < 1/3$ we have
$(a + b)^2 = a^2 + 2ab + b^2 \le a^2 + b$ whenever $\eps < 1/3$. Consequently,
squaring both sides of Eq.~\eqref{eqn:step-one-convergence-fake-rsgd}
and using the $\eps$ bounds of the preceding displays,
there exists a universal constant $C < \infty$ such that
on $\event_{k+1}'$, we have the crucial estimate
\begin{equation}
  \label{eqn:intermediate-step-contrative-bound-deterministic}
  \norm{\Pi_{\M}(\rsgiter_k^t - \stepsize_k \projt{\rsgiter_k^t}g_k^t)
    - x\opt}^2
  \le 
  \normbig{\Delta_k^t - \stepsize_k \projt{\rsgiter_k^t}g_k^t}^2
  + 2C \stepsize_k^2 \normbig{\projt{\rsgiter_k^t} g_k^t}^2.
\end{equation}
The inclusion $\event_{k+1}' \subset \event_k$ coupled with
inequalities~\eqref{eqn:first-step-contrative-bound-deterministic} and
\eqref{eqn:intermediate-step-contrative-bound-deterministic} shows that
for $k \ge N$, we have
\begin{align*}
  \lefteqn{\norm{\Delta_{k+1}^t}^2} \\
  & \le 
  \left(\norm{\Delta_k^t - \stepsize_k \projt{\rsgiter_k^t}g_k^t}^2
  + 2C \stepsize_k^2 \norm{\projt{\rsgiter_k^t} g_k^t}^2\right)
  \indic{\event_{k+1}'} + \bR_k^2 \indic{(\event_{k+1}')^c} \\
  & \le \left(\norm{\Delta_k^t - \stepsize_k \projt{\rsgiter_k^t}g_k^t}^2
  + 2C \stepsize_k^2 \norm{\projt{\rsgiter_k^t} g_k^t}^2\right)
  \indic{\event_k} + \bR_k^2 \indic{(\event_{k+1}')^c}.
\end{align*}
Taking expectations conditional on $\mc{F}_{k-1}$, and using
that $\event_k \in \mc{F}_{k-1}$, we obtain the intermediate
single-step progress bound
\begin{align}
  \E \left[\norm{\Delta_{k+1}^t}^2 \mid \filt_{k-1}\right]
  & \le \E\left[\norm{\Delta_k^t - \stepsize_k \projt{\rsgiter_k^t}g_k^t}^2
    + 2C \stepsize_k^2 \norm{\projt{\rsgiter_k^t}g_k^t}^2 \mid \filt_{k-1}
    \right] \indic{\event_k} \nonumber \\
  &~~~~~~ + \bR_k^2 \indic{\event_k^c}
  + \bR_k^2 \P(\event_{k+1, 2}^c \mid \filt_{k-1}) \indic{\event_k}
  \label{eqn:first-step-conditional-expectation}
\end{align}

Inequality~\eqref{eqn:first-step-conditional-expectation} allows us to
develop a recursive inequality bounding the norm of $\Delta_{k+1}^t$,
which allows us to apply more standard convergence techniques (such as the
Robbins-Siegmund convergence lemma).  To that end, we analyze the
quantities on the right hand side of the one-step
bound~\eqref{eqn:first-step-conditional-expectation} individually.  We
first see that
\begin{align}
  \lefteqn{\E\left[\norm{\Delta_k^t -
        \stepsize_k \projt{\rsgiter_k^t}g_k^t}^2 \mid \filt_{k-1}\right]
    \indic{\event_k}} \nonumber \\
  & = \left(\norm{\Delta_k^t}^2 -
  2 \stepsize_k \< \Delta_k^t, \projt{\rsgiter_k^t} \grad f(\rsgiter_k^t) \>
  + \stepsize_k^2 \norm{\projt{\rsgiter_k^t} \grad f(\rsgiter_k^t)}^2\right)
  \indic{\event_k}  \nonumber \\
  &~~~~~~~ + \stepsize_k^2
  \E\left[\norm{\projt{\rsgiter_k^t}\noise_k^t}^2 \mid \filt_{k-1}\right]
  \indic{\event_k} \nonumber \\
  & \stackrel{(i)}{\le}
  \left((1 - c \stepsize_k + C \stepsize_k^2)
  \norm{\Delta_k^t}^2
  + \stepsize_k^2 \E\left[\normbig{\projt{\rsgiter_k^t} \noise_k^t}^2
    \mid \filt_{k-1}\right]\right) \indic{\event_k} \nonumber \\
  & \stackrel{(ii)}{\le}
  \left((1 + C \stepsize_k^2) \norm{\Delta_k^t}^2
  - c \stepsize_k \norm{\Delta_k^t}^2\right)\indic{\event_k}.
  \label{eqn:riemmanian-descent-expectation}
\end{align}
Here inequality $(i)$ follows because when $\rsgiter_k^t \in \ball(x\opt,
\epsilon)$, which event $\event_k$ guarantees by
Lemma~\ref{lemma:logic-lemma},
Lemma~\ref{lemma:second-order-projection-gradient} gives
$\norms{\projt{\rsgiter_k^t} \grad f(\rsgiter_k^t)} \le C
\norms{\Delta_k^t}$, while
Lemma~\ref{lemma:restricted-strong-convexity-near-rsgd} implies that
$\<\Delta_k^t, \projt{\rsgiter_k^t} \grad f(\rsgiter_k^t)\> \ge c
\norm{\Delta_k^t}^2$ in the same situation.  Inequality~$(ii)$ is a
consequence of Assumption~\ref{assumption:noise-vectors} and that
$\projt{\rsgiter_k^t}$ is a projection matrix.

We now turn to the second term in
inequality~\eqref{eqn:first-step-conditional-expectation}.  Here
we immediately obtain
$\E[\norms{\projt{\rsgiter_k^t}g_k^t}^2 \mid \filt_{k-1}]\indic{\event_k}
\le C(1 + \norms{\Delta_k^t}^2) \indic{\event_k}$ by
Assumption~\ref{assumption:noise-vectors}. For the final
term in the bound~\eqref{eqn:first-step-conditional-expectation},
we can apply Markov's inequality to find that for
a constant $C = C(\epsilon) < \infty$, for $k \ge N$
\begin{equation*}
  \P(\event_{k+1, 2}^c\mid \filt_{k-1})\indic{\event_k}
  \le \frac{4\stepsize_k^2}{\eps^2}
  \E\left[\norm{\projt{\rsgiter_k^t} g_k^t}^2 \mid \filt_{k-1}\right]
  \indic{\event_k} 
  \le C \stepsize_k^2 (1+ \norm{\Delta_k^t}^2)\indic{\event_k},
\end{equation*}
where the last inequality again uses
Assumption~\ref{assumption:noise-vectors}.  Substituting the preceding
bounds and inequality~\eqref{eqn:riemmanian-descent-expectation} into the
single step Eq.~\eqref{eqn:first-step-conditional-expectation}, we see
that for all $k \ge N$,
\begin{align}
  \lefteqn{\E[\norm{\Delta_{k+1}^t}^2 \mid \filt_{k-1}]} \nonumber \\
  & \le (1-c\stepsize_k + C\stepsize_k^2 (1+\bR_k^2))
  \norm{\Delta_k^t}^2 \indic{\event_k}+ C \stepsize_k^2(1 + \bR_k^2)  \indic{\event_k} + \bR_k^2 \indic{\event_k^c}.
  \label{eqn:crucial-estimate}
\end{align}

Inequality~\eqref{eqn:crucial-estimate} is the crucial estimate that
allows us to apply the Robbins-Siegmund lemma to obtain a
convergence rate. Dividing each side of
inequality~\eqref{eqn:crucial-estimate} by $\gamma_k$ and setting
$\kappa_k \defeq \frac{\gamma_k - \gamma_{k+1}}{\gamma_{k+1}}$, we
obtain
\begin{equation*}
  \E\left[\frac{\norms{\Delta_{k+1}^t}^2}{\gamma_{k+1}}
    \mid \filt_{k-1}\right] \le (1+\kappa_k)
  (1-c\stepsize_k + C\stepsize_k^2 (1+\bR_k^2)) \frac{\norms{\Delta_k^t}^2}{
    \gamma_k} +
  C  \frac{\stepsize_k^2}{\gamma_{k+1}}(1+\bR_k^2) + 
  \frac{\bR_k^2}{\gamma_{k+1}} \indic{\event_k^c}.
\end{equation*}
By assumption on $\{\gamma_k\}_{k\in \N}$, for
all large enough $k$ we have $\kappa_k \le (c \stepsize_k / 2) \wedge 1$,
and thus for such $k$ we have
\begin{equation*}
  (1+\kappa_k)(1-c\stepsize_k + C\stepsize_k^2(1+\bR_k^2))  
  \le 1 - c\stepsize_k/2 + 2C\stepsize_k^2 (1+\bR_k^2).
\end{equation*}
Substituting this above, we have
for large enough $k$ that 
\begin{equation}
  \label{eqn:crucial-estimate-rate-convergence-simplify}
  \E\left[\frac{\norms{\Delta_{k+1}^t}^2}{\gamma_{k+1}}
    \mid \filt_{k-1}\right]
  \le 
  (1+C\stepsize_k^2(1+\bR_k^2))\frac{\norms{\Delta_k^t}^2}{\gamma_k}
  - c\stepsize_k \frac{\norms{\Delta_k^t}^2}{\gamma_k}
  + C \frac{\stepsize_k^2}{\gamma_{k+1}}(1 + \bR_k^2) + 
  \frac{\bR_k^2}{\gamma_{k+1}} \indic{\event_k^c}.
\end{equation}
By Theorem~\ref{theorem:as-convergence}, we have
$\wb{x}_k\supda \cas x\opt$, and thus
\begin{equation*}
  \{\event_k~\text{occurs for only finitely many}~k\}
  ~~~\text{and}~~~
  \limsup_{k\to \infty} \bR_k \le 2\norm{x\opt}
\end{equation*}
with probability $1$.  Our assumptions on $\{\gamma_k\}_{k \in \N}$ in
Eq.~\eqref{eqn:def-eps-gamma-k} and that the stepsizes $\stepsize_k$ are
square summable imply that
\begin{equation*}
  \sum_{k=1}^\infty \stepsize_k^2(1+\bR_k^2) < \infty,~
  \sum_{k=1}^\infty \frac{\stepsize_k^2}{\gamma_{k+1}}(1+\bR_k^2) < \infty~\text{and}~
  \sum_{k=1}^\infty  \frac{\bR_k^2}{\gamma_{k+1}} \indic{\event_k^c} <\infty
\end{equation*}
with probability 1.  By applying the Robbins-Siegmund
Lemma~\ref{lemma:robbins-siegmund} to the
recursion~\eqref{eqn:crucial-estimate-rate-convergence-simplify}, we
conclude that $\norm{\Delta_{k}^t}^2 / \gamma_k \cas V$ for some finite
random variable $V$ and that $\sum_{k=1}^\infty \stepsize_k
\norms{\Delta_k^t}^2 / \gamma_k < \infty$. As $\sum_{k=1}^\infty
\stepsize_k \to \infty$, we must have $V = 0$, giving
the desired claim of the lemma.

\subsubsection{Proof of Lemma~\ref{lemma:logic-lemma}}
\label{sec:proof-logic-lemma}

Recall that $\M_{\X} = \X \cap \M$.  On the event $\event_{k, 1}$, we know
that $x\opt \in \M_{\X} \cap \ball_{k, 1}$, and so $\M_{\X} \cap \ball_{k,
  1} \neq \emptyset$. By definition (Alg.~\ref{algorithm:Rsgd-fake},
line~\ref{state:all-fake-cases}) of the
iterates $\rsgiter_{k+1}^t$, either
\begin{equation}
  \label{eqn:logic-case-one}
  \rsgiter_{k+1}^t= \proj_{\X}(\rsgiter_{k+1}^{\manifold,t})
  ~\text{and}~
  \proj_{\X}(\rsgiter_{k+1}^{\manifold,t}) \in \ball_{k, 3}
\end{equation}
or 
\begin{equation}
  \label{eqn:logic-case-two}
  \rsgiter_{k+1}^t =
  \argmin_x \left\{\norm{x} \mid x \in\M_{\X} \cap \ball_{k, 1}\right\}
  ~\text{and}~ 
  \proj_{\X}(\rsgiter_{k+1}^{\manifold,t}) \not\in \ball_{k, 3}
\end{equation}
In either of the two cases
$\rsgiter_{k+1}^t \in \ball_{k, 3}$ and thus
\begin{equation*}
  \norm{\Delta_{k+1}^t} = \norm{\rsgiter_{k+1}^t - x\opt} \le 
  \norms{\rsgiter_{k+1}^t - \wb{x}_k\supda} + \norms{\wb{x}_k\supda - x\opt}
  \le 4\eps_k.
\end{equation*}
This proves the first claim of the lemma.

For the second claim, the non-expansivity result, we note that
the convexity of $\X$ immediately implies
\begin{equation*}
  \norm{\proj_{\X}(\rsgiter_{k+1}^{\manifold,t}) - x\opt}
  \le \norm{\rsgiter_{k+1}^{\manifold,t} - x\opt}.
\end{equation*}
We divide our discussion into two cases. 
\begin{enumerate}[1.]
\item In the first case, we assume that Eq.~\eqref{eqn:logic-case-one}
  holds.  In this case, we know that $\rsgiter_{k+1}^t =
  \proj_{\X}(\rsgiter_{k+1}^{\manifold,t})$ and hence the preceding
  display implies the claim that $\norm{\rsgiter^t_{k+1} - x\opt} \le
  \norm{\rsgiter_{k+1}^{\manifold, t} - x\opt}$.
\item In the second case, we assume that 
  Eq.~\eqref{eqn:logic-case-two} holds. In this case, we know that 
  \begin{equation*}
    \norm{\rsgiter_{k+1}^t - x\opt}\le 
    \norm{\rsgiter_{k+1}^t - \wb{x}_k\supda}
    + \norm{\wb{x}_k\supda - x\opt}
    \le 2\eps_k,
  \end{equation*}
  and
  \begin{equation*}
    \norm{\rsgiter_{k+1}^{\manifold,t} - x\opt} \ge 
    \norm{\proj_{\X}(\rsgiter_{k+1}^{\manifold,t}) - x\opt}
    \ge \norm{\proj_{\X}(\rsgiter_{k+1}^{\manifold,t}) - \wb{x}_k\supda} - 
    \norm{\wb{x}_k\supda - x\opt} \ge 2\eps_k.
  \end{equation*}
  Thus $\norm{\rsgiter_{k+1}^t - x\opt} \le
  \norm{\rsgiter_{k+1}^{\manifold, t} - x\opt}$ as desired.
\end{enumerate}

Lastly, we show the final claim of Lemma~\ref{lemma:logic-lemma},
that is, that $\rsgiter_{k+1}^t \in \M\cap \ball_{k, 3}$.
Again, we divide our proof into the cases~\eqref{eqn:logic-case-one}
and~\eqref{eqn:logic-case-two}.
\begin{enumerate}[1.]
\item  In the first case that Eq.~\eqref{eqn:logic-case-one} holds,
  we have
  \begin{equation*}
    \rsgiter_{k+1}^t
    = \proj_{\X}(\rsgiter_{k+1}^{\manifold,t}) \in \X \cap \ball_{k, 3}.
  \end{equation*}
  If we can show that
  \begin{equation*}
    \rsgiter_{k+1}^{\manifold,t} \in \X,
  \end{equation*}
  then we are done, as $\rsgiter_{k+1}^t =
  \proj_{\X}(\rsgiter^{\manifold,t}_{k+1}) =
  \rsgiter^{\manifold,t}_{k+1}$.  To see this, recall the event
  $\event_{k,1}$ (Eq.~\eqref{eqn:event-projections-fine}), and note that
  $\epsilon_k \le \epsilon/4$ so that $x\opt \in \ball_{k,1}$ and hence
  $\ball_{k,3} \subset \ball(x\opt, \eps)$ on $\event_{k,1}$. By
  assumption~\eqref{eqn:eps-inactive} on $\eps$, we know that
  $f_i(\rsgiter_{k+1}^t) < 0$ for $i > \numactive$, as
  $\norm{\rsgiter_{k+1}^t - x\opt} \le \eps$.  For the sake of
  contradiction, assume that $\rsgiter_{k+1}^{\manifold,t} \not \in
  \X$. As $\rsgiter_{k+1}^{\manifold,t} \in \M$, we must have
  $f_i(\rsgiter_{k+1}^{\manifold,t}) > 0$ for some $i > \numactive$. Now,
  for $\lambda \in (0, 1)$ define
  \begin{equation*}
    \rsgiter_{k+1}^{t}(\lambda)
    \defeq (1-\lambda) \rsgiter_{k+1}^t + \lambda \rsgiter_{k+1}^{\manifold,t}.
  \end{equation*}
  By convexity of the functions $\{f_i\}_{i \in [m]}$, it is clear that
  for $\lambda$ sufficiently small,
  \begin{equation*}
    \max_{i > \numactive} f_i(\rsgiter_{k+1}^{t}(\lambda)) < 0
  \end{equation*}
  so that $\rsgiter_{k+1}^{t}(\lambda) \in \X$. Moreover,
  $\norm{\rsgiter_{k+1}^t(\lambda) - \rsgiter_{k+1}^{\manifold, t}}
  < \norm{\rsgiter_{k+1}^t - \rsgiter_{k+1}^{\manifold, t}}$,
  contradicting our assumption 
  that $\rsgiter_{k+1}^t = \proj_{\X}(\rsgiter_{k+1}^{\manifold,t})$. Thus
  we have $\rsgiter_{k+1}^{\manifold, t} \in \X$ as desired.
\item In the second case, we assume that Eq.~\eqref{eqn:logic-case-two}
  holds. Then by construction $\rsgiter_{k+1}^t \in \M_{\X}\cap
  \ball_{k, 3}$, giving the final claim of Lemma~\ref{lemma:logic-lemma}.
\end{enumerate}
This completes the proof of all three claims.

\subsection{Proof of Lemma~\ref{lemma:fake-projected-cas}}
\label{sec:proof-fake-projected-cas}

First, we prove $\stepsize_k \projt{\rsgiter_k^t}(\noise_k^t) \cas 0$. To
do so, note that (i) the sequence $\{\sum_{i=1}^k \stepsize_i
\projt{\rsgiter_i^t}\noise_{i}^t\}_{k =1}^\infty$ is a square integrable
martingale adapted to $\filt_k$ and (ii), we have
\begin{equation*}
  \sum_{i=1}^\infty \frac{1}{\gamma_i}\E \left[\norm{\alpha_i \projt{\rsgiter_i^t}\noise_{i}^t}^2 \mid \filt_{i-1}\right] 
  \stackrel{(i)}{\le}
  C \sum_{i=1}^\infty \frac{\stepsize_i^2}{\gamma_i}
  (1 + \norm{\Delta_i^t}^2)
  \stackrel{(ii)}{<} \infty,
\end{equation*}
where $(i)$ follows by Assumption~\ref{assumption:noise-vectors} and $(ii)$
because $\Delta_k^t \cas 0$ by Lemma~\ref{lemma:fake-error-cas} and our
choice of sequence $\gamma_i$ satisfying $\sum_{i=1}^\infty
\stepsize_i^2/\gamma_i < \infty$ (recall Eq.~\eqref{eqn:def-eps-gamma-k}).
Thus, standard martingale convergence theorems on
square-integrable martingales~\cite[Theorem 5.3.33 (a)]{Dembo16}
imply that $\sum_{i=1}^\infty \stepsize_i
\projt{\rsgiter_i^t}\noise_{i}^t/\gamma_i^{1/2}$ converges w.p.\ 1, and in
particular
\begin{equation}
  \label{eqn:quick-convergence-of-noise-gamma-k}
  \frac{\stepsize_k}{\sqrt{\gamma_k}}
  \projt{\rsgiter_k^t}\noise_{k}^t \cas 0.
\end{equation} 
As $\gamma_k\to 0$, this proves the first claim. 

Next, we prove $\stepsize_k \projt{\rsgiter_k^t}g_k^t \cas 0$. Note that
$\stepsize_k^2/\gamma_k \to 0$ since $\sum_{i=1}^\infty
\stepsize_i^2/\gamma_i < \infty$. Moreover, we have
$\norms{\projt{\rsgiter_k^t} \grad f(\rsgiter_k^t)} \cas 0$, since
$\rsgiter_k^t \cas x\opt$ by
Lemma~\ref{lemma:fake-error-cas} and $\projmatr\grad
f(x\opt) = 0$.  Therefore, by Cauchy-Schwarz and
Eq.~\eqref{eqn:quick-convergence-of-noise-gamma-k}, we have
\begin{equation*}
  \gamma_k^{-1} \norm{\stepsize_k \projt{\rsgiter_k^t} g_k^t}^2 
  \le 2 \stepsize_k^2 \gamma_k^{-1} \norm{ \projt{\rsgiter_k^t} \grad f(\rsgiter_k^t)}^2 + 
  2 \gamma_k^{-1} \norm{\stepsize_k \projt{\rsgiter_k^t}(\noise_k^t)}^2 \cas 0. 
\end{equation*}
As $\gamma_k\to 0$, this proves the second claim. 

\subsection{Proof of Lemma~\ref{lemma:fake-manifold-error-cas}}
\label{sec:proof-fake-manifold-error-cas}
We have $y_k^t \cas x\opt$ by Lemma~\ref{lemma:fake-error-cas} and
$\stepsize_k \projt{\rsgiter_k^t} g_k^t \cas 0$ by
Lemma~\ref{lemma:fake-projected-cas}.  Thus, for some (random) $T < \infty$,
we have $y_k^t \in \ball(x\opt, \eps)$ and $\stepsize_k \projt{\rsgiter_k^t}
g_k^t \in \ball(0, \eps)$ for all $k \ge
T$. Recall the definition~\eqref{eqn:event-projections-fine} of
the events $\event_k$ and $\event_k'$.
Lemmas~\ref{lemma:second-order-control-of-projection} and
\ref{lemma:logic-lemma} show that, for some constant $C < \infty$, if
$k \ge T$ then
\begin{align*}
  \normbig{\Delta_{k+1}^{\manifold, t}} \indic{\event_k'}
  & =
  \norm{\proj_{\M}(y_k^t - \stepsize_k \projt{\rsgiter_k^t} g_k^t) - x\opt}
  \indic{\event_k'} \\
  & \le  \norm{\Delta_k^t - \stepsize_k \projt{\rsgiter_k^t} g_k^t}
  + C \norm{\stepsize_k \projt{\rsgiter_k^t} g_k^t}^2. 
\end{align*}
By Cauchy-Schwarz, we have
\begin{equation*}
  \gamma_k^{-1}\normbig{\Delta_{k+1}^{\manifold, t}}^2\indic{\event_k'}
  \le C \gamma_k^{-1}
  \left( \norm{\Delta_k^t}^2 +\norm{\stepsize_k \projt{\rsgiter_k^t} g_k^t}^2
  + \norm{\stepsize_k \projt{\rsgiter_k^t} g_k^t}^4
  \right).
\end{equation*}
Thus, Lemmas~\ref{lemma:fake-error-cas} and
\ref{lemma:fake-projected-cas} imply that
$\gamma_k^{-1}\normbig{\Delta_{k+1}^{\manifold, t}}^2 \indic{\event_k}
\cas 0$.
Since $\stepsize_k \projt{\rsgiter_k^t}g_k^t \cas 0$ and 
$\eps_k^{-1} \norm{\wb{x}_k\supda - x\opt} \cas 0$ by 
Lemma~\ref{lemma:eventually-opt-in-ball},
we get $\indic{\event_k} = 1$ eventually. Thus we have
$\gamma_k^{-1}\norms{\Delta_{k+1}^{\manifold, t}}^2 \cas 0$
as desired.

\subsection{Proof of Lemma~\ref{lemma:fake-sequence-rsgd-cas}}
\label{sec:proof-fake-sequence-rsgd-cas}

The proof of this result parallels
Lemma~\ref{lemma:convergence-rate-sequence-da}.
For indices $i_1 \le i_2$, define the events
\begin{equation*}
  \event_{i_1, i_2} \defeq \bigcap_{l=i_1}^{i_2} \event_l' \in \filt_{i_2 - 1}
  ~~\text{and}~~
  \event_{i_1, i_2}' \defeq
  \bigcap_{l=i_1}^{i_2} \event_l' \cap \event_{i_2} \in \filt_{i_2 - 1},
\end{equation*}
where we recall the events~\eqref{eqn:event-projections-fine}.
By Lemmas~\ref{lemma:eventually-opt-in-ball}
and \ref{lemma:fake-projected-cas}, we know that the event 
$\event_{i/2, i}'$ happens eventually. We claim that
\begin{equation}
  \label{eqn:expectation-go-to-zero}
  \lim_{k \to \infty}
  \sum_{i=1}^k \E \left[\frac{1}{\sqrt{i}}
    \norm{\Delta_i^t}^2 \indic{\event_{i/2, i}}\right] < \infty.
\end{equation}
Defering the proof of the claim temporarily, let us show
how it yields the lemma. Monotone convergence implies
$\sum_{i = 1}^\infty \frac{1}{\sqrt{i}} \norms{\Delta_i^t}^2 \indic{\event_{i/2,i}} < \infty$ with probability 1, and the Kronecker lemma
gives
$\frac{1}{\sqrt{k}} \sum_{i = 1}^k \norms{\Delta_i^k} \indic{\event_{i/2,i}}
\cas 0$. As $\event_{i/2,i}$ occurs eventually, the first
claim of the lemma follows.

To see the second claim of the lemma,
we use Assumption~\ref{assumption:noise-vectors} and that
$\event_{k/2,k}' \in \mc{F}_{k-1}$ to obtain
\begin{equation*}
  \E\left[\norm{\projt{\rsgiter_k^t}g_k^t}^2 \mid \filt_{k-1}\right]
  \indic{\event_{k/2, k}'}
  \le C(1+\norm{\Delta_k^t}^2)\indic{\event_{k/2, k}'}.
\end{equation*}
Therefore
\begin{align*}
  \sum_{i = 1}^\infty \frac{\stepsize_i}{\sqrt{i}}
  \E\left[\norms{\projt{\rsgiter_i^t} g_i^t}^2 \indic{\event_{i/2, i}'}
    \right]
  & \le C \sum_{i = 1}^\infty \frac{\stepsize_i}{\sqrt{i}}
  \left(1 + \E\left[\norms{\Delta_i^t}^2 \indic{\event_{i/2,i}}\right]
  \right)
  < \infty,
\end{align*}
where in the last step, we used the
claim~\eqref{eqn:expectation-go-to-zero}
and that $\stepsize_i \to 0$.
Again, we have $\sum_{i = 1}^\infty \frac{\stepsize_i}{\sqrt{i}}
\norms{\projt{\rsgiter_i^t} g_i^t}^2 \indics{\event_{i/2,i}'}
< \infty$ with probability 1, and identical reasoning as above
yields the second claim of the lemma.

Finally, we return to we prove the deferred
claim~\eqref{eqn:expectation-go-to-zero}. We essentially mimic the proof
of Lemma~\ref{lemma:fake-error-cas} in
Section~\ref{sec:proof-fake-error-cas}. As $\event_{i,k+1} \subset
\event_{k+1}'$ and $\event_{i,k+1} \subset \event_{i,k}'$, a derivation
identical, \emph{mutatis mutandis}, to that to derive
inequality~\eqref{eqn:crucial-estimate-rate-convergence-simplify}
yields that for all large enough $k$ and $i \le k$, we have
\begin{align*}
  \E \left[\norm{\Delta_{k+1}^t}^2 \indic{\event_{i, k+1}} \mid
    \filt_{k-1} \right]
  & \le (1- c\stepsize_k + C\stepsize_k^2) \norm{\Delta_k^t}^2
  \indic{\event_{i, k}} + C\stepsize_k^2.
\end{align*}
If $k$ is large enough that $1 - c\stepsize_k + C\stepsize_k^2 \le 1-
c\stepsize_k/2 \le \exp(-c\stepsize_k/2)$, this inequality implies
\begin{equation*}
  \E \left[\norm{\Delta_{k+1}^t}^2 \indic{\event_{i, k+1}} \right] \le \exp(-c\stepsize_k)
  \E \left[\norm{\Delta_k^t}^2 \indic{\event_{i, k}} \right] + C\stepsize_k^2.
\end{equation*}
Recursively applying this
inequality
yields that there exists $0 < c, C < \infty$ such that for all $k$
\begin{equation*}
  \E \left[\norm{\Delta_{k}^t}^2 \indic{\event_{i, k}} \right] 
  \le C \exp\bigg(-c \sum_{j=i}^{k-1} \stepsize_j\bigg) + C
  \sum_{l = i}^{k-1} \stepsize_l^2 
  \exp\bigg(-c\sum_{j = i}^{k-2} \stepsize_j\bigg).
\end{equation*}
As $\stepsize_k \propto k^{-\steppow}$, we have
$\sum_{j = i}^k \stepsize_j \gtrsim k^{1 - \steppow} - i^{1 - \steppow}$,
so we obtain the estimate
\begin{align*}
  \E \left[\norm{\Delta_{k}^t}^2 \indic{\event_{k/2, k}} \right]
  & \le C \exp\left(-\sum_{j=k/2}^{k-1} \stepsize_j\right) 
  + C \sum_{l=k/2}^{k-1} \stepsize_l^2 \exp\left(-c\sum_{j=l}^{k-2}
  \stepsize_j\right) \\
  & \le C \exp\left(-c k^{1 - \steppow}\right)
  + C \sum_{i = k/2}^k i^{-2 \steppow} \exp\left(-c(k^{1 - \steppow}
  - i^{1 - \steppow})\right).
\end{align*}
Applying Lemma~\ref{lemma:funny-zero-sum}
then gives the claim~\eqref{eqn:expectation-go-to-zero}.

\section{Proofs of preliminary
  calculus results (Section~\ref{sec:preliminary})}
\label{sec:proof-Calculus-rsgd}

\subsection{Proof of Lemma~\ref{lemma:stability-of-projection}}

It is equivalent to show that the mapping $x \mapsto \projt{x} = I_n -
\grad F(x) \grad F(x)^{\dag}$ is Lipschitz near $x\opt$. To do so,
we first note that $\grad F(x)$ is differentiable and is thus locally
Lipschitz. Next, we show that $\grad F(x)^{\dag}$ is also
locally Lipschitz: classical matrix perturbation theory results~\cite[Theorem
  3.8]{StewartSu90} show that for a universal constant $C < \infty$,
\begin{align}
  \label{eqn:pseudo-inverse-Lipschitz}
  & \opnorms{\grad F(x)^{\dag} - \grad F(x\opt)^{\dag}} \\
  \nonumber
  & ~~~ \le C \max 
  \left\{\opnorms{\grad F(x)^\dag}, \opnorms{\grad F(x\opt)^{\dag}}\right\}^2
  \opnorm{\grad F(x) - \grad F(x\opt)}.
\end{align}
As $\grad F(x\opt)$ is full column rank by
Assumption~\ref{assumption:linear-independence-cq}, it has
positive minimum singular value, so
$\inf_{x \in \ball(x\opt, \epsilon)} \sigma_{\numactive}(\nabla F(x))
> 0$ for some $\epsilon > 0$. Consequently,
\begin{equation*}
  \sup_{x \in \ball(x\opt, \eps)} \opnorms{\grad F(x)^{\dag}} < \infty.
\end{equation*}
This, coupled with Eq.~\eqref{eqn:pseudo-inverse-Lipschitz}, shows that the
mapping $x \mapsto \grad F(x)^{\dag}$ is Lipschitzian near $x\opt$, and
the lemma follows.

\subsection{Proof of Lemma~\ref{lemma:manifold-second-order-result}}

Note that $F(x) = F(x\opt) = 0$ for any $x \in \M$. Thus, for some $0 < \eps, C
< \infty$, if $x \in \M \cap \ball(x\opt, \epsilon)$ we have
\begin{equation*}
  \norm{\grad F(x\opt)^T(x-x\opt)}
  =  \norm{F(x) - F(x\opt) - \grad F(x\opt)^T(x-x\opt)}
  \le C\norm{x-x\opt}^2.
\end{equation*}
Thus, using the definition of $\projmatr$,
\begin{equation*}
  \norm{\projmatr(x-x\opt) - (x - x\opt)} =
  \normBig{(\grad F(x\opt)^T)^{\dag} \grad F(x\opt)^T(x-x\opt)}
  \le C \norm{x-x\opt}^2
\end{equation*}
for $x \in \M$ near $x\opt$, as $\opnorm{\nabla F(x\opt)^\dag} < \infty$
by Assumption~\ref{assumption:linear-independence-cq}.

\subsection{Proof of Lemma~\ref{lemma:second-order-projection-gradient}}

We start by showing the following
\begin{lemma}
  \label{lemma:manifold-hessian}
  We have the identity
  \begin{equation*}
    \projt{x}
    \cdot
    \grad \left.\left(\projt{x} \grad f(x)\right)\right|_{x = x\opt}
    = \projt{x\opt}
    \left[\grad^2 f(x\opt) + \sum_{i=1}^{\numactive} \lambda_i\opt \grad^2 f_i
      (x\opt)\right].
  \end{equation*}
\end{lemma}
\begin{proof}
  The first order conditions for optimality of $x\opt$
  guarantee that $\projt{x\opt} \grad f(x\opt) =
  0$. The standard Leibniz rule yields that
  \begin{equation*}
    \nabla \left(\projt{x} \grad f(x)\right)
    = \nabla^2 f(x) - \nabla \left(\nabla F(x) \nabla F(x)^\dag
    \nabla f(x)\right).
  \end{equation*}
  Now, let $a(x) = \nabla F(x)^\dag \nabla f(x)$, $a : \R^n \to \R^{\numactive}$.
  The Leibniz rule implies that
  $\nabla (\nabla F(x) a(x)) =
  \sum_{i = 1}^n a_i(x) \nabla^2 f_i(x)
  + \sum_{i = 1}^{\numactive} \nabla f_i(x) \nabla a_i(x)^T$.
  Standard properties of pseudo-inverses, and the fact that
  $\nabla F(x)$ is full column rank near $x\opt$, imply that
  $\projt{x} \nabla f_i(x) = 0$ for all $i$. Finally,
  we use that $\grad f(x\opt) = -\nabla F(x\opt) \lambda\opt$,
  which gives $a(x\opt) = -\lambda\opt$, and so
  \begin{align*}
    \projt{x\opt}
    \grad \left(\grad F(x\opt) \grad F(x\opt)^\dag \nabla f(x\opt)\right)
    = - \projt{x\opt} \sum_{i=1}^\numactive \lambda_i\opt
    \nabla^2 f_i(x\opt)
  \end{align*}
  Lemma~\ref{lemma:manifold-hessian} follows.
\end{proof}

The proof of Lemma~\ref{lemma:second-order-projection-gradient} is now
nearly immediate.  By the first-order optimality conditions for $x\opt$, we
know that $\projt{x\opt} \grad f(x\opt) = 0$, so
Lemma~\ref{lemma:manifold-hessian} and the fact that $f$ is $\mathcal{C}^2$
imply
\begin{align*}
  \projt{x} \grad f(x)
  & = (\grad (\projt{x\opt} \grad f(x\opt))
  (x - x\opt) + O(\norm{x-x\opt}^2) \\
  & = \projmatr H\opt (x - x\opt)  +  O(\norm{x-x\opt}^2). 
\end{align*}
By Lemma~\ref{lemma:manifold-second-order-result}, we have $x-x\opt =
\projmatr (x-x\opt) + O(\norm{x-x\opt}^2)$ for $x \in \M$ near $x\opt$.
Substituting this estimate above yields the lemma.


\subsection{Proof of Lemma~\ref{lemma:restricted-strong-convexity-near-rsgd}}
This follows by the following computations:
\begin{align*}
  \< x - x\opt, \projt{x}\grad f(x) \> &= \< \projt{x}(x - x\opt), \projt{x}\grad f(x) \>\\
  &\stackrel{(i)}{=} 
  \< \projmatr(x - x\opt), \projmatr H\opt
  \projmatr(x-x\opt) \> + O(\norm{x-x\opt}^3) \\
  &\stackrel{(ii)}{\ge} \mu \norm{\projmatr(x-x\opt)}^2 + O(\norm{x-x\opt}^3) \\
  &\stackrel{(iii)}{\ge} \mu \norm{x-x\opt}^2 + O(\norm{x-x\opt}^3).
\end{align*}
Here, in~$(i)$, we use Lemma~\ref{lemma:stability-of-projection} and Lemma
\ref{lemma:second-order-projection-gradient}; in~$(ii)$, we use
Assumption~\ref{assumption:restricted-strong-convexity}; and in~$(iii)$, we
use Lemma~\ref{lemma:manifold-second-order-result}.

\subsection{Proof of Lemma~\ref{lemma:second-order-control-of-projection}}

This result is very similar to the typical implicit function
theorem results in optimization~\cite{DontchevRo14}, though we could
not find a precise statement with the appropriate differentiability
guarantees, so we provide a proof.
Let $x_v = \proj_{\M}(x+v)$. First, by definition of projection, we have, 
\begin{equation*}
\norm{x_v - (x+v)} \le \norm{x-(x+v)} \le \norm{v}.
\end{equation*}
The triangle inequality implies that
\begin{equation}
\label{eqn:very-loose-bound-projection}
\norm{\proj_{\M}(x+v) - \proj_{\M}(x)} = \norm{x_v - x} 
	\le \norm{x_v - (x+v)} + \norm{v} \le 2\norm{v},
\end{equation}
giving that $\proj_{\M}(x+\cdot)$ is continuous at $0$ in $\T_{\M}(x)$.
Now, denote $A(x_1, x_2) \in \R^{\numactive \times \numactive}$ be the inner
product matrix $A_{i, j}(x_1, x_2) = \<\grad f_i(x_1), \grad
f_j(x_2)\>$. Note that $A(x\opt, x\opt)$ is full rank
Assumption~\ref{assumption:linear-independence-cq}. Since the mapping
$F(\cdot)$ is $\mc{C}^2$ around $x\opt$, we have $\sup_{x_1, x_2 \in
  \ball(x\opt, 3\eps)}\opnorm{A(x_1, x_2)^{-1}} \le C < \infty$ for some
constants $0 < \eps, C < \infty$.

Using the KKT conditions for $x_v$ to be the projection of $x+v$ onto $\M$
(which hold for small $v$ by our constraint
qualifications~\cite{DontchevRo14}), we have for some $\lambda(x, v) \in
\R^{\numactive}$ that
\begin{equation}
  \label{eqn:KKT-condition}
  x_v = x+ v + \sum_{i \in [\numactive]} \lambda_i(x, v) \grad f_i(x_v)
  = x + v + \nabla F(x_v) \lambda(x, v),
\end{equation}
so $\< v, \grad f_j(x)\> = 0$ for $j \in [\numactive]$, as 
$v \in \T_{\M}(x)$. Multiplying both sides of Eq.~\eqref{eqn:KKT-condition} by 
$\grad F(x) \in \R^{n \times m_0}$, we obtain
\begin{equation*}
  \grad F(x)^T (x_v - x) = \grad F(x)^T \grad F(x_v) \lambda(x, v)
  = A(x, x_v) \lambda(x, v).
\end{equation*}
By Eq.~\eqref{eqn:very-loose-bound-projection}, $\norm{x_v - x} \le 2
\norm{v}$ and so $A(x, x_v)$ is invertible for small enough $v$,
so that for small $\norm{v}$,
\begin{equation*}
  \lambda(x, v) = A(x, x_v)^{-1} \grad F(x)^T(x_v - x)
  = O(\norm{v}^2),
\end{equation*}
where we have used that $\grad F(x)^T (x_v - x)
= F(x) - F(x_v) + O(\norm{x_v - x}^2) = O(\norm{v}^2)$, because
$F(x) = F(x_v) = 0$.

Using the KKT conditions~\eqref{eqn:KKT-condition}, we have $\norm{x_v - (x
  + v)} \le \opnorm{\nabla F(x_v)} \norm{\lambda(x, v)} = O(\norm{v}^2)$,
as $\nabla F$ is bounded near $x\opt$.

\section{Proof of
  Proposition~\ref{proposition:generalization-of-Polyak-Juditsky}}
\label{sec:proof-generalization-polyak-juditsky}

With slight abuse of notation, we write
$\noise_k = \noise_k(x_k)$ throughout the proof for simplicity. Define
the product matrices and associated weighted sums
\begin{equation}
  \label{eqn:product-matrices}
  B_i^k = \prod_{j=i}^k (I - \stepsize_j \projmatr H \projmatr)
  ~~ \mbox{where} ~~
  B_i^{i-1} = I_{n \times n},
  ~~ \mbox{and} ~~
  \wb{B}_i^k = \stepsize_i \sum_{l=i}^{k-1} B_{i+1}^{l}.
\end{equation}
In this case, by expanding the recursion for $\Delta_k$ and
letting $\wb{\Delta}_k = \frac{1}{k} \sum_{i = 1}^k \Delta_i$,
we immediately
obtain~\cite{PolyakJu92}
\begin{equation}
  k \wb{\Delta}_k  = \sum_{j=1}^k \Delta_j
  =
  \sum_{i=1}^k B_1^{i-1} \Delta_1
  - \sum_{i=1}^k \wb{B}_i^k \projmatr \xi_i
  - \sum_{i=1}^k \wb{B}_i^k \projmatr \notquaderror_i
  + \sum_{i=1}^k \bigg(\sum_{l=i+1}^{k} B_{i+1}^{l-1} \bigg) \offerror_i.
  \label{eqn:expansion-summed-errors}
\end{equation}
We study each of the terms in the 
expansion~\eqref{eqn:expansion-summed-errors} in turn, showing 
that the second term is asymptotically normal when normalized by 
$k^{-\half}$ and the rest of the three terms are of order $o(\sqrt{k})$.

We now collect a few lemmas, most modifications
of results due to \citet{PolyakJu92},
that allow us to analyze the sequences~\eqref{eqn:expansion-summed-errors}.
For the first lemma, we define the seminorm
\begin{equation*}
  \normT{A}
  \defeq \sup \left\{\norm{A x} : x \in \T, \norm{x} \le 1 \right\},
\end{equation*}
where we take $\normT{A} = 0$ if $\T = \{0\}$. Clearly $\normT{A} \le
\opnorm{A}$, where $\opnorm{A} \defeq \sup \{\norm{A x} : \norm{x} \le 1\}$
is the $\ell_2$-operator norm.  With this definition, we see
that the matrices $B_i^k$ shrink quickly to zero.
\begin{lemma}
  \label{lemma:upper-bound-on-B-prods}
  For any nonnegative sequence $\{\stepsize_k\}$ satisfying
  $\alpha_k \rightarrow 0$, there exists $\lambda > 0$ 
  and $M < \infty$ such that, for any $j \in \N$ and $k \geq j$,
  \begin{equation*}
    \normTs{B_j^k} \leq M
    \exp \bigg(-\lambda \sum_{i=j}^{k} \stepsize_i \bigg).
  \end{equation*}
\end{lemma}
\noindent
Except for trivial technicalities because of the restriction
to the subspace $\T$, this result is known~\cite[Lemma 1, part 3]{PolyakJu92}.

Our second lemma addresses the error terms $\{\offerror_k\}_{k \in \N}$, 
which we show are negligible. 
We have the following lemma, whose technical proof
we provide in Appendix~\ref{sec:proof-offerror-zeros}.
\begin{lemma}
  \label{lemma:offerror-zeros}
  Let $\stepsize_k \propto k^{-\steppow}$, where $\steppow \in (0, 1)$
  and assume $\offerror_k = 0$ for all $k \ge T$. Then
  there is a constant $C = C(\steppow) < \infty$ such that
  \begin{equation*}
    \normbigg{
      \sum_{i=1}^k \bigg(\sum_{l=i+1}^{k} B_{i+1}^{l-1} \bigg) \offerror_i
    }
    \le C T^{1 + \steppow} \max_{i \le T} \norm{\offerror_i}.
  \end{equation*}
\end{lemma}

Our third lemma addresses the error terms $\{\notquaderror_k\}_{k \in \N}$.
By noting that
$\sup_{i,k} \normTs{\wb{B}_i^k} < \infty$ (cf.~\cite[Lemma 2]{PolyakJu92}),
we have
\begin{lemma}
  \label{lemma:notquaderror-convergence}
  There is a constant $c < \infty$ such that, for all $k \in \N$
  \begin{equation*}
    \norm{\sum_{i=1}^k \wb{B}_i^k \projmatr \notquaderror_i}
    \le c \sum_{i=1}^k \norms{\projmatr \notquaderror_i}.
  \end{equation*}
\end{lemma}

Our fourth lemma gives the most important result for the theorem, which is
asymptotic normality of the noise sequence $\wt{\noise}_k$ on the subspace
defined by the projection matrices $\projmatr$.  By showing that the matrices
$\wb{B}_i^k$ approximate the projected pseudo-inverse $(\projmatr H
\projmatr)^\dag$, we have the following asymptotic normality result, whose
proof we defer to Appendix~\ref{sec:proof-asymptotic-normality-main-terms}.
\begin{lemma}
  \label{lemma:asymptotic-normality-of-main-terms}
  Let the conditions of
  Proposition~\ref{proposition:generalization-of-Polyak-Juditsky} hold. Then
  \begin{equation*}
    \frac{1}{\sqrt{k}} \sum_{i=1}^{k} \wb{B}_i^k \wt{\noise}_i 
    \cd \normal \left(0, (\projmatr H \projmatr)^\dag \projmatr 
    \Sigma \projmatr (\projmatr H \projmatr)^\dag\right).
  \end{equation*}
\end{lemma}
Returning to the equality~\eqref{eqn:expansion-summed-errors},
we have
\begin{align*}
  \frac{1}{\sqrt{k}} \sum_{i=1}^k \Delta_i
  & = 
  \frac{1}{\sqrt{k}} \sum_{i=1}^{k} B_1^{i-1} \Delta_1 - 
  \frac{1}{\sqrt{k}} \sum_{i=1}^{k} \wb{B}_i^k \wt{\xi}_i
  -\frac{1}{\sqrt{k}} \sum_{i=1}^{k} \wb{B}_i^k \wt{\notquaderror}_i + 
  \frac{1}{\sqrt{k}} \sum_{i=1}^k\sum_{l=i+1}^{k}
  B_{i+1}^{l-1}\offerror_i
\end{align*}
Taking norms, we have
$\norm{B_1^{i-1} \Delta_1} \le
\normTs{B_1^{i-1}} \norm{\Delta_1}$,
and Lemma~\ref{lemma:upper-bound-on-B-prods}
implies $k^{-\half} \sum_{i = 1}^k \normTs{B_1^{i-1}} \to 0$.
By Lemma~\ref{lemma:offerror-zeros}, defining
$T < \infty$ to be the (random but finite) time such that
$\offerror_k = 0$ for $k \ge T$, we have
\begin{equation*}
  \frac{1}{\sqrt{k}} \normbigg{
    \sum_{i=1}^k\sum_{l=i+1}^{k}
    B_{i+1}^{l-1}\offerror_i}
  \le c \frac{1}{\sqrt{k}} T^{1 + \steppow} \max_{i \le T}
  \norm{\offerror_i}
  \to 0
\end{equation*}
as $k \to \infty$. By Lemma~\ref{lemma:notquaderror-convergence}
and Assumption~\ref{assumption:convergence-as}, we have
\begin{equation*}
  \frac{1}{\sqrt{k}}\norm{\sum_{i=1}^k \wb{B}_i^k \projmatr \notquaderror_i} 
  \le \frac{c}{\sqrt{k}} \sum_{i=1}^k \norms{\projmatr \notquaderror_i}
  \cas 0.
\end{equation*}
Applying Lemma~\ref{lemma:asymptotic-normality-of-main-terms} and 
Slutsky's theorem
to the sum thus gives
\begin{equation*}
  \frac{1}{\sqrt{k}} \sum_{i=1}^k \Delta_i
  = -\frac{1}{\sqrt{k}} \sum_{i = 1}^k \wb{B}_i^k \projmatr \noise_i
  + o_P(1)
  \cd 
  \normal\left(0, (\projmatr H \projmatr)^\dag \projmatr 
  \Sigma \projmatr (\projmatr H \projmatr)^\dag\right),
\end{equation*}
which is the desired claim of the proposition.

\section{Proofs of technical lemmas for
  Proposition~\ref{proposition:generalization-of-Polyak-Juditsky}}
\label{sec:proof-generalization-of-Polyak-Judistky-technical-lemma}

In this section, we collect the proofs of the various technical lemmas
required for the proof of
Proposition~\ref{proposition:generalization-of-Polyak-Juditsky}.
Before proving the lemmas, we state one technical result, similar
to~\cite[Lemma 13]{DuchiChRe15}, that is useful for what follows.  (We
include a proof for completeness in
Section~\ref{sec:proof-funny-gamma-integral}).
\begin{lemma}
  \label{lemma:funny-gamma-integral}
  Let $c > 0$ and $\kappa, \rho \in (0, 1)$ be constants and $b \ge a >
  0$. Then
  \begin{equation*}
    \int_a^b (t^\rho - a^\rho)
    \exp\left(-c (t^\kappa - a^\kappa)\right) dt
    \le O_{c,\kappa,\rho}(1)
    \left[\Gamma\Big(\frac{1 + \rho}{\kappa}\Big)
      + a^{1 + \rho - \kappa}\right],
  \end{equation*}
  where $O_{c,\kappa,\rho}(1)$ denotes a multiplicative constant
  dependent only on $c, \kappa, \rho$. If
  $\kappa = \rho$, then we have moreover that
  \begin{equation*}
    \lim_{a \to \infty} \frac{1}{a} \int_a^\infty (t^\rho - a^\rho)
    \exp\left(-c (t^\kappa - a^\kappa)\right) dt
    = 0.
  \end{equation*}
\end{lemma}

\subsection{Proof of Lemma~\ref{lemma:offerror-zeros}}
\label{sec:proof-offerror-zeros}

By assumption, $\offerror_k = 0$ when $k \ge T$.
Then we have
\begin{equation*}
  \normbigg{\sum_{i=1}^k
    \bigg(\sum_{l=i+1}^{k} B_{i+1}^{l-1} \bigg) \offerror_i}
  \le \sum_{i = 1}^T \sum_{l = i + 1}^k \normT{B_{i+1}^{l - 1}}
  \norm{\offerror_i}
  \le M \sum_{i = 1}^T
  \sum_{l = i + 1}^k \exp\bigg(-\lambda \sum_{j = i + 1}^l
  \stepsize_j\bigg) \norm{\offerror_i},
\end{equation*}
where $M < \infty$ and $\lambda > 0$ are defined as in
Lemma~\ref{lemma:upper-bound-on-B-prods}.  But because $\stepsize_j
\propto j^{-\steppow}$ for some $\steppow < 1$, we have for constants $c,
c' \in (0, \infty)$ that may change from line to line that
\begin{align*}
  \sum_{i = 1}^T
  \sum_{l = i}^k
  \exp\left(-\lambda \sum_{j = i}^l \stepsize_j\right)
  & \le \sum_{i = 1}^T \sum_{l = i}^k \exp\left(-c (l^{1 - \steppow}
  - i^{1 - \steppow})\right) \\
  & \le c' \sum_{i = 1}^T
  \int_i^k \exp\left(-c (t^{1 - \steppow} - i^{1 - \steppow})\right)
  dt \\
  & \le c' \sum_{i = 1}^T (1 + i^\steppow),
\end{align*}
where the final inequality is a consequence of
Lemma~\ref{lemma:funny-gamma-integral}. Using $\sum_{i = 1}^T i^\steppow \le
c T^{1 + \steppow}$ gives the final result.


\subsection{Proof of Lemma~\ref{lemma:asymptotic-normality-of-main-terms}}
\label{sec:proof-asymptotic-normality-main-terms}

\newcommand{\errmat}{E}

For shorthand, define the matrix $\wt{H} \defeq \projmatr H \projmatr$.
The first step in our proof is to argue that the averaged matrices
$\wb{B}_j^k$ approximate the pseudo-inverse $\wt{H}^\dag$.
Our argument is similar to that of
\citeauthor{PolyakJu92} in the case that $H$ is invertible and the problem is
unconstrained~\cite[Lemma 1]{PolyakJu92}, where one obtains $\wb{B}_j^k \to
H^{-1}$ in an appropriate sense. To that end, we state two technical lemmas.

\begin{lemma}
  \label{lemma:invertible}
  Let Assumption~\ref{assumption:growth-and-noise} hold.
  Then for any $x\in \T$,
  \begin{equation*}
    \wt{H}^\dag \wt{H} x = x ~~\text{and}~~\wt{H} \wt{H}^\dag x = x.
  \end{equation*}
\end{lemma}
\noindent
This is immediate from the definition of the pseudoinverse and that
$\projmatr$ is an orthogonal projector.  The next lemma,
paralleling \citet[Lemma 1]{PolyakJu92}, shows how $\wb{B}_j^k$ approximates
the pseudo-inverse $\wt{H}^\dag$. We define the error
sequence
\begin{equation*}
  \errmat_j^k \defeq \wt{H}^\dag - \wb{B}_j^k
\end{equation*}
and recall the norm $\normT{A} = \sup\{ \norm{Ax} : x\in \T, \norm{x} \le 1\}$,
which results in the following
\begin{lemma}[Polyak and Juditsky~\cite{PolyakJu92}, Lemmas 1 and 2]
  \label{lemma:approximate-error-matrices}
  Suppose $\stepsize_k \propto k^{-\steppow}$, where $\steppow \in (\half,
  1)$. Then there exists $M < \infty$ such that
  \begin{equation*}
    \sup_{j, k}\normTs{\errmat_j^k} \leq M 
    ~~ \mbox{and} ~~
    \lim_{k\rightarrow \infty} \frac{1}{k} \sum_{j=0}^{k-1} \normTs{\errmat_j^k}= 0.
  \end{equation*}
\end{lemma}
\noindent
Except for trivialities to deal with the subspace $\T$, this
result is contained in \cite[Lemmas 1 \& 2]{PolyakJu92}.

Finally, we can show the desired asymptotic normality. Similar to $H$,
for any vector $v$ define the shorthand  $\wt{v} \defeq \projmatr v$.
By
algebraic manipulations, we have
\begin{align}
  \frac{1}{\sqrt{k}} \sum_{i=1}^k \wb{B}_i^k \projmatr \noise_i
  & = \wt{H}^\dag \frac{1}{\sqrt{k}} \sum_{i=1}^k \projmatr \noisezero_i
  + \wt{H}^\dag \frac{1}{\sqrt{k}} \sum_{i=1}^k \projmatr \noisier_i(x_i)
  - \frac{1}{\sqrt{k}} \sum_{i=1}^k \errmat_i^k \wt{\xi}_i,
  \label{eqn:decompose-weighted-noise-sequence}
\end{align}
where we have used the decomposition $\noise_i = \noisezero_i +
\noisier_i(x_i)$ that Assumption~\ref{assumption:growth-and-noise} defines. We
control each of the terms in the
expansion~\eqref{eqn:decompose-weighted-noise-sequence}, starting with the
last two, which converge to zero.
\begin{lemma}
  \label{lemma:kill-error-noises}
  Let the conditions of
  Proposition~\ref{proposition:generalization-of-Polyak-Juditsky} hold. Then
  \begin{equation*}
    \frac{1}{\sqrt{k}} \sum_{i = 1}^k \errmat_i^k \wt{\noise}_i \cp 0.
  \end{equation*}
\end{lemma}
\begin{proof}
  We break the sum into two terms, depending on whether
  $\norm{x_i - x\opt} > \epsilon$ or $\norm{x_i - x\opt} \le \epsilon$.  In the
  former case, Assumption~\ref{assumption:convergence-as}
  guarantees that $\norm{x_i - x\opt} > \epsilon$ occurs only finitely often.
  In the second case, we note that the event that $\norm{x_i - x\opt} \le
  \epsilon$ belongs to the $\sigma$-field $\mc{F}_{i-1}$. Thus,
  by expanding a square, we have
  \begin{align*}
    \lefteqn{\E\left[\norm{\frac{1}{\sqrt{k}} \sum_{i=1}^k \errmat_i^k \wt{\xi}_i
          \indics{\norm{x_i - x\opt} \le \epsilon}}^2\right]
      = \frac{1}{k} \sum_{i,j}^k
      \E\left[\<\errmat_i^k \wt{\noise}_i, \errmat_j^k \wt{\noise}_j\>
        \indics{\norm{x_i - x\opt} \le \epsilon}
        \indics{\norm{x_j - x\opt} \le \epsilon} \right]} \\
    & \qquad\qquad\qquad\qquad\qquad\qquad
    = \frac{2}{k} \sum_{i < j}
    \E\Big[\underbrace{\E\left[\<\errmat_i^k \wt{\noise}_i, \errmat_j^k \wt{\noise}_j\>
        \mid \mc{F}_{j-1}\right]}_{= 0}
      \indics{\norm{x_i - x\opt} \le \epsilon,
        \norm{x_j - x\opt} \le \epsilon}\Big] \\
    & \qquad\qquad\qquad\qquad\qquad\qquad\qquad
    ~ + \frac{1}{k} \sum_{i = 1}^k
    \E\left[\norm{\errmat_i^k \wt{\noise}_i}^2
      \indics{\norm{x_i - x\opt} \le \epsilon}\right].
  \end{align*}
  Noting that $\norms{\errmat_i^k \wt{\noise}_i} \le
  \normTs{\errmat_i^k} \norms{\wt{\noise}_i}$, since $\wt{\noise}_i \in \T$ by 
  construction of $\wt{\noise}$. Thus we have
  \begin{equation*}
    \E\left[\norm{\frac{1}{\sqrt{k}} \sum_{i=1}^k \errmat_i^k \wt{\xi}_i
        \indics{\norm{x_i - x\opt} \le \epsilon}}^2\right]
    \le 
    \frac{1}{k} \sum_{i = 1}^k
    \normTs{\errmat_i^k}^2
    \E\left[\norm{\wt{\noise}_i}^2
    \indics{\norm{x_i - x\opt} \le \epsilon}\right].
  \end{equation*}
  But of course, Assumption~\ref{assumption:growth-and-noise} guarantees that
  \begin{equation*}
    \E[\norms{\wt{\noise}_i}^2 \indics{\norm{x_i - x\opt} \le \epsilon}] \le c
    \E\left[(1 + \norm{x_i - x\opt}^2) \indic{\norm{x_i - x\opt} \le \epsilon}
      \right] \le 2c
  \end{equation*}
  for some constant $c < \infty$, and
  using Lemma~\ref{lemma:approximate-error-matrices}, we thus obtain
  \begin{equation*}
    \E\left[\norm{\frac{1}{\sqrt{k}} \sum_{i=1}^k \errmat_i^k \wt{\xi}_i
        \indics{\norm{x_i - x\opt} \le \epsilon}}^2\right]
    \le 
    \frac{c}{k} \sum_{i = 1}^k
    \normTs{\errmat_i^k}^2
    \le \frac{M c}{k} \sum_{i = 1}^k \normTs{\errmat_i^k}
    \to 0
  \end{equation*}
  as $k \to \infty$. Thus
  \begin{equation*}
    \frac{1}{\sqrt{k}} \sum_{i = 1}^k \errmat_i^k \wt{\noise}_i
    = \underbrace{\frac{1}{\sqrt{k}} \sum_{i = 1}^k \errmat_i^k \wt{\noise}_i
    \indic{\norm{x_i - x\opt} > \epsilon}}_{\cas 0}
    + \underbrace{\frac{1}{\sqrt{k}} \sum_{i = 1}^k \errmat_i^k \wt{\noise}_i
      \indic{\norm{x_i - x\opt} \le \epsilon}}_{\cp 0}
    \cp 0,
  \end{equation*}
  as desired.
\end{proof}
\begin{lemma}
  \label{lemma:convergence-of-weird-terms}
  Let the conditions of
  Proposition~\ref{proposition:generalization-of-Polyak-Juditsky} hold.
  Then
  \begin{equation*}
    \frac{1}{\sqrt{k}} \sum_{i=1}^k \noisier_i(x_i)
    \cas 0.
  \end{equation*}
\end{lemma}
\begin{proof}
  We again break the sum into two terms, depending on whether $\norm{x_i -
    x\opt} > \epsilon$ or $\norm{x_i - x\opt} \le \epsilon$.  In the former
  case, Assumption~\ref{assumption:convergence-as} guarantees
  that $\norm{x_i - x\opt} > \epsilon$ occurs only finitely often.  In the
  second case, we first show that,
  \begin{equation}
    \label{eqn:almost-sure-bound-harmonic-average}
    \sum_{i=1}^{\infty} \frac{1}{i} \norm{x_i - x\opt}^2 < \infty.
  \end{equation}
  Indeed, for $k \in \N$, we have the identity
  \begin{equation*}
    \sum_{i=1}^k \frac{1}{i} \norm{x_i - x\opt}^2 = \frac{1}{k}
    \sum_{i=1}^k \norm{x_i - x\opt}^2 + \sum_{i=1}^{k-1}\frac{1}{i(i+1)}
    \Big(\sum_{j=1}^i \norm{x_j - x\opt}^2\Big). 
  \end{equation*}
  By Assumption~\ref{assumption:convergence-as}, there exists a 
  random variable $M$ with $M < \infty$ such that for all $k \in \N$
  $k^{-1/2}\sum_{j=1}^k \norm{x_j - x\opt}^2 \le M$.
  Substituting this estimate into the preceding display,
  we immediately obtain
  \begin{equation*}
    \sum_{i=1}^{\infty} \frac{1}{i} \norm{x_i - x\opt}^2
    = \lim_{k \to \infty} \sum_{i=1}^k \frac{1}{i} \norm{x_i - x\opt}^2   
    \le M\sum_{i=1}^\infty \frac{1}{i^{3/2}} < \infty.
  \end{equation*}
  This proves Eq.~\eqref{eqn:almost-sure-bound-harmonic-average}. Now, 
  using Assumption~\ref{assumption:growth-and-noise}, 
  Eq.~\eqref{eqn:almost-sure-bound-harmonic-average} immediately implies that
  \begin{equation*}
    \sum_{i=1}^\infty \frac{1}{i}\E\left[\norms{\noisier_i(x_i)}^2
      \mid \mc{F}_{i-1}\right]
    \indic{\norm{x_i - x\opt}\le \eps} \le 
    \sum_{i=1}^\infty \frac{1}{i} \norm{x_i - x\opt}^2 < \infty. 
  \end{equation*}
  Applying Lemma~\ref{lemma:dembo-l2}, gives that
  \begin{equation*}
    \frac{1}{\sqrt{k}}\sum_{i=1}^k \noisier_i(x_i)\indic{\norm{x_i - x\opt}
      \le \eps} \cas 0.
  \end{equation*}
  As $\norm{x_i - x\opt} \le \epsilon$ except for finitely many $i$,
  this gives the lemma.
\end{proof}

With Lemmas~\ref{lemma:kill-error-noises}
and~\ref{lemma:convergence-of-weird-terms} in hand, we return to the
expansion~\eqref{eqn:decompose-weighted-noise-sequence}.  By
Lemma~\ref{lemma:kill-error-noises}, the final sum converges in probability to
zero, while Lemma~\ref{lemma:convergence-of-weird-terms} shows that the second
to last term converges almost surely to zero. The first term on the right side
of expression~\eqref{eqn:decompose-weighted-noise-sequence}, on the other
hand, is asymptotically normal: Assumption~\ref{assumption:growth-and-noise}
guarantees that $k^{-\half} \sum_{i = 1}^k \projmatr \noisezero_i \cd \normal(0,
\projmatr \Sigma \projmatr)$.  Slutsky's theorem thus implies
Lemma~\ref{lemma:asymptotic-normality-of-main-terms}
as desired.

\subsection{Proof of Lemma~\ref{lemma:funny-gamma-integral}}
\label{sec:proof-funny-gamma-integral}

We prove the result via a change of variables. Let
$u = c (t^\kappa - a^\kappa)$, so that
\begin{equation*}
  t = \left(u/c + a^\kappa\right)^{\frac{1}{\kappa}},
  ~~~
  du = \kappa c t^{\kappa - 1} dt
  = \kappa c \left(u/c + a^\kappa\right)^{\frac{\kappa - 1}{\kappa}}
  dt,
  ~~~ \mbox{or} ~~~
  dt = (\kappa c)^{-1}
  \left(u / c + a^\kappa\right)^{\frac{1 - \kappa}{\kappa}}
  du.
\end{equation*}
That is, by our change of variables, we have
\begin{align*}
  \int_a^b (t^\rho - a^\rho)
  \exp\left(-c (t^\kappa - a^\kappa)\right) dt
  & = \frac{1}{\kappa c}\int_0^{c(b^\kappa - a^\kappa)}
  \left(\left(\frac{u}{c} + a^\kappa\right)^\frac{\rho}{\kappa}
  - a^\rho\right)
  \left(\frac{u}{c} + a^\kappa\right)^\frac{1 - \kappa}{\kappa}
  e^{-u}
  du \\
  & \le \frac{1}{\kappa c}
  \int_0^{c(b^\kappa - a^\kappa)}
  \left[
    \left(\frac{u}{c} + a^\kappa\right)^\frac{\rho + 1 - \kappa}{\kappa}
    - a^{\rho + 1 - \kappa}\right] e^{-u} du
\end{align*}
where the last inequality follows because $(u/c + a^\kappa)^\frac{1 -
  \kappa}{\kappa} \ge a^\kappa$. Now, we note that
this final quantity is upper bounded by
\begin{align*}
  O_{\kappa,c,\rho}(1)
  \left[
    \int_0^\infty
    u^\frac{1 + \rho - \kappa}{\kappa} e^{-u} du
    +
    \int_0^\infty
    a^{1 + \rho - \kappa} e^{-u} du\right]
\end{align*}
by convexity of $t \mapsto t^\frac{1 -
  \kappa}{\kappa}$, for $\kappa < \half$ and
the fact that $(t_1 + t_2)^\frac{1 - \kappa}{\kappa}
\le t_1^\frac{1 - \kappa}{\kappa} + t_2^\frac{1 - \kappa}{\kappa}$
for $\kappa \ge \half$ (or $\frac{1 - \kappa}{\kappa} \le 1$).
Noting that $\int_0^\infty u^{\alpha - 1} e^u du = \Gamma(\alpha)$ by
definition,
we obtain our first result.

For the second, we
let $b = \frac{1}{\kappa}$, so that
we consider the integral
\begin{equation*}
  \int_0^\infty \left[(cu + a^{1/b})^b - a \right] e^{-u} du
\end{equation*}
where $\kappa b = 1$.
Dividing the integral by $a \ge 1$, we obtain
\begin{equation*}
  \int_0^\infty \left[a^{-1} (cu + a^{1/b})^b - 1 \right] e^{-u} du
  = \int_0^\infty \left[(a^{-1/b} cu + 1)^b - 1 \right] e^{-u} du.
\end{equation*}
The first term in the integral is dominated by
$(cu + 1)^b e^{-u}$ for all $a \ge 1$, so that taking $a \to \infty$
allows us to apply the dominated convergence theorem as
$(a^{-1/b} cu + 1) \to 1$ as $a \to \infty$.

\end{document}